\RequirePackage{fix-cm}
\documentclass[smallextended,envcountsame,numbook]{svjour3}       % onecolumn (second format)
\smartqed  % flush right qed marks, e.g. at end of proof
\usepackage{graphicx}
\usepackage[normalem]{ulem}
\usepackage{soul}
\usepackage{amsmath}
  \usepackage{paralist}
  \usepackage{graphics} 
  \usepackage{epsfig}
\usepackage{graphicx}  \usepackage{epstopdf}
 \usepackage[colorlinks=true]{hyperref}
\hypersetup{urlcolor=blue, citecolor=red}

\usepackage{amsfonts,amssymb,stmaryrd,mathabx}
\usepackage{units}
\usepackage{subfigure}
\usepackage{times} 
\usepackage{hyperref}
\usepackage{amscd,empheq}
\usepackage{yfonts}
\usepackage{mathrsfs}
\usepackage[scr=rsfs,cal=boondox]{mathalfa}
\usepackage[cmtip,arrow]{xy}
\usepackage{pb-diagram,pb-xy}
\usepackage{titletoc}
\usepackage{palatino}
\usepackage{bm}
\usepackage{color}
\usepackage{mathtools}
\usepackage{old-arrows}
\usepackage{dsfont}
\usepackage{multirow}
\usepackage[margin=2cm]{geometry}
\usepackage{adjustbox}

\usepackage{tikz-cd}
\usepackage{multicol}
\usepackage{scalerel}
\usepackage[colorinlistoftodos]{todonotes}

\newcommand\restr[2]{{
  \left.\kern-\nulldelimiterspace 
  #1 
  \vphantom{\big|} 
  \right|_{#2}
  }}

\begin{document}
\begin{sloppypar}
%\setcounter{tocdepth}{3}

% THEOREMS
%\newtheorem{theorem}{Theorem}[section]
%\newtheorem{lemma}[theorem]{Lemma}
%\newtheorem{proposition}[theorem]{Proposition}
%\newtheorem{definition}[theorem]{Definition}
%\newtheorem{corollary}[theorem]{Corollary}
%\newtheorem{remark}[theorem]{Remark}
%\newtheorem{claim}[theorem]{Claim}
%\newtheorem{example}[theorem]{Example}
%\newtheorem{exercise}[theorem]{Exercise}

%\newtheorem{theorem}{Theorem}[equation]
%\newtheorem{lemma}[equation]{Lemma}
%\newtheorem{proposition}[equation]{Proposition}
%\newtheorem{definition}[equation]{Definition}
%\newtheorem{corollary}[equation]{Corollary}
%\newtheorem{remark}[equation]{Remark}
%\newtheorem{claim}[equation]{Claim}
%\newtheorem{example}[equation]{Example}
%\newtheorem{exercise}[equation]{Exercise}

% \newtheorem{hypothesis}{H}
% \newtheorem{property}[hypothesis]{P}

% COUNTERS
%\renewcommand{\thesection}{\thechapter.\arabic{section}}
%\renewcommand{\theequation}{\thesection.\arabic{theorem}}
%\renewcommand{\theequation}{\thesection.\arabic{equation}}
%\renewcommand{\thefigure}{\thechapter.\arabic{figure}}
%\newcommand\step{\stepcounter{theorem}}

% OLD fig ENVIRONMENT - can use figure as well
\newenvironment{fig}
{\begin{figure}[hbt]}%\stepcounter{theorem}}
{\end{figure}}
\newcommand{\bfig}{\begin{fig}}
\newcommand{\efig}{\end{fig}}

% SET-RELATED MACROS
\newcommand{\rc}{{\scriptscriptstyle\#}}
\newcommand{\set}[1]{\left| {#1}\right|}
\newcommand{\setof}[1]{\left\{ {#1}\right\}}
\newcommand{\mvmap}{\raisebox{-0.2ex}{$\,\overrightarrow{\to}\,$}}
\newcommand{\setdef}[2]{\left\{{#1}\,\left|\,{#2}\right.\right\}}
\newcommand{\bdf}{\hbox{\rm bd$_f$}}
\newcommand{\etale}{{\'etal\'e~}}
\newcommand{\Cech}{{{\v C}ech}~}

% BOLD LETTERS
\newcommand{\bv}{{\textbf{v}}}
\newcommand{\bz}{{\textbf{z}}}

% BLACKBOARD BOLD LETTERS
\newcommand{\D}{{\mathbb{D}}}
\newcommand{\F}{{\mathbb{F}}}
\newcommand{\N}{{\mathbb{N}}}
\newcommand{\Q}{{\mathbb{Q}}}
\newcommand{\R}{{\mathbb{R}}}
\newcommand{\T}{{\mathbb{T}}}
\newcommand{\Z}{{\mathbb{Z}}}
\newcommand{\Sb}{{\mathbb{S}}}

% HOMOTOPY SYMBOL -- CHANGE THIS
\newcommand{\sh}{{\mathsf h}}

% FRAK LETTERS
%\newcommand{\cF}{{\mathfrak{f}}} % FOR MULTIVALUED MAPS

%%%% Sheaves

% \newcommand{\gA}{{\textgoth A}}
% \newcommand{\gB}{{\textgoth B}}
% \newcommand{\gC}{{\textgoth C}}
% \newcommand{\gD}{{\textgoth D}}
% \newcommand{\gE}{{\textgoth E}}
% \newcommand{\gF}{{\textgoth F}}
% %\newcommand{\gG}{{\textgoth G}}
% \newcommand{\gG}{{\textgoth G}} % CHANGE THIS!!!!
% \newcommand{\gH}{{\textgoth H}}
% \newcommand{\gI}{{\textgoth I}}
% \newcommand{\gJ}{{\textgoth J}}
% \newcommand{\gK}{{\textgoth K}}
% \newcommand{\gL}{{\textgoth L}}
% \newcommand{\gM}{{\textgoth M}}
% \newcommand{\gN}{{\textgoth N}}
% \newcommand{\gO}{{\textgoth O}}
% \newcommand{\gP}{{\textgoth P}}
% \newcommand{\gQ}{{\textgoth Q}}
% \newcommand{\gR}{{\textgoth R}}
% \newcommand{\gS}{{\textgoth S}}
% \newcommand{\gT}{{\textgoth T}}
% \newcommand{\gU}{{\textgoth U}}
% \newcommand{\gV}{{\textgoth V}}
% \newcommand{\gW}{{\textgoth W}}
% \newcommand{\gX}{{\textgoth X}}
% \newcommand{\gY}{{\textgoth Y}}

% \newcommand{\fS}{{\mathfrak S}}

\newcommand{\gAtt}{{\mathcal{A}\!\mathcal{t}\!\!\mathcal{t}}}
\newcommand{\gMorse}{{\mathcal{M}\!\!\mathcal{o}\!\mathcal{r}\!\mathcal{s}\!\mathcal{e}}}
\newcommand{\gANbhd}{{\mathcal{A}\!\!\mathcal{N}\!\mathcal{b}\!\mathcal{h}\!\mathcal{d}}}

%%%%%%%

% CALIGRAPHIC LETTERS

\newcommand{\calP}{{\mathcal P}}
\newcommand{\calW}{{\mathcal W}}
\newcommand{\calB}{{\mathcal B}}

% BF letters
\newcommand{\cA}{{\mathbf A}}
\newcommand{\cB}{{\mathbf B}}
\newcommand{\cC}{{\mathbf C}}
\newcommand{\cD}{\mathbf{D}}
\newcommand{\cE}{{\mathbf E}}
\newcommand{\cF}{{\mathbf F}}
\newcommand{\cG}{{\mathbf X}} % CHANGE THIS!!!!
\newcommand{\cH}{{\mathbf H}}
\newcommand{\cI}{{\mathbf I}}
\newcommand{\cJ}{{\mathbf J}}
\newcommand{\cK}{{\mathbf K}}
\newcommand{\cL}{{\mathbf L}}
\newcommand{\cM}{{\mathbf M}}
\newcommand{\cN}{{\mathbf N}}
\newcommand{\cO}{{\mathcal O}}
\newcommand{\cP}{{\mathbf P}}
\newcommand{\cQ}{{\mathbf Q}}
\newcommand{\cR}{{\mathbf R}}
\newcommand{\cS}{{\mathbf S}}
\newcommand{\cT}{{\mathbf T}}
\newcommand{\cU}{{\mathbf U}}
\newcommand{\cV}{{\mathbf V}}
\newcommand{\cW}{{\mathbf W}}
\newcommand{\cX}{{\mathbf X}}
\newcommand{\cY}{{\mathbf Y}}

% San Serif LETTERS
\newcommand{\sA}{{\mathsf A}}
\newcommand{\sa}{{\mathsf a}}
\newcommand{\sB}{{\mathsf B}}
\newcommand{\B}{{\mathsf B}}
\newcommand{\ssb}{{\mathsf b}}
\newcommand{\sC}{{\mathsf C}}
\newcommand{\sD}{{\mathsf D}}
\newcommand{\sd}{{\mathsf d}}
\newcommand{\sE}{{\mathsf E}}
\newcommand{\sF}{{\mathsf F}}
\newcommand{\sG}{{\mathsf G}}
\newcommand{\sH}{{\mathsf H}}
\newcommand{\sI}{{\mathsf I}}
\newcommand{\sJ}{{\mathsf J}}
\newcommand{\sK}{{\mathsf K}}
\newcommand{\sL}{{\mathsf L}}
\newcommand{\sM}{{\mathsf M}}
\newcommand{\sMM}{{\mathsf {MM}}}
\newcommand{\sN}{{\mathsf N}}
\newcommand{\sn}{{\mathsf n}}
\newcommand{\sm}{{\mathsf m}}
\newcommand{\sv}{{\mathsf v}}
\newcommand{\sO}{{\mathsf O}}
\newcommand{\sP}{{\mathsf P}}
\newcommand{\sQ}{{\mathsf Q}}
\newcommand{\sR}{{\mathsf R}}
\newcommand{\sS}{{\mathsf \Sigma}}
\newcommand{\sT}{{\mathsf T}}
\newcommand{\sU}{{\mathsf U}}
\newcommand{\sW}{{\mathsf W}}
\newcommand{\sw}{{\mathsf w}}
\newcommand{\sV}{{\mathsf V}}
\newcommand{\sX}{{\mathsf X}}
\newcommand{\sY}{{\mathsf Y}}
\newcommand{\sZ}{{\mathsf Z}}
\newcommand{\sfS}{{\mathsf \Sigma}}

\newcommand{\CF}{{\mathsf{C}}}
\newcommand{\CCF}{{\widehat{\CF}}}
\newcommand{\Co}{{\mathsf{Co}}}
\newcommand{\sOcl}{{\mathsf O}^{\rm clp}}
\newcommand{\sUcl}{{\mathsf U}^{\rm clp}}

% Fraks
\newcommand{\fM}{{\mathfrak M}}
\newcommand{\fD}{{\mathfrak D}}
\newcommand{\fT}{{\mathfrak T}}
% Boldface chars, used for categories
\newcommand{\bC}{{\mathfrak C}}
\newcommand{\bD}{{\mathfrak D}}
\newcommand{\bF}{{\mathtt F}}
\newcommand{\ob}{{\mathtt{Obj}}}
\newcommand{\homC}{{\mathtt{Hom}}}
\newcommand{\morC}{{\mathfrak{Mor}}}
\newcommand{\Psh}{{\mathfrak{Psh}}}
\newcommand{\Sh}{{\mathfrak{Sh}}}
\newcommand{\Shsp}{{\mathfrak{Shsp}}}
\newcommand{\Poset}{{\mathfrak{Poset}}}
\newcommand{\Ab}{{\mathfrak{Ab}}}
\newcommand{\Latt}{{\mathfrak{Lat}}}
\newcommand{\iddd}{{\mathfrak{1}}}
\newcommand{\DF}{{\mathfrak{Lat}_D^F}}
\newcommand{\PF}{{\mathfrak{Poset}_F}}
\newcommand{\ARop}{{\rm AR}}
\newcommand{\wX}{{\widehat X}}
\newcommand{\sId}{{\mathsf{id}}}
%\newcommand{\sId}{{\mathsf{Id}}}

% Scr's
\newcommand{\scrA}{{\mathcal A}}
\newcommand{\scrB}{{\mathscr B}}
\newcommand{\scrM}{{\mathscr M}}
\newcommand{\scrN}{{\mathcal N}}
\newcommand{\scrF}{\mathscr{F}}
\newcommand{\scrG}{\mathscr{G}}
\newcommand{\scrE}{\mathscr{E}}
\newcommand{\scrH}{{\mathscr H}}
\newcommand{\scrJ}{{\mathscr J}}
\newcommand{\scrK}{\mathscr{K}}
\newcommand{\scrS}{\mathscr{S}}
\newcommand{\scrR}{{\mathscr R}}
\newcommand{\scrD}{{\mathscr D}}
\newcommand{\scrO}{{\mathscr O}}
\newcommand{\scrC}{{\mathscr C}}
\newcommand{\scrU}{{\mathscr U}}
\newcommand{\scrV}{{\mathscr V}}
\newcommand{\scrW}{{\mathscr W}}
\newcommand{\scrT}{{\mathscr T}}
\newcommand{\tf}{\mathcal{f}}
\newcommand{\tg}{\mathcal{g}}
\newcommand{\thh}{\mathcal{h}}
\newcommand{\LL}{{\mathrm{L}}}

\newcommand{\cFm}{{\mathcal F}_{{\rm o}}}
\newcommand{\cFme}{{\mathcal F}_{{\rho}}}
\newcommand{\cFmen}{{\mathcal F}_{{\rho}_n}}
\newcommand{\be}{{\bm{\varepsilon}}}
\newcommand{\sRA}{{\mathsf{ RA}}}
\newcommand{\sAtt}{{\mathsf{ Att}}}
\newcommand{\sRep}{{\mathsf{ Rep}}}
\newcommand{\sMorse}{{\mathsf{ Morse}}}
\newcommand{\sMRep}{{\mathsf{ MRep}}}
\newcommand{\sMTess}{{\mathsf{ MTess}}}
\newcommand{\sOrdTess}{{\mathsf{ OrdTess}}}
\newcommand{\sMRepr}{{\mathsf{ MRepr}}}
\newcommand{\sMTile}{{\mathsf{ MTile}}}
\newcommand{\sMTileC}{{\mathsf{ MTile}_\scrC}}
\newcommand{\sMTileR}{{\mathsf{ MTile}_\scrR}}
\newcommand{\sInvset}{{\mathsf{ Invset}}}
\newcommand{\sSInvset}{{\mathsf{ SInvset}}}
\newcommand{\sConvset}{{\mathsf{CXset}}}
\newcommand{\IS}{{\mathsf{ Invset}}}
\newcommand{\SIS}{{\mathsf{ SInvset}}}
\newcommand{\sPLyap}{{\mathsf{ PLyap}}}
\newcommand{\nbhd}{{\mathsf{Nbhd}}}
\newcommand{\sHom}{{\mathsf{Hom}}}
\newcommand{\sConvex}{{\mathsf{Co}}}
\newcommand{\sGrid}{{\mathsf{Grid}}}
\newcommand{\sJj}{\mathsf{J}}
\newcommand{\sJm}{\mathsf{J^\wedge}}
\newcommand{\sANbhd}{{\mathsf{ ANbhd}}}
\newcommand{\sRNbhd}{{\mathsf{ RNbhd}}}
\newcommand{\sANbhdR}{{\mathsf{ ANbhd}}_{\scrR}}
\newcommand{\sRNbhdR}{{\mathsf{ RNbhd}}_{\scrR}}
\newcommand{\sINbhd}{{\mathsf{ INbhd}}}
\newcommand{\sMNbhd}{{\mathsf{ MNbhd}}}
\newcommand{\sINbhdR}{{\mathsf{ INbhd}}_{\scrR}}
\newcommand{\sTrapR}{{\mathsf{ TrapR}}}
\newcommand{\sRepR}{{\mathsf{ RepR}}}
\newcommand{\sTrapRC}{{\mathsf{ TrapR}}_\scrC}
\newcommand{\sRepRO}{{\mathsf{ RepR}}_{\scrO}}
\newcommand{\sTrapRR}{{\mathsf{ TrapR}}_{\scrR}}
\newcommand{\sRepRR}{{\mathsf{ RepR}}_{\scrR}}
\newcommand{\sTrapRRo}{{\mathsf{ TrapR}}_{\scrR^*}}
\newcommand{\sTrapRRu}{{\mathsf{ TrapR}}_{\scrR^\rc}}
\newcommand{\sRepRRo}{{\mathsf{ RepR}}_{\scrR^*}}
\newcommand{\sABlockR}{{\mathsf{ABlock}}_{\scrR}}
\newcommand{\sRBlockR}{{\mathsf{RBlock}}_{\scrR}}
\newcommand{\sRBlockRo}{{\mathsf{RBlock}}_{\scrR^*}}
\newcommand{\sABlockC}{{\mathsf{ABlock}}_{\scrC}}
\newcommand{\sRBlockO}{{\mathsf{RBlock}}_{\scrO}}
\newcommand{\sABlock}{{\mathsf{ABlock}}}
\newcommand{\sRBlock}{{\mathsf{RBlock}}}
\newcommand{\sIBlock}{{\mathsf{IBlock}}}
\newcommand{\sIBlockR}{{\mathsf{IBlock}}_{\scrR}}
\newcommand{\sPow}{{\mathsf{F}}}
\newcommand{\sMD}{{\mathsf{MD}}}
\newcommand{\sMT}{{\mathsf{MT}}}
\newcommand{\sDS}{{\mathbf{DS}}}
\newcommand{\sDSo}{{\mathbf{DSo}}}
\newcommand{\sCDS}{{\mathbf{CDS}}}
\newcommand{\sCDSo}{{\mathbf{CDSo}}}
\newcommand{\sob}{{\mathsf{ob}}}
\newcommand{\cod}{{\mathsf{cod}}}
\newcommand{\dom}{{\mathsf{dom}}}
\newcommand{\Hom}{{\mathsf{hom}}}

\newcommand{\sOC}{{\mathsf{oc}}}
\newcommand{\sub}{{\rm sub}}
\newcommand{\subAttF}{{{\rm sub}_F\mathsf{Att}}}
\newcommand{\subANbhdRF}{{{\rm sub}_F\mathsf{ANbhd}_{\scrR}}}
\newcommand{\subANbhdF}{{{\rm sub}_{0,1}^F\mathsf{ANbhd}}}
\newcommand{\subRepF}{{{\rm sub}_{0,1}^F\mathsf{Rep}}}
\newcommand{\subRNbhdF}{{{\rm sub}_{0,1}^F\mathsf{RNbhd}}}
\newcommand{\subFR}{{{\rm sub}_F\mathscr{R}}}
\newcommand{\subO}{{{\rm sub}_{0,1}\mathsf{O}}}
\newcommand{\sIsol}{{\mathsf{ Isol}}}
\newcommand{\sMd}{{\mathsf{M}}}
\newcommand{\sLatt}{{\mathsf{L}}}
\newcommand{\sLat}{{\mathbf{Lat}}}
\newcommand{\sBLat}{{\mathbf{BLat}}}
\newcommand{\sSet}{{\mathsf{Set}}}
\newcommand{\sASet}{{\mathsf{ASet}}}
\newcommand{\sRSet}{{\mathsf{RSet}}}
\newcommand{\DG}{{(\cG,\cE)}}
\newcommand{\con}{{\leftrightarrowtriangle}}
\newcommand{\conn}{{\circlearrowleft}}
\newcommand{\reach}{{\rightsquigarrow}}
\newcommand{\sSCC}{{\mathsf{SC}}}
\newcommand{\sSC}{{\mathsf{S}}}
\newcommand{\sSAtom}{{\mathsf{SAtom}}}
\newcommand{\bomega}{{\bm{\omega}}}
\newcommand{\balpha}{{\bm{\alpha}}}
\newcommand{\bth}{{\bm{\vartheta}}}
\newcommand{\try}{{\bm{\eta}}}
\newcommand{\dyn}{{\mathscr{D}(X,\T)}}
\newcommand{\sBool}{{\mathbf{Bool}}}
\newcommand{\sBoolRng}{{\mathbf{BoolRing}}}
\newcommand{\sFBool}{{\mathbf{FBool}}}
\newcommand{\sFDLat}{{\mathbf{FDLat}}}
\newcommand{\sBDLat}{{\mathbf{BDLat}}}
\newcommand{\sRing}{{\mathbf{Ring}}}
\newcommand{\rmod}{R{\mathbf{\text{-}Mod}}}
\newcommand{\sMonoid}{{\mathbf{Monoid}}}
\newcommand{\sAlg}{{\mathbf{Alg}}}
\newcommand{\sTop}{{\mathbf{Top}}}
\newcommand{\sTco}{{\mathbf{Tco}}}
\newcommand{\MeetLat}{{\mathbf{MLat}}}
\newcommand{\sSetCat}{{\mathbf{Set}}}
\newcommand{\sAb}{{\mathbf{Ab}}}
\newcommand{\sSh}{{\mathbf{Sh}}}
\newcommand{\sPrSh}{{\mathbf{PrSh}}}
\newcommand{\sEt}{{\mathbf{Et}}}
\newcommand{\sDLat}{{\mathsf{DLat}}}
\newcommand{\sPries}{{\mathbf{Pries}}}
\newcommand{\sStone}{{\mathbf{Stone}}}
\newcommand{\FSet}{{\mathbf{FSet}}}
\newcommand{\sRC}{{\mathsf{RC}}}
\newcommand{\smuDLat}{\rho{\mathbf{DLat}}}
\newcommand{\smuFDLat}{\rho{\mathbf{FDLat}}}
\newcommand{\sSp}{\mathbf{Sp}}

\newcommand{\bas}{{\bm *}}
\newcommand{\bphi}{{\bm{\phi}}}
\newcommand{\bpsi}{{\bm{\psi}}}
\newcommand{\obphi}{{\bm{{\phi_{\bm *}}}}}
\newcommand{\obpsi}{{\bm{{\psi_{\bm *}}}}}
\newcommand{\obzeta}{{\bm{{\zeta_{\bm *}}}}}
\newcommand{\obxi}{{\bm{{\xi_{\bm *}}}}}
\newcommand{\obeta}{{\bm{{\eta_{\bm *}}}}}
\newcommand{\obtheta}{{\bm{{\theta_{\bm *}}}}}
\newcommand{\iobphi}{{\bm{{\phi^{-1}_{\bm *}}}}}
\newcommand{\iobpsi}{{\bm{{\psi^{-1}_{\bm *}}}}}
\newcommand{\iobtheta}{{\bm{{\theta^{-1}_{\bm *}}}}}
\newcommand{\monZ}{{\mathds{Z}_2}}
\newcommand{\subF}{{\mathrm{sub_F}}}

% MISC CHARACTERS
\newcommand{\Chi}{\raise .75ex\hbox{$\chi$}}
\newcommand{\bU}{{\mathbf U}}
\newcommand{\bV}{{\mathbf V}}
\newcommand{\vv}{\textit{\textbf{v}}}
\renewcommand{\aa}{\textit{\textbf{a}}}
\newcommand{\pred}[1]{\overleftarrow #1}
\newcommand{\pr}[1]{\overleftarrow #1}
\newcommand{\scc}[1]{\overrightarrow #1}
\newcommand{\Succ}{^\to}
\newcommand{\zo}{$\{0,1\}$}
\newcommand{\down}{\downarrow\!}
\newcommand{\up}{\uparrow\!}
\newcommand{\birk}{\downarrow^\vee}
\newcommand{\idd}{\mathrm{id}}
\newcommand{\al}{\mathbf{\nu}}
\newcommand{\all}{\mathbf{\varrho}}
\newcommand{\allb}{\mathbf{\varrhob}}
\newcommand{\alb}{\bm{\mathbf{\nu}}}
\newcommand{\ji}{\mathbf{\mu}}
\newcommand{\jib}{\bm{\mathbf{\mu}}}
\newcommand{\pib}{\bm{\mathbf{\pi}}}
\newcommand{\varrhob}{\bm{\mathbf{\varrho}}}
\newcommand{\varpib}{\bm{\mathbf{\varpi}}}
\newcommand{\jio}{\mathbf{\rho}}
\newcommand{\alo}{\mathbf{\sigma}}
\newcommand{\two}{{\mathbf{2}}}
\newcommand{\Ob}{{\textnormal{Ob}}}
\newcommand{\cdm}{\textnormal{Cdm}}
\newcommand{\Img}{\textnormal{Im}}
\newcommand{\Idt}{\textnormal{Id}}
\newcommand{\smin}{\smallsetminus}
\newcommand{\im}{\textnormal{im~}}

\newcommand{\FPoset}{\mathbf{FPoset}}
\newcommand{\sPoset}{\mathbf{Poset}}
\newcommand{\ePoset}{\epsilon\mathsf{Poset}}
\newcommand{\eFPoset}{\epsilon\mathsf{FPoset}}
%\newcommand{\DF}{\mathrm{Latt_F^D}}

% ACCENTS
\renewcommand{\hat}{\widehat}

% Fat letters

\newcommand{\ba}{{\mathsf a}}
\newcommand{\bb}{{\mathsf b}}
\newcommand{\bu}{{\mathsf u}}
\newcommand{\bs}{{\mathsf s}}
\newcommand{\bc}{{\mathsf c}}

% MISC COMMANDS
\newcommand{\eproof}{\hfill{\vrule height5pt width5pt depth0pt}\medskip}
\newcommand{\rmref}[1]{{\rm(\ref{#1})}}
\newcommand{\vgln}{\lb\begin{array}{rcl}}
\newcommand{\eindvgln}{\end{array}\right.}
\newcommand{\alphaOg}{\alpha_{\rm o}(\gamma_x^-)}
\newcommand{\balphaO}{\balpha_{\rm o}}
\newcommand{\alphaO}{\alpha_{\rm o}}
\newcommand{\omegaOg}{\omega_{\rm o}(\gamma_x^-)}
\newcommand{\bomegaO}{\bomega_{\rm o}}

%  Maps and Arrows

\def\mapright#1{\stackrel{#1}{\longrightarrow}}
\def\mapleft#1{\stackrel{#1}{\longleftarrow}}
\def\mapdown#1{\Big\downarrow\rlap{$\vcenter{\hbox{$\scriptstyle#1$}}$}}
\def\mapup#1{\Big\uparrow\rlap{$\vcenter{\hbox{$\scriptstyle#1$}}$}}
\def\mapne#1{\nearrow\rlap{$\vcenter{\hbox{$\scriptstyle#1$}}$}}
\def\mapse#1{\searrow\rlap{$\vcenter{\hbox{$\scriptstyle#1$}}$}}
\def\mapsw#1{\swarrow\rlap{$\vcenter{\hbox{$\scriptstyle#1$}}$}}
\def\mapnw#1{\nwarrow\rlap{$\vcenter{\hbox{$\scriptstyle#1$}}$}}

% INDEXING MACRO

\newcommand{\symindex}[1]{{\index{#1}}}

% MATH 
\newcommand{\isdef}{\stackrel{\rm def}{=}}
\newcommand{\opensubset}{\stackrel{\rm open}{\subset}}
\newcommand{\conv}{{\hbox{conv}\,}}
\newcommand{\dist}{{\rm dist}\,}
\newcommand{\rank}{\hbox{\rm rank}\,}
\newcommand{\res}{{\rm res}}
\newcommand{\covv}{{\rm cov}}
\newcommand{\cl}{{\rm cl}\,}%{\text{cl}}
\newcommand{\interior}{\hbox{\rm int}\,}
\newcommand{\diam}{{\rm diam}\,}
\newcommand{\Int}{{\rm int\,}}
\newcommand{\Inv}{\mbox{\rm Inv}}  
\newcommand{\bd}{\partial}
\newcommand{\f}{\varphi}
\newcommand{\ft}{\overline{\varphi}}
\newcommand{\supp}[1]{\left| {#1}\right|}
\newcommand{\col}{{\rm col}\,}
\newcommand{\ord}{{\rm Ord}}
\newcommand{\lhom}{{\rm LHom}}

%fractions
\newcommand{\half}{{\textstyle{1\over 2}}}
\newcommand{\quarter}{{\textstyle{1\over 4}}}
\newcommand{\threequarters}{{\textstyle{3\over 4}}}
\newcommand{\threesixteen}{{\textstyle{3\over 16}}}

% MISC MACROS
\newcommand{\cFk}{{\mathsf f}^{(0,k]}}
\newcommand{\cFkk}{{\mathsf f}^{[0,k]}}
\newcommand{\mb}{{\mbox{-}}}
%\newcommand{\cFk}{{\mathcal F}^{(0,k]}}
%\newcommand{\cFkk}{{\mathcal F}^{[0,k]}}

%  linear algebra terms
\newcommand{\codim}{\mathop{\rm codim }}
\renewcommand{\span}{\mathop{\rm span}}

%  combinatorial geometry terms
\renewcommand{\star}{\mathop{\rm star }}
\newcommand{\hull}{\mathop{\rm hull }}
\newcommand{\id}{\mathop{\rm id }\nolimits}
\newcommand{\image}{\mathop{\rm image }\nolimits}
\newcommand{\gker}{\mathop{\rm gker }\nolimits}
\newcommand{\Ave}{\mathop{\rm Ave}\nolimits}
\newcommand{\Con}{\mathop{\rm Con}\nolimits}
\newcommand{\sgn}{\mathop{\rm sgn }\nolimits}
\newcommand{\spec}{\mathop{\rm spec }\nolimits}
\newcommand{\cov}{\mathop{\rm cov}}

% comment commands
\newcommand{\commD}{\todo[inline, color=blue!40]}
\newcommand{\commK}{\todo[inline, color=green!40]}
\newcommand{\commV}{\todo[inline, color=red!40]}

\newcommand{\btau}{\bm{\tau}}
\newcommand{\bmu}{\bm{\mu}}
\newcommand{\biota}{\bm{\iota}}
\newcommand{\bpi}{\bm{\pi}}

% sheaves

\newcommand{\shAtt}{\mathscr{Att}}
\newcommand{\shC}{\mathscr{C}}
\newcommand{\shD}{\mathscr{D}}

\title{  Continuation sheaves in dynamics: sheaf cohomology\\ and bifurcation\thanks{The work of W.D.K.\ was partially supported by the Army Research Office under award W911NF1810306. }}


\titlerunning{Continuation sheaves in dynamics}        % if too long for running head

\author{K. Alex Dowling         \and
        William D. Kalies \and \\
        Robert C.A.M. Vandervorst %etc.
}


\institute{K. Alex Dowling \at
              Rutgers University \\
              \email{kad378@scarletmail.rutgers.edu}           %  \\
           \and
           William D. Kalies \at
              Florida Atlantic University\\
              \email{wkalies@fau.edu}
              \and
              Robert C.A.M. Vandervorst \at
              VU University\\
              \email{r.c.a.m.vander.vorst@vu.nl}
}

\date{Version date: \today}

\maketitle

\begin{abstract}
Continuation of algebraic structures in families of dynamical systems is described using category theory, sheaves, and lattice algebras. Well-known concepts in dynamics, such as attractors or invariant sets, are formulated as functors on appropriate categories  of dynamical systems mapping to categories of lattices, posets, rings or abelian groups. Sheaves are constructed from such functors, which encode data about the continuation of structure as system parameters vary. Similarly, morphisms for the sheaves in question arise from natural transformations. This framework is applied to a variety of lattice algebras and ring structures associated to dynamical systems, whose algebraic properties carry over to their respective sheaves. Furthermore, the cohomology of these sheaves are algebraic invariants which contain information about bifurcations of the parametrized dynamical systems.
\end{abstract}

%%%%%%%%%%%%%%%%%%%%%%%%%%%%%%%%%%%%%%%%%%%%
%%%%%%%%%%%%%% main body of the paper %%%%%%%%%%%%%%%%
%%%%%%%%%%%%%%%%%%%%%%%%%%%%%%%%%%%%%%%%%%%%

\section{Introduction}
\label{sec:intro}

Dynamical systems theory is the study of the structure of invariant sets, particularly those sets that govern the long-term behavior of a system.  
As a dynamical system is perturbed, its global dynamics can change qualitatively through bifurcation.  
The global dynamics of a system can be described by the structure of its attractors. Indeed, in his fundamental decomposition theorem, Conley \cite{Conley} uses the set of all attractors in a system to establish a global decomposition into minimal (chain)-recurrent components and connecting orbits between them. % \cite{Robinson}.
The algebraic structure that underlies this decomposition is codified in the fact that the set of all attractors naturally forms a bounded, distributive lattice \cite{KMV-1a}. In a series of papers \cite{KMV-0,KMV-1a,KMV-1b,KMV-1c,KastiKV,GKV} the theoretical framework for such dynamically meaningful algebraic structures has been developed. This framework has been used to design algorithms to compute global dynamical information rigorously for explicit maps, cf.~\cite{database}. 
To understand dependence on parameters in dynamical systems we develop
a natural mathematical framework to capture this variation.
For a parametrized family of dynamical systems, structures like the lattice of attractors at a fixed parameter value may vary dramatically under small perturbations. Indeed, bifurcations can occur on a Cantor set with positive measure in the parameter space.
Continuation relates these structures for varying parameter values. 

Continuation of isolated invariant sets has been the central theme in Conley index theory. Conley \cite{Conley} and Montgomery \cite{Mont}, and later Salamon \cite{Salamon}, formulate continuation in terms of the \emph{space of isolated invariant sets}:
\[
\mathrm{\mathrm{\Pi}}[\sIsol] := \Bigl\{ (\phi, S) ~\big|~ \phi \text{ a flow on } X \text{ and } S \text{ an isolated invariant set for } \phi \Bigr\}.
\]
Two pairs are "close" in the space of isolated invariant sets if the flows are "close" and the isolated invariant sets share a common isolating neighborhood. Two isolated invariant sets are related by continuation if they lie in the same quasicomponent of the space of isolated invariant sets. Both Montgomery and Salamon give proofs of the following crucial result:

\centerline{
\emph{Two invariant sets related by continuation share the same Conley index.}}

\noindent Montgomery proves this in the language of \etale spaces (sheaf theory). 
By choosing an appropriate topology on $\mathrm{\Pi}[\sIsol]$ the map
\[
(\phi,S) \xmapsto{\pi} \phi,
\]
where the flow $\phi$ is an element of a topological space of flows on $X$, is a local homeomorphism.
With it, he constructs a long exact sequence of \etale spaces relating the cohomology of the invariant sets, the index of invariant sets, and the cohomology of some asymptotic sets. 
Montgomory's approach is not to be confused with sheaf cohomology.
By contrast, Salamon provides a direct flow-defined homotopy to establish the invariance of the Conley index. This avoids much of the algebraic topological machinery necessary for Montgomery's approach. Salamon's explicit method appeals to an analytic perspective of dynamics.
The latter approach does  not suffice for understanding the continuation of algebraic structures such as lattices of attractors or Morse representations.

In \cite{fran3} Franzosa extends the approach of Kurland for attractor-repeller pairs, cf.\ \cite{Kurland}, to develop an intricate theory of continuation of Morse decompositions. The idea is to augment the concept of the space of isolated invariant sets to a space of ``Morse decompositions" with an appropriate topology. 
Our interest in data-driven dynamics, and the recent success of sheaf theory in topological data analysis, motivate us to revisit Montgomery's viewpoint. We acknowledge this increased algebraic topology overhead with the confidence that this abstract formulation will yield a comprehensive computational framework for continuation.
The abstract formalism developed in this paper gives an elementary treatment of the continuation theory of Morse representations extending the theory by Franzosa which makes it applicable to techniques from sheaf theory such as sheaf cohomology.

In a series of papers, cf.\ \cite{KMV-1a,KMV-1b,KMV-1c}, we developed an algebraic theory of attractors via distributive lattice theory. We use attractors as the starting point of our approach.
In Diagram \ref{fig:outline}[left] below we outlined the algebraic theory of attractors. Diagram \ref{fig:outline}[middle] reformulates the left diagram in terms of functors on appropriately chosen categories of dynamical systems and Diagram \ref{fig:outline}[right] gives the associated \etale spaces.
% \vskip-.4cm
% \begin{figure}[h]
    %\centering
    \begin{equation}\label{fig:outline}
 %   \begin{center}
    \begin{tikzcd}
    \sANbhd(\phi)\arrow[d,"\omega_\phi"]\arrow[r,"{}^c",leftrightarrow] & \sRNbhd(\phi) \arrow[d,"\alpha_\phi"]\\ \sAtt(\phi)\arrow[r,"{}^*",leftrightarrow] & \sRep(\phi)
    \end{tikzcd}\;
    \begin{tikzcd}
    \sANbhd\arrow[d,"\omega", Rightarrow]\arrow[r,"{}^c", Leftrightarrow] & \sRNbhd \arrow[d,"\alpha", Rightarrow] \\ \sAtt\arrow[r,"{}^*", Leftrightarrow] & \sRep
    \end{tikzcd}\;
    \begin{tikzcd}
    \mathrm{\Pi}[\sANbhd]\arrow[d, "{\mathrm{\Pi}[\omega]}"]\arrow[r, "{\mathrm{\Pi}[{}^c]}",leftrightarrow] & \mathrm{\Pi}[\sRNbhd] \arrow[d, "{\mathrm{\Pi}[\alpha]}"] \\ \mathrm{\Pi}[\sAtt]\arrow[r, "{\mathrm{\Pi}[{}^*]}",leftrightarrow] & \mathrm{\Pi}[\sRep]
    \end{tikzcd}
 %   \end{center}
    \end{equation}
    % \caption{The diagram from \cite{KMV-1a} [left] appears categorically in the second diagram [middle], and then in the associated \etale spaces [right].}
    % \label{fig:outline}
% \end{figure}
% \vskip-.4cm
The first diagram [left], which was established in  \cite[Diag.\ (1)]{KMV-1a}, appears categorically in the second diagram [middle], and then in the associated \etale spaces [right].
For example, the space $\mathrm{\Pi}[\sAtt]$ is the space of points
$(\phi,A)$, where $A\in \sAtt(\phi)$ is an attractor. The latter is the analogue of the space of isolated invariant sets. 
Diagram  \eqref{fig:outline}[right] allows us to define a \emph{sheaf of attractors} over the space of dynamical systems, cf.\ Fig.\ \ref{fig:sec1}.
In the first sections of the paper we develop the categorical theory of continuation for contravariant functors. We can apply this theory in various dynamical settings such as the lattice of attractors, the semi-lattice of isolated invariant sets, but also the (non-distributive) lattice of Morse representations. 
For the latter
 we adopt the algebraic treatment of Morse representations and decompositions introduced in \cite[Def.\ 7]{KMV-1c},  which repairs the classical approach of labeling of invariant sets.
%as opposed to the classical labeling of invariant sets. 
%
%\emph{A Morse representation} $\sM$ for a dynamical system $\phi$ is a finite, partially ordered set of mutually disjoint, nonempty, compact invariant sets. This poset must have that for every point in phase space $x$, there exists $M<M'$ such that $\omega_\phi(x)\subset M$ and for each complete orbit $\gamma_x$ with $x\notin \cup_{M\in \sM} M$, $\alpha_\phi(\gamma_x^-)\subset M'$. 
Morse decompositions are formulated as an order embedding from a Morse representation into a finite poset. By applying the categorical theory to Morse representations and studying the resulting sheaves we obtain a generalization of 
Franzosa's theory of continuation of Morse decompositions. 
In the various settings we can define associated sheaves, e.g.\ the sheaf of attractors and the sheaf of Morse representations. If we invoke sheaf cohomology the algebraic structures of attractors and Morse representations can be used to define new invariants. We investigate sheaf cohomology in the setting of bifurcation theory as an illustration.
With the strides made in computational  dynamics 
% (lattice structure papers, database paper, computational framework for connection matrices, etc.), 
and the success of sheaf theory in topological data analysis (cellular sheaves), we have the necessary prerequisites to achieve a computational method for modeling continuation.

% \corrc
% This last paragraph is a bit terse. Maybe add something that sheaves provide the appropriate language to formulate continuation of Morse representation and justifies the definition of Morse decomposition.
% <<>>

% {\color{purple} Sheaf cohomology, which we will show to be capable of detecting bifurcations in dynamics, can be approximated with cellular sheaf cohomology.} 
%\vskip.5cm
%\vspace{15ex}
%

%

% The outline of the paper is as follows.
% In Section \ref{sec:DSCat} we review the topological and categorical structures associated to dynamical systems. Section \ref{sec:AttFun} demonstrates how existing dynamical invariants fit into the functorial language. Then, Section \ref{sec:ShStr} builds the \etale space construction from continuation frames, and investigates induced maps from natural transformations. Section \ref{sec:algconstr} utilizes Section \ref{sec:ShStr} for algebraic constructions in dynamics. Section \ref{sec:shofsec} translates to sheaves of sections, and Section \ref{sec:bifs} constructs abelian sheaves, discusses the implications of parametrized systems, and establishes conjugacy invariance. Section \ref{ssc:examples} begins exploration of sheaf cohomology's connection to bifurcations, and Section \ref{sec:exonepar} computes examples of this cohomology for several one-parameter bifurcations.

As a first step, the focus of this paper is to construct and study sheaves which encode the continuation of structures in dynamics. The first seven sections of the paper detail an abstract approach to building sheaves for an arbitrary structure. We routinely return to attractors to showcase how this approach may be applied.

Our first goal is to formulate algebraic structures in dynamics as functors. Section \ref{sec:DSCat} equips dynamical systems with the compact-open topology, and then with a categorical structure using the notion of topological conjugacies. This yields the domain category for the aforementioned functors, and a topology to attach algebraic information to with sheaves. 
Section \ref{sec:AttFun} then explicitly shows that the attracting neighborhood lattice $\sANbhd(\phi)$ and the attractor lattice $\sAtt(\phi)$ can be expressed as functors from the category of dynamical systems to a category of lattices. These constitute an example to which we apply the later theory. 
 Section \ref{sec:ShStr} gives prerequisites for continuation: A category with a topology on the objects, and a pair of functors we call a "continuation frame". We prove the existence of an \etale space encoding continuation for these functors. The end of the section constructs morphisms of the \etale spaces using natural transformations between the corresponding functors. 
Section \ref{contatt12} first applies the framework built in Section \ref{sec:ShStr} to the attractor case from Section \ref{sec:AttFun}. This yields the \etale space of attractors. Furthermore, we formulate a morphism of \etale spaces from the dual repeller operator, seen in Fig.\ \ref{fig:outline}. Lastly, we describe the functorial setup for finite sublattices of attractors and Morse representations. This will eventually define a Morse representation sheaf.
 While Section \ref{sec:ShStr} produces \etale spaces for attractors, repellers, etc., in Section \ref{sec:algconstr} we port over their algebra. We begin by equipping the attractor and repeller \etale spaces with the binary lattice operations. Then, the Conley form on attractors is stated on the level of the attractor \etale space. Later, when we discuss sheaf cohomology, the Conley form will be a crucial tool in building sheaves in an abelian category. The end of Section \ref{sec:algconstr} expands on this, detailing how the algebra of attractors can be stored in a ring. 
Section \ref{sec:shofsec} passes through the equivalence between \etale spaces and sheaves. \emph{From an abstract continuation frame we build a sheaf which encodes the continuation of the unstable structure.} The attractor functor begets an \emph{attractor lattice sheaf}, and the Conley form becomes a morphism of sheaves. We also discuss the functors built at the end of Section \ref{sec:algconstr}, which give us $\sRing$-valued sheaves storing the continuation of attractors. To conclude the section we consider the sheaf of finite attractor sublattices, and the sheaf of Morse representations, as set up in Section \ref{contatt12}. 
\vspace{2ex}
\begin{figure}[h!]
\scalebox{1.25}{
\begin{picture}(300,135)
\put(55,20){\includegraphics[width=6.7cm]{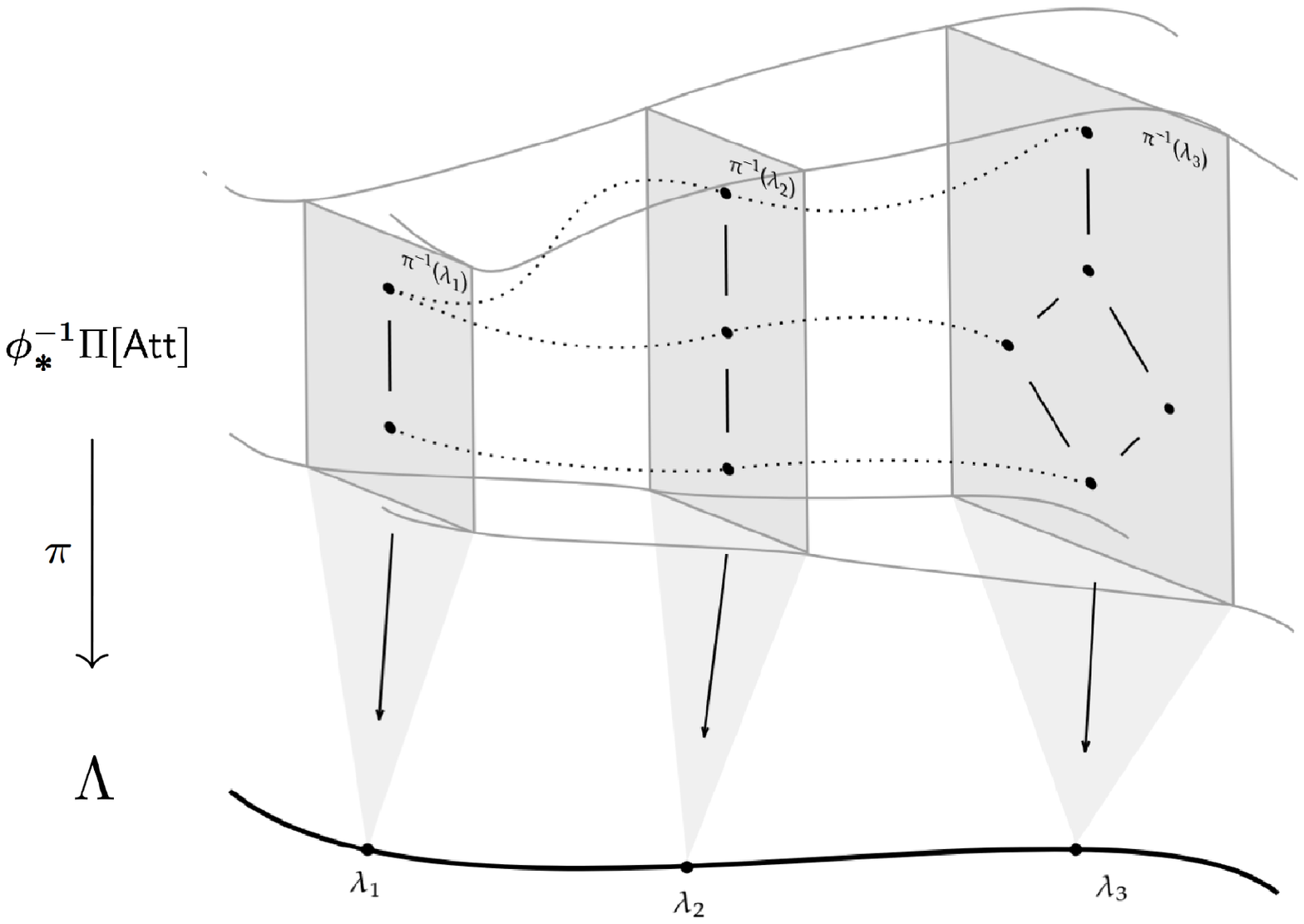}}
\end{picture}
}
\vspace{-4ex}
\caption{ Given a parametrized dynamical system $\obphi$, we construct an \etale space of attractors $\iobphi\mathrm{\Pi}[\sAtt]$ over parameter space, cf.\ Sect.\ \ref{sn1}. The fiber at a parameter value $\lambda\in \mathrm{\Lambda}$ is the attractor lattice of $\phi_\lambda$. Global sections are illustrated by dotted lines. The failure of these global sections to reach all attractors can be measured using sheaf cohomology.
%Global sections of attractors over the base $\mathrm{\Lambda}$, cf.\ Sect.\ \ref{sn1} for more details on the attractor lattices.
}
\label{fig:sec1}
\end{figure}    
\vspace{-4ex}
In the last three sections of the paper, we use the continuation sheaves to study bifurcations in parametrized dynamical systems.
Section \ref{sec:bifs} pulls the continuation sheaves back to parameter space for any given parametrized dynamical system. 
For a topological space $\mathrm{\Lambda}$ we define
a \emph{parametrized dynamical system} on $X$ as a continuous map
$\bphi\colon \T\times X\times \mathrm{\Lambda} \to X$ such that 
$\phi^{\lambda} := \bphi(\cdot,\cdot,\lambda)$ is dynamical system on $X$ for all $\lambda\in \mathrm{\Lambda}$.
The map
$\lambda \xmapsto{\obphi} \phi^\lambda$, called the \emph{transpose}, is a continuous map without additional assumptions on the topological spaces $\mathrm{\Lambda}$ and $X$.
In
Section \ref{ssc:conjinvariance} we show that the continuation of attractors is conjugacy invariant.
\vskip.2cm
\noindent {\textbf{Theorem}}~(Conjugacy Invariance Theorem, cf.\ Thm.\ \ref{thm:CIT})
{\em
Let $X$, $Y$ be 
compact metric spaces. Suppose $\obphi\colon \mathrm{\Lambda} \to \sDS(\T, X)$ and $\obpsi\colon \mathrm{\Lambda} \to \sDS(\T, Y)$ are conjugate parametrized dynamical systems.
Then, the \etale spaces $\iobphi\mathrm{\Pi}[\sAtt]$ and $\iobpsi\mathrm{\Pi}[\sAtt]$ are homeomorphic.
}

% We prove conjugacy invariance in Theorem \ref{thm:CIT}: the attractor lattice sheaves built from conjugate parametrized systems are isomorphic.
%
Section \ref{ssc:examples} applies the continuation sheaves to bifurcations. A stable parametrization is shown to produce locally constant sheaves. 
The pull-back for the attractor sheaf is denoted by $\scrA^{\obphi}$. In terms of sheaf cohomology with respect to $\scrA^{\obphi}$ we obtain the following result:
\vskip.2cm
\noindent {\textbf{Theorem}}~(cf.\ Thm.\ \ref{relshco3})
{\em
Let $\mathrm{\Lambda}$ be both contractible and locally contractible, and let
 $\mathrm{\Lambda}'\subset \mathrm{\Lambda}$ be a deformation retract of $\mathrm{\Lambda}$ 
with  $\obphi$  stable on $\mathrm{\Lambda}'$.
% Let $\mathrm{\Lambda}$ be a {\color{purple} contractible manifold} and $\mathrm{\Lambda}'\subset \mathrm{\Lambda}$ be a submanifold which is a deformation retract of $\mathrm{\Lambda}$ 
% with  $\obphi$  stable on $\mathrm{\Lambda}'$.
Suppose that 
\[
H^k(\mathrm{\Lambda}, \mathrm{\Lambda}';\scrA^\obphi)\neq0,
\quad\text{ for some }\quad k\ge 0.
\]
Then, there  exist a bifurcation point in $\lambda_0 \in \mathrm{\Lambda}\smin \mathrm{\Lambda}'$.
}
\vskip.2cm

% So with conditions on the topology of the parameter space, we prove that nontrivial cohomology of the attractor sheaf implies the existence of bifurcations.
%
Section \ref{sec:exonepar} computes attractor sheaf cohomology for the pitchfork, saddle-node, transcritical, and S-shaped bifurcations. 
% We study the difference between the attractor and free attractor sheaf cohomologies and discuss what this means dynamically.
\vskip.2cm

\noindent {\textbf{Theorem}}~(cf.\ Thm.\ \ref{thm:pitchchm})
{\em
Let $\obphi$ be a parametrized dynamical system over $\R$ with a pitchfork bifurcation at $\lambda_0$. Then,
\[
\scrA^\obphi ~\text{~ is acyclic and ~}~ H^0(\mathrm{\Lambda};\scrA^\obphi) \cong\Z_2^3.
\]
Moreover,
there exists a value $\lambda_0\in \R$ such that
\[
H^k\bigl(\mathrm{\Lambda}, \mathrm{\Lambda}';\scrA^\obphi\bigr) \cong \begin{cases} \Z_2^2 &\textnormal{if } k = 1\textnormal{ and } a>\lambda_0; \\
0 & \textnormal{if $k\neq  1$ or $a\le \lambda_0$,} \end{cases}
\]
where $\mathrm{\Lambda}' = [a,\infty)$,
Furthermore, for $\mathrm{\Lambda}' := (-\infty,a]$, then
$H^k\bigl(\mathrm{\Lambda}, \mathrm{\Lambda}';\scrA^\obphi\bigr) \cong 0$ for all $k$ and for all $a\in \R$.
}
\vskip.2cm
Different types of bifurcations can have different cohomology in their attractor sheaves, but if two systems experience the same type of bifurcation, the cohomology is isomorphic. We propose this as a tool for classifying bifurcations, in much the same way singular homology classifies topological spaces. We believe this invariant to be computable by utilizing the existing theory for combinatorial dynamics \cite{KMV-1b,database} and cellular sheaf cohomology \cite{curry2015discrete}. This will be the subject of future work.

\section{Categories of dynamical systems}
\label{sec:DSCat}

% We give spaces of
% dynamical systems a categorical structure as well as a topology.
Throughout this paper we use the following
definition of dynamical system.
We give spaces of
dynamical systems a categorical  as well as a topological structure as outlined below.

\begin{definition}
\label{defn:local-dyn}
Let $(X,\scrT)$ be a topological space and let $\T$ be the (additive) topological  monoid (or group) with topology $\scrT_\T$.
A \emph{dynamical system} is a continuous map 
$\phi\colon \T \times X \to X$ that satisfies
\begin{enumerate}
\item [(i)] $\phi(0,x) = x$ for all $x\in X$;
\item [(ii)] $\phi(t,\phi(s,x)) = \phi(t+s,x)$ for all $s,t\in \T$ and all $x\in X$.
\end{enumerate}
The set of all dynamical systems on the {\em phase space} $X$ with {\em time space} $\T$
is denoted by $\sDS(\T,X)$. \label{pg:dscat} Also, $\phi(t,x)$ may be denoted $\phi_t(x).$
\end{definition}

%\corrl we talk about this backward extension later when it is needed in Section 5.1 -- BK 12-28-21 Agree with new part RC 1-12-22 Agreed AD 1-12-22
%<<
%The time space $\T$ is either $\Z$, $\Z^+$, $\R$, or $\R^+$. In case of $\Z^+$ and $\R^+$ we may extend $\phi$ for negative time by $\phi_{-t} = \phi_t^{-1}$, the inverse image.
%If the topology $\scrT_\T$ on $\T=\R$ or $\R^+$ is the discrete topology, then $\phi$ is a one-parameter family of continuous maps $\phi_t := \phi(t,\cdot)$ on $X$ satisfying the (semi)group property in (ii), and the continuity in time does not place any restriction on the system. In applications, for example those arising from differential equations, it is common for the topology on $\T=\R$ to be the standard topology, which we will assume throughout the rest of this paper, but certain results do not require the topology $\scrT_\T$ to have specific properties. %Therefore we typically use the notation $\phi_t(x).$
%Certain properties of the phase space topology $\scrT$ do play a crucial role. In particular, for clarity of the presentation of the main ideas of this paper, \emph{we always consider the phase space $X$ to be a compact topological space}. For some results, such as Theorem \ref{thm:CIT}, we additionally assume a metric on $X$. Such restrictions are explicitly stated and explained when needed.
%||
The time space $\T$ is either $\Z$, $\Z^+$, $\R$, or $\R^+$. \label{pg:timespace}
%In case of $\Z^+$ and $\R^+$ we may extend $\phi$ for negative time by $\phi_{-t} = \phi_t^{-1}$, the inverse image.
%If the topology $\scrT_\T$ on $\T=\R$, or $\R^+$ is the discrete topology, then $\phi$ is a one-parameter family of continuous maps $\phi_t := \phi(t,\cdot)$ on $X$ satisfying the (semi)group property in (ii), and the continuity in time does not place any restriction on the system. 
In applications, for example those arising from differential equations, it is common for the topology on $\T=\R$ (or $\R^+$) to be the standard topology, which we assume throughout the rest of this paper, but certain results do not require the topology $\scrT_\T$ to have specific properties. %Therefore we typically use the notation $\phi_t(x).$
Certain properties of the phase space topology $\scrT$ do play a crucial role. In particular, for clarity of the presentation of the main ideas of this paper, \emph{we always consider the phase space $X$ to be a compact topological space}. For some results, such as Theorem \ref{thm:CIT}, we additionally assume a metric on $X$. Such restrictions are explicitly stated and explained when needed.
%>>

We endow $\sDS(\T,X)$ with a suitable topology. One natural choice arises by viewing $\sDS(\T,X)$ as a function space with the compact-open topology, i.e. the topology generated by the subbasis of sets of the form 
\[
\bigl\{\phi~|~\phi(K) \subset U\hbox{ for $K$ compact in $\T \times X$ and $U$ open in $X$}\bigr\}
\]
by varying pairs $(K,U)$.%,
%cf.\ Appendix \ref{cot}.

Next we endow $\sDS(\T,X)$  with a categorical structure  and  refer to   $\sDS(\T,X)$ as the \emph{category of dynamical systems on $X$ over $\T$}. An object $\phi \in \sob(\sDS(\T,X))$ is a dynamical system $\phi\colon \T\times X \to X$. A morphism in $\Hom(\phi,\psi)$ is defined as $\tau\times h$ such that
\begin{enumerate}
    \item [(i)] $h\colon X\to X$ is a continuous map;
    \item [(ii)] $\tau\colon \T\times X \to \T$ is a continuous reparametrization 
that is strictly monotone and  bijective for each $x$ and  satisfies $\tau(0,x) =0$; 
\item [(iii)] the following diagram commutes. cf.~\cite{deVries},
\begin{center}
\begin{tikzcd}
\T\times X \arrow[d,"\tau\times h"]\arrow[rr,"\phi"] && X\arrow[d,"h"] \\
\T\times X 
\arrow[rr,"\psi"] && X
\end{tikzcd}
\end{center}
\end{enumerate} 
 We refer to such a morphism $\tau\times h$ as a (topological) quasiconjugacy. Note that $\hom(\phi,\psi)$ can also be endowed with the compact-open topology, so that both the objects and the hom-set of $\sDS(\T, X)$ are topological spaces.
We abuse notation so that  an open subset $\mathrm{\Omega}\subset\sob(\sDS(\T,X))$ is referred to as an open set $\mathrm{\Omega}$ in $\sDS(\T,X)$.

\begin{remark}
Note that the conditions on reparametrizations imply that $\tau=\id$ in the case that $\T=\Z$, cf.\ \cite[II(7.2)]{Vries2}. This in part motivates the terminology of quasiconjugacy. When $\T=\R$, the case $h=\id$ yields a reparametrization  of time $\tau(t,x)$.
\end{remark}

\begin{remark}\label{rem:DS}
For special subsets of dynamical systems, such as smooth flows on a manifold, topologies other than the compact-open topology 
may be more appropriate.
\emph{For clarity of presentation, we use the notation  $\sDS(\T,X)$ to mean that the objects and morphisms of this category are given the compact-open topology.}
However, in other cases, similar results to those obtained for $\sDS(\T,X)$  follow from the abstract theory presented in Section~\ref{sec:ShStr}.   
\end{remark}

\begin{remark}\label{rem:hom}
More restrictive choices of the set of morphisms lead to  subcategories. For example, one may consider from least restrictive to most restrictive: topological (semi)equivalence with reparametrization of time, topological (semi)conjugacy, or no structure on the morphism set, ie. $\Hom(\phi,\phi)=\{{\rm id}\times{\rm id}\}$ and $\Hom(\phi,\psi)=\emptyset$ when $\phi\neq\psi$, cf.~\cite{deVries,Robinson}.
\end{remark}

%Considering dynamical systems without specifying the time space $\T$ or the topological space $X$ yields the category $\sDS$ of all dynamical systems. In this case we extend the morphisms to be maps $\tau \times f\in \Hom(\phi,\psi)$ consisting of continuous maps $f\colon X\to Y$ and continuous time reparametrizations  $\tau\colon \T\times X \to \Sb$ (strictly monotone, bijective for each $x$ and $\tau(0,x) =0$) such that  the following diagram commutes:
%\begin{center}
%\begin{tikzcd}
%\T\times X \arrow[d,"\tau\times f"]\arrow[rr,"\phi"] && X\arrow[d,"f"] \\
%\Sb\times Y 
%\arrow[rr,"\psi"] && Y 
%\end{tikzcd}
%\end{center}
%Morphisms will again be referred to as quasiconjugacies. Morphisms with $f$ surjective are called \emph{semiconjugacies}, and when $f\colon X\to Y$ is a homeomorphism, they are called \emph{conjugacies}.
%If the time space is fixed, the associated subcategory is denoted by $\sDS(\T)$.
%We do  typically not equip $\sDS$ or $\sDS(\T)$ with a topological structure on the objects and morphisms. 

\section{Functoriality of dynamics}

\label{sec:AttFun}

The study of a dynamical system often focuses on the properties of its invariant sets. A subset $S\subset X$ is {\em invariant} if it is the union of complete orbits, or equivalently $\phi_t(S)=S$ for all $t\in\T$. One of the most important classes of invariant sets are the attractors. In \cite{KMV-1a}, it is shown that the set of attractors has the algebraic structure of a bounded, distributive lattice.
In this section, we characterize such algebraic structures in terms of functors on the category of dynamical systems.

For a given dynamical system $\phi\colon \T\times X \to X$ and a subset $U\subset X$, the {\em maximal invariant set in $U$} is
\[
\Inv_\phi(U) := \bigcup\{S\subset U : \phi(t, U) = U \textnormal{ for all } t\in \T^+\}.
\]\label{pg:inv}
The {\em omega-limit set} of $U$ is defined by
\[
\omega_\phi(U) := \bigcap_{t\ge 0} \cl \bigcup_{s\ge t} \phi_s(U).
\]\label{pg:omegalimit}
Recall from \cite{KMV-1a} some properties of $\omega_\phi(U)$.
\begin{enumerate}
    \item[(i)] $\omega_\phi(U)$ is compact, closed, and nonempty whenever $U\neq \emptyset$,
    \item[(ii)] $\omega_\phi(U)$ is an invariant set,\footnote{For a forward invariant set $U$, i.e. $\phi_t(U)\subset U$ for all $t\ge 0$, it holds that $\Inv_\phi(\cl U) = \omega_\phi(U)$.} 
    \item[(iii)] $\omega_\phi\bigl(\omega_\phi(U)\bigr) = \omega_\phi(U)$,
    \item[(iv)] $\omega_\phi(\cl U)=\omega_\phi(U)$,
    \item[(v)]$\omega_\phi(U\cup V) = \omega_\phi(U) \cup \omega_\phi(V)$.
\end{enumerate}

A subset $U\subset X$ is called an {\em attracting neighborhood} if $\omega_\phi(U)
\subset \Int U$. Attracting neighborhoods form a bounded, distributive lattice denoted by $\sANbhd(\phi)$. The binary operations are $\cap$ and $\cup$, see \cite{KMV-1a}. \label{pg:anbhd}
A subset $A\subset X$ is called an {\em attractor} if there exists an attracting neighborhood $U\subset X$ such that $A=\omega_\phi(U)$, which is a neighborhood of $A$ by definition. Attractors are compact, closed invariant sets, and the set of all attractors is a bounded, distributive lattice $\sAtt(\phi)$ with binary operations: $A\vee A'=A\cup A'$ and
$A\wedge A' := \omega_\phi(A\cap A')$, cf.\ \cite{KMV-1a}. \label{pg:att}

\begin{remark}
In the above listed properties of omega-limit sets and attractors, the compactness of $X$ is crucial. If we drop the compactness assumption on $X$, some of the properties, such as invariance and idempotency, do not hold in general.
\end{remark}

%\corrl BK 12-28-31, Ok RC 1-12-22, Agreed AD 1-12-22
%<<
%The categorical structure of $\sDS$ can now be used to reformulate the above lattices in terms of functors.
%For notational convenience we write $\psi_t^\dagger := \psi(\tau(t,\cdot),\cdot)$
%||
When the spaces $\T,X$ are fixed, we often write $\sDS$ in place of $\sDS(\T,X).$
The categorical structure of $\sDS$ can now be used to reformulate the above lattices in terms of functors.
For notational convenience we write $\psi_t^\dagger := \psi(\tau(t,\cdot),\cdot)$
%>>

\begin{lemma}\label{lem:finv}
Let $\phi,\psi\in \sob(\sDS)$
and let  $\tau\times h\in \Hom(\phi,\psi)$. Then, for all $U\subset Y$  we have
\[
\phi_t(h^{-1}(U))\subset h^{-1}(\psi^\dagger_{t}(U))\quad \forall t\ge 0.
\]
In particular, 
\begin{equation}
\label{pullbackatt2}
\omega_\phi(h^{-1}(U))=\omega_\phi(h^{-1}(\omega_\psi(U)))\subset h^{-1}(\omega_\psi(U)).
\end{equation}
\end{lemma}
\proof
See Appendix~\ref{sec:app1}.
\eproof

Now suppose we have $\tau \times h\in\Hom(\phi,\psi)$ and $U\in \sANbhd(\psi)$. Then, by Lemma \ref{lem:finv}, 
$$
\omega_\phi(h^{-1}(U)) \subset h^{-1}(\omega_\psi(U)) \subset h^{-1}(\interior(U))  
\subset \interior(h^{-1}(U)),
$$
where the latter inclusion follows from the continuity of $f$.
Therefore $h^{-1}(U)\in \sANbhd(\phi),$ and the inverse image operator induces a well-defined map $h^{-1}\colon \sANbhd(\psi)\to \sANbhd(\phi)$. 
This map is in fact a homomorphism by the properties of inverse images, since the lattice operations on $\sANbhd(\phi)$ and $\sANbhd(\psi)$ are union and intersection, so using functor notaton, $\sANbhd(\tau\times h) = h^{-1}$. Thus, by assigning to each dynamical system its attracting neighborhood lattice and to each morphism its inverse image operator, by the properties of inverse images and Lemma~\ref{lem:finv}, we have a contravariant functor, $\sANbhd\colon \sDS \to \sBDLat$, from the category of dynamical systems  to the category of bounded, distributive lattices. 

\begin{remark} \label{rem:ablock}
\label{attextviabl}
A neighborhood $U\in\sANbhd(\phi)$ is an {\em attracting block} if
$\phi_t(\cl U)\subset\Int(U)$
for all $t>0$.
Now suppose $\tau \times h\in\Hom(\phi, \psi)$ and  $U\in \sABlock(\psi)$. Then for all $t>0$
\[
\begin{aligned}
\phi_t\bigl(\cl(h^{-1}(U))\bigr)&\subset \phi_t\bigl(h^{-1}(\cl U)\bigr) \subset  h^{-1}\bigl(\psi^\dagger_{t}(\cl U)\bigr)\\
&\subset h^{-1}(\interior(U)) \subset \interior(h^{-1}(U)),
\end{aligned}
\]
which implies that $h^{-1}(U)\in \sABlock(\phi)$ so that  we can  restrict $h^{-1}$ to $h^{-1}\colon \sABlock(\psi) \to \sABlock(\phi)$. As before, $\sABlock(\tau\times h) = h^{-1}$. This makes $\sABlock\colon \sDS \to \sBDLat$ a contravariant functor. We will primarily use attracting neighborhoods, but Remark \ref{rem:nbhdvsblock} demonstrates that restricting to attracting blocks does not change the theory.
\end{remark}

A similar construction can be used to define a contravariant functor $\sAtt\colon\sDS\to\sBDLat$, but its action on morphisms must be modified, since the inverse image of an attractor need not be an attractor.

\begin{proposition}\label{prop:finvAtt}
Suppose $\tau \times h\in\Hom(\phi, \psi)$ and $A\in \sAtt(\psi)$. Then $\omega_\phi(h^{-1}(A))\in \sAtt(\phi)$. Moreover, for $\tau \times h\in\Hom(\phi,\psi),$ the map $\omega_\phi\circ h^{-1}\colon\sAtt(\psi)\to\sAtt(\phi)$ is a lattice homomorphism.
\end{proposition}
\proof
See Appendix~\ref{sec:app1}.
\eproof

Thus, by assigning each dynamical system its attractor lattice and each morphism $\tau\times h\in\Hom(\phi,\psi)$ the operator $\sAtt(\tau\times h) = \omega_\phi\circ h^{-1}$, we have a contravariant functor $\sAtt\colon \sDS\to \sBDLat$.

\begin{remark}
\label{conjatt}
If $\tau\times h\in \Hom(\phi,\psi)$ is a conjugacy, i.e. $h\colon X\to Y$ is a homeomorphism, then 
also $\tau^{-1}\times h^{-1}\in \Hom(\psi,\phi)$ is a conjugacy, where $\tau^{-1}(s,y)$ is defined by 
$s=\tau(t,h^{-1}(y))$. As a consequence, $A\in \sAtt(\phi)$ if and only if $h(A)\in\sAtt(\psi)$, cf.\ Appendix \ref{sec:app1}.
\end{remark}

\begin{remark}\label{remark:alpha} 
Similar statements as in Lemma~\ref{lem:finv} also hold for $\alpha$-limit sets, as defined in \cite{KMV-0}. %, see Appendix~\ref{sec:app1}. 
Therefore we can define
functors $\sRNbhd,\sRBlock\colon \sDS \to \sBDLat$ for repelling neighborhoods and repelling blocks analogously. 
As for $\sAtt$, one builds $\sRep\colon \sDS \to \sBDLat$ by replacing $\omega$ with $\alpha$. The details for these constructions are in Appendix \ref{repels12}.
\end{remark}

\begin{remark}
In some situations it is  useful to consider attracting neighborhoods in the algebra of regular closed sets using analogous constructions, cf.\ \cite{KMV-1b}.
\end{remark}

\begin{remark}
\label{isolnbhd}
In the spirit of Montgomery we one can also consider the semilattice of isolating neighborhoods $\sINbhd(\phi)$ defined by the property $\phi_t(\cl U) \subset \Int U$, cf.\ \cite{Mont}. An isolated invariant set is obtained as the maximal invariant set of an isolating neighborhood: $S=\Inv_\phi(U)$. The semilattice of isolated invariant sets is denoted by $\sIsol(\phi)$, with $S\wedge S' = \Inv_\phi(S\cap S')$. As for attractors $\Inv_\phi\colon \sINbhd(\phi)\to \sIsol(\phi)$ is a semi-lattice homomorphism and $\sINbhd$ and $\sIsol$ may be regarded as functors.
\end{remark}

% The objective in this paper is to regard the continuation of dynamical features  in terms of sheaves over $\sDS(\T,X)$. Formulating continuation in terms of \etale spaces was carried out by Montgomery~\cite{Mont} for isolated invariant sets. 
% There are a variety of dynamical features of interest: attractors and repellers as well as their neighborhoods, and also Morse sets and Morse representations, to name a few. The categories of algebraic structures corresponding to these features are also varied. To keep the underlying theory flexible, and so as not to repeat theoretical arguments, we first introduce the underlying concepts and theorems abstractly, and then apply this general theory in specific contexts.
%||
Given this functorial description of dynamical structures, we now turn to the primary focus of this paper, representing continuation of dynamical features in terms of sheaves over $\sDS(\T,X)$.  To keep the underlying theory flexible, and so as not to repeat theoretical arguments, we first introduce the underlying concepts and theorems abstractly, and then apply this general theory in specific contexts.
%>>

\section{Abstract continuation}
\label{sec:ShStr}
%\corrl BK 12-28-31 , Ok RC -12-1-22, Agreed AD 1-12-22
%<<
%Recall from the introduction that a fundamental feature of Conley theory is that an isolated invariant set continues under perturbation of a dynamical system. 
%In this section, we provide an abstract framework to expand continuation to algebraic structures of certain isolated invariant sets such as attractors and repellers. 
%||
Recall from the introduction that a fundamental feature of Conley theory is that an isolated invariant set continues under perturbation of a dynamical system, which leads to the concept of continuation of isolated invariant sets.
In this section, we provide an abstract framework to expand the continuation property to algebraic structures
of dynamics.
%of certain isolated invariant sets such as attractors and repellers. 
%>>

\subsection{$\cC$-structures and categories of elements}
\label{Cstruc}
Let $\cD$ be a category such that $\sob(\cD)$ forms a topological space, 
and let $\cC$ be a \emph{concrete} category, i.e.  there exists a faithful functor, $\cC \to \sSetCat$, into the category of sets.
In applications $\cD$ is a category  of dynamical systems equipped with a topology on $\sob(\cD)$,
such as $\sDS(\T,X),$ and $\cC$ is the category characterizing the algebraic structure of the dynamical feature to be continued, for example bounded, distributive lattices. 

A $\cC$-valued contravariant functor on $\cD$ is referred to as a $\cC$\textit{-structure} on $\cD$. Let $\sE,\sG\colon \cD\to \cC$ be $\cC$-structures and let $\sw\colon \sE \Rightarrow \sG$ be a natural transformation. 
For  objects $\phi\in \sob(\cD)$ the functors $\sE$ and $\sG$ yield objects $\sE(\phi)$ and $\sG(\phi)$ in $\cC$ and the component $\sw_\phi$ of the natural transformation yields a morphism $\sw_\phi\colon \sE(\phi) \to \sG(\phi)$.

In applications, typically the functor $\sG$ represents a dynamical feature such as attractors, and the functor $\sE$ denotes a corresponding neighborhood feature such as attracting neighborhoods or attracting blocks.

Furthermore, we assume the existence of a constant, contravariant functor $\sPow\colon \cD\to \cC$, $\phi\xmapsto{\sPow} \sPow_0$, referred to as the \emph{universe functor}, for which there exists an injective natural transformation
$\biota\colon \sE\Rightarrow \sPow$. \label{pg:universefunctor} In dynamics applications, when $\cD= \sDS(\T,X)$, the universe functor assigns to each $\phi$ a fixed subalgebra of the Boolean algebra $\sSet(X)$, the power set of the phase space, or a fixed subalgebra of the Boolean algebra $\scrR(X),$ the regular closed subsets of $X$. For example, if $\sE=\sANbhd$, the lattice of attracting neighborhoods, $\sE(\phi),$ is a sublattice of $\sSet(X)$.
%\corrl remind me -- what is an example where $\sF$ is only locally constant? -- BK 12-28-31 Make constant functor. Agree RC 12-1-22  Agreed AD 1-12-22
%<<
%For larger subcategories the functor $\sPow$ is locally constant.
%For the remainder of this section we assume $\sPow(\phi) = \sPow$ to be constant.
%||
%For larger subcategories the functor $\sPow$ can be merely locally constant. However,
%for the remainder of this section it suffices to assume $\sPow(\phi) = \sPow$ is constant.
%>>

Now we have a span of functors and natural transformations $\sPow \xLeftarrow{\biota} \sE \xRightarrow{\sw} \sG$ which are summarized in the following diagrams: %(compare to \eqref{eqn:cfnt}):
\begin{equation}
    \label{funcdiagranms}
\begin{tikzcd}[column sep=huge]
\cD
  \arrow[bend left=50]{r}[name=U,label=above:$\sE$]{}
  \arrow[bend right=50]{r}[name=D,label=below:$\sPow$]{} &
\cC
  \arrow[shorten <=5pt,shorten >=4pt,Rightarrow,to path={(U) -- node[label=left:$\biota$] {} (D)}]{}
\end{tikzcd}
\qquad\qquad
\begin{tikzcd}[column sep=huge]
\cD
  \arrow[bend left=50]{r}[name=U,label=above:$\sE$]{}
  \arrow[bend right=50]{r}[name=D,label=below:$\sG$]{} &
\cC
  \arrow[shorten <=5pt,shorten >=4pt,Rightarrow,to path={(U) -- node[label=left:$\sw$] {} (D)}]{}
\end{tikzcd}
\end{equation}

Since $\cC$ is a concrete category, we may regard a functor $\sE$ into $\cC$ as a $\sSetCat$-valued functor, and thus consider $\mathrm{\Pi}[\sE]$, its \emph{category of elements}, cf.\ \cite{MacLane,MacLaneMoerdijk}. \label{pg:catelements} 
The category of elements construction is used in the next section to generate an \etale space. To define the category of elements $\mathrm{\Pi}[\sE]$, let $\sob(\mathrm{\Pi}[\sE])$ be the set of all pairs $(\phi,U)$ such that $\phi\in \sob(\cD)$ and $U\in \sE(\phi)$. 
The morphisms of $\mathrm{\Pi}[\sE]$ are maps $(\phi,U) \to (\phi',U')$ for which
there is a $\cD$-morphism $h\colon \phi \to \phi'$ with 
$\sE(h)(U') = U$. 
The projection $(\phi,U) \mapsto \phi$ defines a canonical projection functor  
\begin{equation*}
    \bpi\colon \mathrm{\Pi}[\sE] \to \cD.
\end{equation*}
Moreover, given a natural transformation between functors, $\sw\colon \sE \Rightarrow \sG,$ we have the functor between the associated categories of elements
\begin{equation*}
\begin{aligned}
\mathrm{\Pi}[\sw]\colon \mathrm{\Pi}[\sE] &\to \mathrm{\Pi}[\sG] \\
 (\phi,U) &\mapsto (\phi,\sw_\phi(U)).
\end{aligned}
\end{equation*}

From the span of functors $\sPow \xLeftarrow{\biota} \sE \xRightarrow{\sw} \sG$ we obtain a span of functors on the  associated categories of elements:
\begin{equation}
\label{elcats22}
\begin{aligned}
\mathrm{\Pi}[\sPow] &\xleftarrow{\mathrm{\Pi}[\biota]} \mathrm{\Pi}[\sE] \xrightarrow{\mathrm{\Pi}[\sw]} \mathrm{\Pi}[\sG]\\
(\phi,U) &\xmapsfrom{\phantom{000\,}} (\phi,U) \xmapsto{\phantom{000\,}} (\phi,\sw_\phi(U)).
\end{aligned}
\end{equation}
Note that in  \eqref{elcats22}, the set $U\in\sE(\phi)$. To localize $\mathrm{\Pi}[\sE]$, for a fixed element $U \in \sPow_0$ 
we define the subcategory
$\mathrm{\Pi}[\sE;U]$ via
\[
\sob \bigl(\mathrm{\Pi}[\sE;U] \bigr) := \Bigl\{(\phi,U) \in \sob(\mathrm{\Pi}[\sE])~|~ U\in \sE(\phi)\Bigr\}
\]
with morphisms $(\phi,U) \to (\phi',U)$ for which
there is a $\cD$-morphism $h\colon \phi \to \phi'$ with $\sE(h)(U) = U$.

Applying the projection functor $\bpi$ yields a corresponding subcategory 
$\mathrm{\Phi}[\sE;U]$ of $\cD$. \label{pg:phiopen} The objects of $\mathrm{\Phi}[\sE;U]$ are given by
$\sob\bigl(\mathrm{\Phi}[\sE;U] \bigr) = \bigl\{\phi \in \sob(\cD)~|~ U\in \sE(\phi)\bigr\}$
with morphisms $h\colon \phi \to \phi'$ with $\sE(h)(U) = U$.
This yields the following commutative diagrams:
\begin{equation}
    \begin{tikzcd}[column sep=large, row sep=huge]
\mathrm{\Pi}[\sE]\arrow[r,"{\bpi}{}"]
  &
\cD 
\\
\mathrm{\Pi}[\sE;U] \arrow[u,"\subset"] \arrow[r,"\bpi"]& \mathrm{\Phi}[\sE;U] \arrow[u,"{\subset}{}"]
\end{tikzcd}
\qquad
\begin{diagram}
\node{} \node{\mathrm{\Pi}[\sG]}\arrow{s,l}{\bpi}\\
\node{\mathrm{\Phi}[\sE;U]} \arrow{ne,l}{\mathrm{\Theta}[\sw;U]}\arrow{e,l}{\subset}\node{\cD}
\end{diagram}
\qquad
\dgARROWLENGTH=2.0em
\begin{diagram}
\node{\mathrm{\Pi}[\sE;U]}\arrow{e,l}{\mathrm{\Pi}[\sw]}\arrow{s,l}{\bpi} \node{\mathrm{\Pi}[\sG]}\\
\node{\mathrm{\Phi}[\sE;U]} \arrow{ne,r}{\mathrm{\Theta}[\sw;U]}
\end{diagram}
\end{equation}
where $\mathrm{\Theta}[\sw;U](\phi) := \bigl(\phi,\sw_\phi(U)\bigr)$ is called the \emph{partial section functor},\label{pg:partialsec} which satisfies:
\begin{equation}
    \label{comps2}
    \mathrm{\Theta}[\sw;U]\circ \bpi = \mathrm{\Pi}[\sw] \hbox{\; and\; }  \bpi \circ \mathrm{\Theta}[\sw;U] = \sId.
\end{equation}
We leave it to the reader to verify functoriality.
\begin{remark}
In settings where $\mathrm{\Phi}[\sE;U]$ is a used as a subset of $\cD$ we abuse notation and write $\phi\in \mathrm{\Phi}[\sE;U]$. The same applies to open subsets $\mathrm{\Omega} \subset \cD$, cf.\ Sect.\ \ref{sec:DSCat}.
\end{remark}

%\begin{remark}
%\label{constsecfunctor}
%The partial section functor is indeed a functor which is useful when topologizing the categories of elements, cf.\ Appendix \ref{sec:app1}.
%\end{remark}

\subsection{Continuation frames and \etale spaces}
\label{etale}

% \corrc in the above section we describe $\Pi[\cdot]$ as a category of elements and we describe the morphisms. We also describe $\Theta[\cdot;\cdot]$ as a functor. In this section, the categorical structure is dropped and replaced by a topological structure on the objects of $\Pi[\cdot]$ and continuity of the maps $\Theta[\cdot;\cdot]$. This transition is a bit jarring and we use $\Pi[\cdot]$ in place of $\sob(\Pi[\cdot])$ without comment -- BK 12-28-21
% We can make a better transition here, RC 12-1-22 -- Agreed, although we could also drop $\Pi[\cdot]$'s categorical structure. -- AD 1-26-22 <<>>

%\marginpar{!change!}
In the same way that $\sob(\cD)$ forms topological space, we will equip $\sob(\mathrm{\Pi}[\sG])$ with a topology. This is done such that the functors $\mathrm{\Theta}[\sw, U]$ become continuous maps on objects. We will abuse notation and drop the $\sob(-)$ when referring to "elements of $\mathrm{\Pi}[\sG]$." 
A $\cC$-structure $\sE\colon \cD \to \cC$ is called  \textit{stable} if $\mathrm{\Phi}[\sE;U]$ is open in $\cD$ for all elements $U\in \sPow_0$. 
Otherwise a $\cC$-structure is said to be \emph{unstable}. 
{In the remainder of the paper we will always assume that $\sE\colon \cD \to \cC$ admits a universe $\sPow_0$ for which it is stable.}

\begin{definition}
\label{Cframe}

A $\cC$\textit{-continuation frame} on $\cD$ is a triple $(\sG,\sE,\sw)$
consisting of  
$\cC$-structures $\sE,\sG\colon \cD\to \cC$  and a natural transformation
$\sw\colon \sE\Rightarrow\sG$
such that
\begin{enumerate}
        \item [(i)] $\sw_\phi\colon \sE(\phi) \twoheadrightarrow \sG(\phi)$ is surjective for all $\phi\in \sob(\cD)$;
    \item [(ii)] $\sE$ is a stable $\cC$-structure. 
    \item [(iii)] The sets $\bigl\{\phi \in \mathrm{\Phi}[\sE; U] \cap \mathrm{\Phi}[\sE; U'] : \sw_\phi(U) = \sw_\phi(U')\bigr\}$ are open for all pairs $U, U'\in \sF_0$.
\end{enumerate}
Condition (i) can be paraphrased by saying that $\sw$ is \emph{componentwise} surjective.
The $\cC$-structure $\sE$ in a continuation frame is 
is called a \emph{stable extension} of $\sG$. 
\end{definition}

Condition (iii) is crucial
%important 
for continuity of the sections $\mathrm{\Theta}[\sw; U]$.
%When the natural transformation is not surjective componentwise, we restrict the codomain to the image of $\sw$ to define a continuation frame.
In the application of $\cC$-structures in dynamics,
$\sG$ is typically unstable as the examples in the next section show. 

The next step is to
  topologize $\mathrm{\Pi}[\sG]$ with the topology generated by
  % the    subbasis 
    \[
     \scrB(\sG) :=\Bigl\{\mathrm{\Theta}[\sw;U](\mathrm{\Omega})~|~U\in \sPow_0, \,\hbox{$\mathrm{\Omega}\subset \mathrm{\Phi}[\sE;U]$ open}\Bigr\},
  \]
  where $\mathrm{\Theta}[\sw;U](\mathrm{\Omega})$ is the image under $\mathrm{\Theta}[\sw;U]$ of objects $\phi\in \mathrm{\Omega}$.

\begin{lemma} \label{lem:basis}
$\scrB(\sG)$ is a basis for a
%the 
topology on $\mathrm{\Pi}[\sG]$. The maps $\mathrm{\Theta[\sw; U]}\colon \mathrm{\Phi}[\sE; U] \to \mathrm{\Pi}[\sG]$ are all continuous.
\end{lemma}
\proof
Let $\mathrm{\Theta}[\sw;U](\mathrm{\Omega_1})$, $\mathrm{\Theta}[\sw;U'](\mathrm{\Omega_2})$ be some elements of $\scrB(\sG)$. We can write their intersection in the following way:
\[
\mathrm{\Theta}[\sw;U](\mathrm{\Omega_1}) \cap \mathrm{\Theta}[\sw;U'](\mathrm{\Omega_2}) = \mathrm{\Theta}[\sw;U](\mathrm{\Omega}),
\]
where we let $\mathrm{\Omega} = \mathrm{\Omega}_1 \cap \mathrm{\Omega}_2 \cap \bigl\{\phi \in \mathrm{\Phi}[\sE; U] \cap \mathrm{\Phi}[\sE; U'] : \sw_\phi(U) = \sw_\phi(U')\bigr\}$. For any given $\mathrm{\Theta}[\sw; U]$ and basis element $\mathrm{\Theta}[\sw; U'](\mathrm{\Omega})$, one has
\[
\mathrm{\Theta}[\sw; U]^{-1}\big( \mathrm{\Theta}[\sw; U'](\mathrm{\Omega}) \big) = \mathrm{\Omega} \cap \bigl\{\phi \in \mathrm{\Phi}[\sE; U] \cap \mathrm{\Phi}[\sE; U'] : \sw_\phi(U) = \sw_\phi(U')\bigr\},
\]
which is open.
\eproof

%This is the finest topology   such that all maps $\mathrm{\Theta}[\sw;U]\colon \mathrm{\Phi}[\sE;U] \to \mathrm{\Pi}[\sG]$ are continuous. 
% \corrc is the injectivity of $\Theta[\omega;U]$ used here? -- BK 12-28-21  -- Since the inclusion $i\colon \mathrm{\Phi}[\sE;U] \to \cD$ is a local homeomorphism, and $\mathrm{\Theta}[\omega;U]$ commutes with the projection and inclusion, open is equivalent to continuous. The above topology is the minimal topology where $\mathrm{\Theta}[\omega]$ is open. I don't think this uses injectivity. Is there a reference??
% Finest topology for which $\Theta$ is positive?
% Coarsest for which $\Theta$ is open.
% <<>>
The functor $\bpi\colon \mathrm{\Pi}[\sG] \to \cD$ may be regarded as a projection
 $\pi\colon \sob\bigl(\mathrm{\Pi}[\sG]\bigr) \twoheadrightarrow \sob(\cD),$
 and with the above defined topology on $\sob\bigl(\mathrm{\Pi}[\sG]\bigr)$, it is also a continuous map between topological spaces. 
In this setting we denote the objects of the category of elements by $\mathrm{\Pi}[\sG]$,
 and we show that $(\mathrm{\Pi}[\sG],\pi)$ is an \emph{\'etal\'e space} in the category $\sSetCat$ by establishing that $\pi$ is a local homeomorphism.

\begin{theorem}\label{thm:shgen}
Let $(\sG,\sE,\sw)$ be a $\cC$-continuation frame on $\cD$. Then, the pair $(\mathrm{\Pi}[\sG],\pi)$ is an \etale space on $\cD$.
\end{theorem}

\proof
To establish $(\mathrm{\Pi}[\sG],\pi)$ as an \etale space with the above defined topology on $\mathrm{\Pi}[\sG]$ we show that $\pi$ is a local homeomorphism.

Let $(\phi,S)$ be a point in $\mathrm{\Pi}[\sG]$. Then, since $\sw_\phi\colon \sE(\phi) \twoheadrightarrow \sG(\phi)$ is surjective for all $\phi$, there exists $U\in \sE(\phi)$ such that $\sw_\phi(U) = S$. Consequently, the point $(\phi,S)$ is contained in 
the image of the map
$\mathrm{\Theta}[\sw;U]\colon \mathrm{\Phi}[\sE;U]\to \mathrm{\Pi}[\sG]$, 
which is \emph{open} by the definition of the topology. The image under $\pi$ of the set 
$\Img\bigl(\mathrm{\Theta}[\sw;U]\bigr)$  
is the set $\mathrm{\Phi}[\sE;U]$ which is open by assumption. It remains to show that 
$\pi\colon \Img\big(\mathrm{\Theta}[\sw;U]\big) \to \mathrm{\Phi}[\sE;U]$ 
is a homeomorphism.

First we show bijectivity. 
By definition 
$\pi\colon \Img\big(\mathrm{\Theta}[\sw;U]\big) \to \mathrm{\Phi}[\sE;U]$ 
is onto and since $\phi\mapsto (\phi,\sw_\phi(U))$ for $\phi\in \mathrm{\Phi}[\sE;U]$ is a section by Lemma \ref{lem:basis},
we establish bijectivity. 

Second we show that 
$\pi\colon \Img\big(\mathrm{\Theta}[\sw;U]\big) \to \mathrm{\Phi}[\sE;U]$ 
is continuous and open.
Let $\mathrm{\Omega} \subset \mathrm{\Phi}[\sE;U]$ be open. Then, 
$\pi^{-1}(\mathrm{\Omega}) = \mathrm{\Theta}[\sw;U](\mathrm{\Omega})$
is open by the definition of the topology which proves the continuity of $\pi$.
Let 
$\mathrm{\Theta}[\sw;U'](\mathrm{\Omega})$ be a basic open set.
Then,
\[
\pi\big( \mathrm{\Theta}[\sw;U'](\mathrm{\Omega}) \cap \Img\bigl(\mathrm{\Theta}[\sw;U]\bigr)\big)
= \mathrm{\Omega} \cap \bigl\{\phi \in \mathrm{\Phi}[\sE; U] \cap \mathrm{\Phi}[\sE; U'] : \sw_\phi(U) = \sw_\phi(U')\bigr\}
\]
is open, and thus $\pi$ is a homeomorphism. This proves that $\mathrm{\Pi}[\sG]$ is an  \etale space in $\sSetCat$.
\eproof

In the spirit of \cite{Mont} two points $(\phi,S)$ and  $(\phi',S')$ are \emph{related by continuation} if they are contained in the same quasicomponent of $\mathrm{\Pi}[\sG]$, or equivalently $(\phi',S')$ is contained in the quasicomponent of $(\phi,S)$. Recall that a quasicomponent of $(\phi,S)$ of $\mathrm{\Pi}[\sG]$ is the intersection of all clopen subsets of $\mathrm{\Pi}[\sG]$ containing $(\phi,S)$. 
The following result characterizes this topology.

\begin{proposition} 
Let $(\sG,\sE,\sw)$ be a $\cC$-continuation frame. The topology generated by the basis $\scrB(\sG)$ is the coarsest topology such that the maps $\mathrm{\Theta}[\sw; U]\colon \mathrm{\Phi}[\sE; U] \to \mathrm{\Pi[\sG]}$ are continuous, and $\pi\colon \mathrm{\Pi}[\sG] \to \cD$ is a local homeomorphism.
\end{proposition}
\proof
Theorem \ref{thm:shgen} and Lemma \ref{lem:basis} give us that $\scrB(\sG)$ generates such a topology. Now suppose the maps $\mathrm{\Theta}[\sw; U]\colon \mathrm{\Phi}[\sE; U] \to \mathrm{\Pi[\sG]}$ are continuous, and $\pi\colon \mathrm{\Pi}[\sG] \to \cD$ is a local homeomorphism in some topology $\tau$. Let $\mathrm{\Omega}\subset \mathrm{\Phi}[\sE; U]$ be an open set in $\cD$ for some $U$. We have that $\pi \circ \mathrm{\Theta}[\sw; U] = \iota$ where $\iota$ denotes the inclusion of $\mathrm{\Omega}$ into $\cD$. Since $\iota$ and $\pi$ are both local homeomorphisms, so is $\mathrm{\Theta}[\sw; U]$, which implies that the image $\mathrm{\Theta}[\sw; U](\mathrm{\Omega})$ is  in $\tau$. Every element of $\scrB(\sG)$ is of this form; thus, $\scrB(\sG)$ is coarser than $\tau$.
\eproof

%We will study continuation in terms of sheaves as we will see in forthcoming sections.

The stable $\cC$-structure in a $\cC$-continuation frame yields a second \etale space.
Topologize $\mathrm{\Pi}[\sE]$ as follows. 
Define the embedding $\mathrm{\Theta}[\sId;U]\colon\mathrm{\Phi}[\sE;U] \to \mathrm{\Pi}[\sE]$ as the trivial section $\phi \mapsto (\phi,U)$ and define a subbasis for the topology on $\mathrm{\Pi}[\sE]$ as follows:
  \[
  \scrB(\sE) :=\Bigl\{\mathrm{\Theta}[\sId;U](\mathrm{\Omega})~|~U\in \sPow_0, \,\hbox{$\mathrm{\Omega}\subset \mathrm{\Phi}[\sE;U]$ open}\Bigr\}.
  \]
\vskip-.5cm  
\begin{corollary}\label{prop:shgen2}
Let $(\sG,\sE,\sw)$ be a $\cC$-continuation frame on $\cD$. Then, the pair $(\mathrm{\Pi}[\sE],\pi)$ is an \etale space on $\cD$.
\end{corollary}

\proof
The projection $\pi\colon \mathrm{\Pi}[\sE] \to \sob(\cD)$ given by $(\phi,U) \mapsto \phi$ is a local homeomorphism with the above defined topology, i.e.
$\pi\colon \mathrm{\Pi}[\sE;U] \to \mathrm{\Phi}[\sE;U]$ is a homeomorphism.
\eproof

The category of elements $\mathrm{\Pi}[\sPow]$ trivially gives an \etale space and makes the span of functors in \eqref{elcats22} into a span of \etale spaces. A continuous map $\mathrm{\Pi}[\sw] \colon \mathrm{\Pi}[\sE]\to \mathrm{\Pi}[\sG]$ is called an \emph{\etale morphism} if the following diagram in commutes, cf.\  \cite[Definition 3.3]{Tenni}. 
\begin{equation}\label{etaletrig}
\begin{diagram}
\node{\mathrm{\Pi}[\sE]}\arrow[2]{e,l}{\mathrm{\Pi}[\sw]}\arrow{se,r}{\pi}\node{}\node{\mathrm{\Pi}[\sG]}\arrow{sw,r}{\pi}\\
\node{}\node{\cD}\node{}
\end{diagram}
\end{equation}
Such a map is then a local homeomorphism by
\cite[Proposition 2.4.8]{Borceux},  \cite[Lemma 3.5]{Tenni}.

\begin{corollary}
\label{etalemaps2}
The map $\mathrm{\Pi}[\sw]\colon \mathrm{\Pi}[\sE] \to \mathrm{\Pi}[\sG]$ is an \etale morphism.
\end{corollary}

\proof
%We start with proving that $\mathrm{\Pi}[\sw]$ is continuous.
By definition of the topologies on $\mathrm{\Pi}[\sE]$ and $\mathrm{\Pi}[\sG]$, the inverse image under $\mathrm{\Pi}[\sw]$ of a basis element $\mathrm{\Theta}[\sw;U](\mathrm{\Omega})$, $\mathrm{\Omega}\subset \cD$ open, is open in $\mathrm{\Pi}[\sE]$, which proves that $\mathrm{\Pi}[\sw]$ is continuous. %open.
% This gives the commutative triangle
% which proves that $\mathrm{\Pi}[\sw]$ is a morphism of \etale spaces and a local homeomorphism, cf.\ \cite[Proposition 2.4.8]{Borceux}, \cite[Lemma 3.5]{Tenni}.
\eproof

% \begin{remark}
% A continuous map $\mathrm{\Pi}[\sw] \colon \mathrm{\Pi}[\sE]\to \mathrm{\Pi}[\sG]$ is called an \emph{\etale morphism} if the diagram in \eqref{etaletrig} commutes, cf.\  \cite[Definition 3.3]{Tenni}. Such a map is then a local homeomorphism
% \cite[Proposition 2.4.8]{Borceux},  \cite[Lemma 3.5]{Tenni}.
% \end{remark}
%>>

\subsection{Induced \etale  morphisms}
\label{unique}
The following lemma provides a criterion to construct an \etale space morphism from a natural transformation of structures.
\begin{lemma}
Let  $\sE, \sE'$ be stable $\cC$-structures on $\cD$ and let $\sn\colon \sE \Rightarrow \sE'$ be a natural transformation. Then, the induced functor $\mathrm{\Pi}[{\sn}]\colon \mathrm{\Pi}[\sE] \to \mathrm{\Pi}[\sE']$, defined by
\[
(\phi,U) \mapsto (\phi, {\sn}_\phi(U)),
\]
defines a morphism of  \etale spaces if and only if the sets
$$
\{\phi \in \mathrm{\Phi}[\sE;U] : {\sn}_\phi(U) = U'\},
$$
are open for every pair $U,U'\in \sPow_0$.
\end{lemma}

\proof
Suppose $\mathrm{\Pi}[{\sn}]$ is continuous, and $U,U'\in \sPow_0$. Then, 
$$
\big\{\phi \in \mathrm{\Phi}[\sE;U] ~|~ {\sn}_\phi(U) = U'\big\} = \pi\Big(\mathrm{\Pi}[{\sn}]^{-1}\big(\Img(\mathrm{\Theta}[\sId; U])\big)\cap \Img(\mathrm{\Theta}[\sId; U'])\Big ),
$$ 
which is open. Now for the converse. Let $\mathrm{\Theta}[\sId; U'](\mathrm{\Omega})$ be a subbasis element of $\mathrm{\Pi}[\sE']$. Then,
$$
\bigcup_{U\in \sPow} \Big( \mathrm{\Omega}\cap\big\{\phi \in \mathrm{\Phi}[\sE;U] ~|~ {\sn}_\phi(U) = U'\big\}\Big) = \mathrm{\Pi}[{\sn}]^{-1}\big(\mathrm{\Theta}[\sId; U'](\mathrm{\Omega})\big),
$$
which is a union of open sets, and therefore open.
\eproof

When the action of ${\sn}_\phi$ is independent of $\phi\in \sob(\cD)$, the openness condition is trivially satisfied. This condition for stable structures translates to unstable structures in the following proposition.

\begin{proposition}\label{prop:shmor}
Let  $(\sG,\sE,\sw)$ and $(\sG',\sE',\sw')$ be continuation frames on $\cD$, and
let $\sv\colon \sG \Rightarrow \sG'$ be a natural transformation. Suppose there exists a natural transformation $\widetilde{\sv}\colon \sE\Rightarrow \sE'$ %--- the lift of $\sv$ --- 
such that $\mathrm{\Pi}[\widetilde{\sv}]$ is continuous, and the following diagram commutes: 
\begin{equation}\label{eq:etalelift}
\begin{tikzcd}
\sE\arrow[r,Rightarrow, "\widetilde{\sv}"]\arrow[d, Rightarrow, "\sw"]& 
\sE'\arrow[d, Rightarrow, "\sw'"] \\ \sG\arrow[r, Rightarrow,"\sv"]& 
\sG'\end{tikzcd}
\end{equation}
Then, $\mathrm{\Pi}[\sv]$ is a morphism of \etale spaces. The lift $\widetilde{\sv}$ is called a {stable extension} of $\sv$.
\end{proposition}

\proof
We have the following maps on \etale spaces:
\[
\begin{tikzcd}
\mathrm{\Pi}[\sE]\arrow[rr,"\mathrm{\Pi}{[}\widetilde{\sv}{]}"]\arrow[d, "\mathrm{\Pi}{[}\sw{]}"]&& 
\mathrm{\Pi}[\sE']\arrow[d, "\mathrm{\Pi}{[}\sw'{]}"] \\ \mathrm{\Pi}[\sG]\arrow[rr, "\mathrm{\Pi}{[}\sv{]}"]\arrow[dr,"\pi"]&& 
\mathrm{\Pi}[\sG']\arrow[dl,"\pi"]
\\
& \cD
\end{tikzcd}
\]
By Corollary \ref{etalemaps2} $\mathrm{\Pi}[\sw]$ and $\mathrm{\Pi}[\sw']$ are \etale morphisms and by \cite[Proposition 2.4.8]{Borceux},  \cite[Lemma 3.5]{Tenni} the map $\mathrm{\Pi}[\tilde \sv]$ is an \etale morphism. Diagram chasing then  shows that $\mathrm{\Pi}[\sv]$ is also an \etale morphism by using the same results.
\eproof

\section{Continuation of attractors and Morse representations}
\label{contatt12}

In this section we establish continuation frames for attractors and for finite sublattices of attractors. The latter induces continuation of Morse representations.

% \corrc Morse representations need to described in the introduction and/or we need to say something more substantive here or ....  BK 12-29-21 
% We can do that in the intro and refer to KMVIII, RC 12-1-22 -- Agreed, AD 1-26-22 *** 4/7/22 TO RESOLVE THIS COMMENT --
% Give a four or five sentence explanation of Morse representation and its relation to Morse decomposition for people in the know of dynamics. Put this in the prelude.
% <<>>

\subsection{Attractors}\label{sec:att}
In Section \ref{sec:AttFun} we establised the functors $\sANbhd$ and $\sAtt$ acting between the category of dynamical systems and the category of bounded, distributive lattices. The topologies introduced in Section \ref{sec:DSCat} yield the following result.
\begin{lemma}
\label{stalANbhd}
$\sANbhd\colon\sDS(\T,X)\to\sBDLat$ is a stable structure.
\end{lemma}

\proof
As subset of $\sDS(\T,X)$ we define 
$\mathrm{\Phi}[\sANbhd;U] := \{\phi\in \sDS(\T,X) ~|~ U\in \sANbhd(\phi)\}$ for any subset $U\subset X$. 
The condition $U\in \sANbhd(\phi)$ is equivalent to
$\omega_\phi(U) \subset \Int U$. By \cite{GKV,KMV-1a} we have the equivalent characterization:
$U\in \sANbhd(\phi)$ if and only if there exists a time $\tau>0$ such that
\begin{equation}\label{chareqan}
\phi_t(\cl U) \subset \Int U\quad \forall t\ge \tau,
\end{equation}
which is equivalent to
\begin{equation}\label{chareqan2}
\bigcup_{t\in [\tau,2\tau]} \phi_t(\cl U)  = \phi\bigl( [\tau,2\tau]\times \cl U\bigr) \subset \Int U.
\end{equation}
Indeed, if \eqref{chareqan} is satisfied then \eqref{chareqan2} follows. On the other hand if 
\eqref{chareqan2} is satisfied then 
\[
\begin{aligned}
 \phi\bigl( [2\tau,3\tau]\times \cl U\bigr) &= 
 \bigcup_{t\in [\tau,2\tau]} \phi_{t+\tau}(\cl U)  = \phi_\tau\Bigl( \bigcup_{t\in [\tau,2\tau]} \phi_t(\cl U) \Bigr)\\
 &= \phi_\tau(\Int U) \subset \phi_\tau(\cl U) \subset \Int U.
\end{aligned}
\]
By induction $ \phi\bigl( [n\tau,(n+1)\tau]\times \cl U\bigr) \subset \Int U$ for all $n\ge 1$, which establishes \eqref{chareqan}.
Summarizing, 
\begin{equation}
    \label{chareqan3}
\phi \in \mathrm{\Phi}[\sANbhd;U]~\text{~~~ if and only~~~}~ \phi\bigl( [\tau,2\tau]\times \cl U\bigr) \subset \Int U\;\text{for some~}~\tau>0.
\end{equation}

For any $\tau>0$ define
  $K_\tau = [\tau,2\tau] \times \cl U\subset \T\times X$ which is a compact set.
Consider the basic open sets
\[
\calB\bigl(K_\tau,\Int U\bigr) := \Bigl\{\phi~|~\phi(K_\tau) = \phi\big([\tau,2\tau] \times \cl U\big) \subset \Int U\Bigr\},
\]
which are contained in the subbasis for the compact-open topology on $\sDS(\T,X)$. By \eqref{chareqan3} an element $\phi\in \mathrm{\Phi}[\sANbhd;U]$
is contained in $\calB\bigl(K_\tau,\Int U\bigr)$ for some $\tau>0$
and thus $\mathrm{\Phi}[\sANbhd;U] \subset \bigcup_{\tau>0} \calB\bigl(K_\tau,\Int U\bigr)$.
On the other hand if $\phi \in \calB\bigl(K_\tau,\Int U\bigr)$ for some $\tau>0$, then \eqref{chareqan3}
implies that $\phi\in \mathrm{\Phi}[\sANbhd;U]$ which shows that $\bigcup_{\tau>0} \calB\bigl(K_\tau,\Int U\bigr)
\subset \mathrm{\Phi}[\sANbhd;U]$ and thus $\mathrm{\Phi}[\sANbhd;U] = \bigcup_{\tau>0} \calB\bigl(K_\tau,\Int U\bigr)$
which is a union of basic open set and thus open.
\eproof

The next result we prove for isolating neighborhoods, cf.\ Remark \ref{isolnbhd}, 
%sets $U\subset X$ such that $\Inv_\phi(U)\subset \interior(U)$, 
which applies to the special case of attracting neighborhoods.

\begin{lemma}
The sets $\bigl\{\phi \in \mathrm{\Phi}[\sINbhd; U] \cap \mathrm{\Phi}[\sINbhd; U'] : \Inv_\phi(U) = \Inv_\phi(U')\bigr\}$ are open in $\sDS(\T,X)$. 
\end{lemma}
\proof
The proof is identical to Montgomery's proof in \cite{Mont} for flows. Let $U_1, U_2 \in \sINbhd(\phi)$ for $\phi \in \sDS(\T,X)$. If $\Inv_\phi(U_1) = \Inv_\phi(U_2) = S$, then $\Inv_\phi(U) = S$ for $U := U_1 \cap U_2$. For any point $x\in V_i := \cl U_i \smin \interior U, i=1,2$, there exists a time $\tau>0$ such that $\phi(\tau, x) \in X \smin V_i$. By compactness, we may in fact pick $\tau>0$ such that $\phi(\tau, V_i) \subset X \smin V_i$. The set
\[
\begin{aligned}
\mathrm{\Omega} = \mathrm{\Phi}[\sINbhd; U_1] \cap \mathrm{\Phi}[\sINbhd; U_2] &\cap \{ \phi \in \sDS(\T,X) : \phi(\tau, V_1) \subset X \smin V_1 \}\\  &\cap \{ \phi \in \sDS(\T,X) : \phi(\tau, V_2) \subset X \smin V_2 \},
\end{aligned}
\]
is open in the compact-open topology. For any $\psi\in \mathrm{\Omega}$, $\Inv_\psi(U_1) = \Inv_\psi(U) = \Inv_\psi(U_2)$, so we are done.
\eproof

\begin{remark}
Similar to  attractors,  the triple $(\sINbhd,\sIsol,\Inv)$ is a $\MeetLat$-continuation frame on $\sDS(\T,X)$, cf.\ \cite{Mont}.
\end{remark}

In particular, $\Inv_\phi(U) = \omega_\phi(U)$ when $U\in\sANbhd(\phi)$ by Corollary 3.6 of \cite{KMV-1a}. Consequently, the triple $\big(\sAtt,\sANbhd,\omega \big)$ is a $\sBDLat$-continuation frame on $\sDS(\T,X)$
and $\sANbhd$ is a stable extension for $\sAtt$.
By
 Theorem \ref{thm:shgen} we have that $(\mathrm{\Pi}[\sAtt],\pi)$ is an \etale space in $\sSetCat$.
 
 \begin{remark} \label{rem:nbhdvsblock}
Stable extensions in a continuation frame are not unique. Following Remark \ref{attextviabl}, attracting blocks define attractors via 
% {\color{red} A subset $U\subset X$ is called an \emph{attracting block} if $\phi_t(\cl U) \subset \Int U$
% for all $t>0$. 
% This defines the lattice $\sABlock(\phi)$ of attracting blocks for $\phi$
 $\omega_\phi\colon \sABlock(\phi) \twoheadrightarrow \sAtt(\phi)$. As before we may regard $\sABlock\colon \sDS(\T,X) \to \sBDLat$ as a contravariant functor which is a stable extension of $\sAtt\colon \sDS(\T,X) \to \sBDLat$. 
  Using the inclusion transformation $\iota \colon \sABlock \Rightarrow \sANbhd$, 
  we obtain the following commutative diagram of transformations:
%  By using the inclusion transformation $\iota \colon \sABlock \Rightarrow \sANbhd$ to lift the identity transformation $\id \colon \sAtt \Rightarrow \sAtt$,
\[
\begin{tikzcd}
\sABlock\arrow[r,Rightarrow, "\iota"]\arrow[d, Rightarrow, "\mathrm{\omega}"]& 
\sANbhd\arrow[d, Rightarrow, "\mathrm{\omega}"] \\ \sAtt\arrow[r, Rightarrow,"\id"]& 
\sAtt\end{tikzcd}
\]
Proposition \ref{prop:shmor} obtains an isomorphism between the \etale spaces generated from the two continuation frames $(\sAtt, \sANbhd, \mathrm{\omega})$ and $(\sAtt, \sABlock, \mathrm{\omega})$. %In fact, in a category with concrete pullbacks, the \etale space is independent from the choice of stable extension. 
\end{remark}
 
%The continuation frame $(\sAtt,\sABlock,\omega)$ defines the same \etale space $\mathrm{\Pi}[\sAtt]$, cf.\ Corollary \ref{unique3}.

The functor $\sAtt$ is the structure we wish to continue, with stable extension $\omega\colon\sANbhd\Rightarrow\sAtt$. This gives
\[
\begin{tikzcd}[column sep=huge]
\sDS(\T,X)
  \arrow[bend left=50]{r}[name=U,label=above:$\sANbhd$]{}
  \arrow[bend right=50]{r}[name=D,label=below:$\sAtt$]{} &
\sBDLat
  \arrow[shorten <=10pt,shorten >=10pt,Rightarrow,to path={(U) -- node[label=left:$\omega$] {} (D)}]{}
\end{tikzcd}
\qquad
\begin{diagram}
\node{} \node{\mathrm{\Pi}[\sAtt]}\arrow{s,l}{\bpi}\\
\node{\mathrm{\Phi}[\sANbhd;U]} \arrow{ne,l}{\mathrm{\Theta}[\omega;U]}\arrow{e,l}{\subset}\node{\sDS(\T,X)}
\end{diagram}
% \quad
% \dgARROWLENGTH=2.5em
% \begin{diagram}
% \node{\mathrm{\Pi}[\sANbhd;U]}\arrow{e,l}{\mathrm{\Pi}[\sw]}\arrow{s,l}{\bpi} \node{\mathrm{\Pi}[\sG]}\\
% \node{\mathrm{\Phi}[\sANbhd;U]} \arrow{ne,r}{\mathrm{\Theta}[\sw;U]}
% \end{diagram}
\]

We have the partial section map
\[
\mathrm{\Theta}[\omega; U]\colon \mathrm{\Phi}[\sANbhd;U] \to \mathrm{\Pi}[\sAtt],
\]
which maps a dynamical system $\phi$ with attracting neighborhood $U$ to the pair $(\phi,A)$ with its associated attractor $A=\omega_\phi(U)$. Since $\omega_\phi$ is surjective, given a pair  $(\phi, A)\in \mathrm{\Pi}[\sAtt]$, there exists an attracting neighborhood $U$ such that $\mathrm{\Theta}[\omega; U](\phi) = (\phi, A)$. 

\begin{remark}\label{remark:alpha2}
Following Remark \ref{remark:alpha}, we can build a continuation frame $(\sRep, \sRNbhd, \alpha)$ for repellers, which gives us an \etale space $\mathrm{\Pi}[\sRep]$. Proposition \ref{prop:shmor} and the setup in Appendix \ref{repels12} allow us to construct an isomorphism of \etale spaces from the dual repeller operator
\[
\mathrm{\Pi}[*]\colon \mathrm{\Pi}[\sAtt] \to \mathrm{\Pi}[\sRep] \quad (\phi, A)\mapsto (\phi, A^*),
\]
by using the set complement on attracting neighborhoods as a stable extension of $*$. To view set complement and $*$ as natural transformations, one can augment the hom-set of $\sBDLat$ with anti-homomorphisms or compose with an opposite functor. This technicality appears again in \ref{Morserepcont} with $\bmu$ and $\btau$. Alas, note that so far these are $\sSetCat$-valued \etale spaces. When we introduce lattice operations in Section \ref{sec:algconstr}, this will become an anti-isomorphism of $\sBDLat$-valued \etale spaces.
\end{remark}

%\corrc this comment was in the margin: {\footnotesize Should $\sRep(f)=h^{-1}$?} --- BK 12-29-21 This needs to be checked, see on of the other documents we have, RC 12-1-22 -- No, $\sRep(f)$ is not always $h^{-1}$, AD 1-26-22 <<>>

% {\color{red}\sout{
% This implies the middle commutative diagram in Fig.\ \ref{fig:outline},
% %where we lift $*$ 
% via $\widetilde * := \,^c$.  
% This verifies commutativity in Proposition \ref{prop:shmor}, and since $\widetilde *_\phi$ is independent of $\phi$, continuity of $\mathrm{\Pi}[\widetilde *]$ is automatically satisfied.
%  Applying Proposition \ref{prop:shmor} to the dual repeller transformation yields the following important isomorphism of \etale spaces: 
% \[
% \mathrm{\Pi}[*]\colon \mathrm{\Pi}[\sAtt] \to \mathrm{\Pi}[\sRep] \quad (\phi, A)\mapsto (\phi, A^*),
% \]
% which yields the right diagram in Fig.\ \ref{fig:outline}.}

% Therefore we obtain the commutative diagram
% \[
%   \begin{tikzcd}
%     \mathrm{\Pi}[\sANbhd]\arrow[d, "{\mathrm{\Pi}[\omega]}"]\arrow[r, "{\mathrm{\Pi}[{}^c]}",leftrightarrow] & \mathrm{\Pi}[\sRNbhd] \arrow[d, "{\mathrm{\Pi}[\alpha]}"] \\ \mathrm{\Pi}[\sAtt]\arrow[r, "{\mathrm{\Pi}[{}^*]}",leftrightarrow] & \mathrm{\Pi}[\sRep]
%     \end{tikzcd}
% \]
% Since $\widetilde *_\phi$ is independent of $\phi$, the map $\mathrm{\Pi}[\widetilde *]$ is continuous.
% Applying Proposition \ref{prop:shmor} yields the following important isomorphism of \etale spaces: 
% \[
% \mathrm{\Pi}[*]\colon \mathrm{\Pi}[\sAtt] \to \mathrm{\Pi}[\sRep] \quad (\phi, A)\mapsto (\phi, A^*).
% \]
%}
\begin{remark}
Note that
the dual repeller operator $*$ is dependent on the underlying system $\phi$:
\begin{equation}
    \label{pg:dualrepeller}
    A \mapsto A^* = \{x\in X : \omega_\phi(x) \cap A =\emptyset \}.
\end{equation}
For convenience of notation, we will omit the subscript when the underlying system is understood.
\end{remark}

\subsection{Morse representations}
\label{Morserepcont}
%{\color{purple} A detailed description of Morse representations and their properties is given in \cite[Sect.\ 9]{KMV-1c}.}
Define the set $\subF\sAtt(\phi)$ consisting of all the finite sublattices $\sA\subset \sAtt(\phi)$. \label{pg:subf} Every finite sublattice is understood to contain at least the elements $\emptyset$ and $\omega_\phi(X)$.
The set of finite sublattices can be given the structure of a 
semibounded 
lattice 
\[
\sA\wedge \sA' := \sA\cap \sA',\quad \sA\vee \sA' := [\sA\cup \sA'], \quad \sA,\sA'\subset \sAtt(\phi),
\]
where $[\sA\cup \sA']$ is  the smallest sublattice containing $\sA\cup \sA'$. Note that, since the $\sAtt(\phi)$ may be infinite, there may be no maximal element in $\subF\sAtt(\phi)$, and hence $\subF\sAtt(\phi)$ may not be a bounded lattice.
The lattice $\subF\sAtt(\phi)$ has  minimal element $\{\emptyset,\omega_\phi(X)\}$.
%
% \corrc can we make these general statements with a reference about $\subF\sL$ for a general BD lattice? since we apply $\subF$ not only to $\sAtt$ but $\sANbhd$ and $\sSet$ etc -- is dense important? we don't describe the topology -- BK 12-30-21
% Yes, let's make a general rmk. Dense is needed to ensure all sublattice occur as vee of finite sublattices, RC 12-1-22 -- Where do we use that every sublattice occurs as a vee of finite sublattices? I know we've mentioned this property but I'm not sure it's used in any proofs in here, AD 1-26-22<<>>
%
%, but is a atomic and dense in the lattice of all sublattices $\sub\sAtt(\phi)$.
%
% \corrc the $\sBLat$'s neeed to be replaced by $\sLat$'s here -- BK 12-30-21 <<>>
% \corrc fixed the $\sBLat$'s -- AD 1/3/2022 <<>>
Also $\subF\sAtt(\phi)$ is not a distributive lattice in general. The assignment $\sAtt(\phi) \to \subF\sAtt(\phi)$ may be regarded as covariant functor $\subF\colon \sBDLat \to \sLat$, where $\sLat$ is the category of  lattices.
Indeed, if $\sL$ and $\sK$ are bounded, distributive lattices and $g\colon \sL\to \sK$ is a lattice homomorphism (preserves $0$ and $1$), then the inclusion of a finite sublattice $i\colon\sL'\subset \sL$ defines a finite sublattice $\sK'\subset \sK$ as the range of the composition $g\circ i$. We define the arrow 
\[
\subF(g)\colon \subF\sL \to \subF\sK,\quad \sL' \mapsto \subF(g)(\sL') :=\sK'.
\]
The composition of functors yields the the contravariant functors $\subF\circ\sANbhd$ and $\subF\circ \sAtt$ which provide the following diagrams:

% \corrc since we use $\Omega$ as open sets in $\sDS$ should we use a different symbol ? -- BK 12-30-21 
% I understand, but it is different. The bold face omega is close to omega. Boldface small omega could be an option, RC 1-17-22 -- I think I lean slightly towards the existing notation, but boldface small omega could also work, AD 1-26-22
% <<>>

\begin{equation}\label{funcdiag1}
\begin{tikzcd}[column sep=huge]
\sDS(\T,X)
  \arrow[bend left=50]{r}[name=U,label=above:$\subF\sANbhd$]{}
  \arrow[bend right=50]{r}[name=D,label=below:$\subF\sSet$]{} &
\sLat
  \arrow[shorten <=10pt,shorten >=10pt,Rightarrow,to path={(U) -- node[label=right:$\bm{\mathsf{\iota}}$] {} (D)}]{}
\end{tikzcd}
\quad\quad
\begin{tikzcd}[column sep=huge]
\sDS(\T,X)
  \arrow[bend left=50]{r}[name=U,label=above:$\subF\sANbhd$]{}
  \arrow[bend right=50]{r}[name=D,label=below:$\subF\sAtt$]{} &
\sLat
  \arrow[shorten <=10pt,shorten >=10pt,Rightarrow,to path={(U) -- node[label=right:$\bm{\mathsf{\omega}}$] {} (D)}]{}
\end{tikzcd}
\end{equation}
where $\bm{\mathsf{\iota}}$ is the natural transformation defined by inclusion. The bounded lattice $\subF\sSet(X)$ consists of finite rings of sets over $X$ in the universe for the continuation frame we construct.
Moreover,  $\bm{\mathsf{\omega}}$ is a natural transformation defined as follows. 
%
%\corrl BK 12-30-21 OK RC 12-1-22, Agreed AD  1-12-22
%<<
%Let $\sU\in \subF\sANbhd(\phi)$, then $\sU\mapsto \sA =: \bm{\mathsf{\omega}}_\phi(\sU)$, where $\sA=\{A=\omega_\phi(U)~|~U\in \sU\}$.
%The construction also yields the lattice homomorphism $\omega_\phi\colon \sU \twoheadrightarrow \sA$.
%From the definition of $\subF\sANbhd$ we obtain the following lemma.  
%||
Let $\sU\in \subF\sANbhd(\phi)$, then $\bm{\mathsf{\omega}}_\phi(\sU):=\{A=\omega_\phi(U)~|~U\in \sU\}$.
This construction also yields the lattice homomorphism  $\omega_\phi\colon \sU \twoheadrightarrow \bm{\mathsf{\omega}}_\phi(\sU)$.
From the definition of $\subF\sANbhd$ we obtain the following lemma.  
%>>

\begin{lemma}
\label{stablesub}
$\subF\sANbhd\colon \sDS(\T,X)\to \sLat$ is a stable structure. The sets $\{\phi \in \sDS(\T,X) : \bm{\mathsf{\omega}}_\phi(\sU) = \bm{\mathsf{\omega}}_\phi(\sU')\}$ are open.
\end{lemma}
\proof
For any finite sublattice $\sU \in \subF\sSet(X)$ we  observe that
\[
\mathrm{\Phi}[\subF\sANbhd; \sU] = \bigcap_{U\in \sU}\mathrm{\Phi}[\sANbhd; U],
\]
which, by Lemma \ref{stalANbhd}, is open. Suppose we have two finite sublattices $\sU, \sU'\in \subF\sANbhd(\psi)$ such that $\bm{\mathsf{\omega}}_\psi(\sU) = \bm{\mathsf{\omega}}_\psi(\sU')$. Then for each $U\in \sU$, there exists a $U'\in \sU'$ such that $\omega_\psi(U) = \omega_\psi(U')$, and vice versa. Take a finite intersection of open sets 
\[
\bigcap_{\substack{U\in \sU, U'\in \sU' \\ \omega_\psi(U) = \omega_\psi(U')}}
\{\phi \in \mathrm{\Phi}[\sANbhd; U] \cap \mathrm{\Phi}[\sANbhd; U'] : \omega_\phi(U) = \omega_\phi(U')\} 
\]
\[
\subset \{\phi \in \mathrm{\Phi}[\subF\sANbhd; \sU] \cap \mathrm{\Phi}[\subF\sANbhd; \sU'] : \bm{\mathsf{\omega}}_\phi(\sU) = \bm{\mathsf{\omega}}_\phi(\sU')\},
\]
and we are done.
\eproof

By construction the natural transformation $\bm{\mathsf{\omega}}\colon \subF\sANbhd \Longrightarrow \subF\sAtt$ is surjective, which in combination with Lemma \ref{stablesub} implies the following:
\begin{lemma}
\label{contframe}
The triple $\bigl(\subF\sAtt,\subF\sANbhd,\bm{\mathsf{\omega}} \bigr)$ is a $\sLat$-continuation frame.
\end{lemma}

% From Proposition \ref{prop:shgen} we conclude that the continuation frame
%  $\bigl(\subF\sAtt,\subF\sANbhd,\mathrm{\Omega} \bigr)$ yields the sheaf $\scrS[\subF\sAtt]$.
%  Instead of considering $\scrS[\subF\sAtt]$ as the mean object of study we now discuss a structure dual to $\scrS[\subF\sAtt]$.

Following \cite{KastiKV,KMV-1c} we associate an ordered partition $\sT(\sU)$ for every finite
sublattice $\sU\in \subF\sANbhd(\phi)$: let $\sJ(\sU)$ be the poset of join-irreducible elements in $\sU$, elements $U$ with a unique predecessor $\pred U$ , and 
%\corrl 
%OK, RC 12-1-22 Agreed  AD 1-12-22
%<<
%define the map $\sJ(\sU)\ni U \mapsto U\smin \pred U =: T(U)$. The elements $T(U)$ compose the isomorphic poset  $\sT(\sU) :=\{T(U)~|~U\in \sJ(\sU)\}$
%with $T(U)\le T(U')$ if and only if $U \subseteq U'$.
%||
for $U\in\sJ(\sU)$ define $T(U)= U\smin \pred U$. The poset $\sJ(\sU)$, ordered by inclusion, induces an isomorphic poset structure on $\sT(\sU) :=\{T(U)~|~U\in \sJ(\sU)\}$
with $T(U)\le T(U')$ if and only if $U \subseteq U'$.
%>>
The poset $\sT(\sU)$ is called a \emph{Morse tessellation} for $\phi$.
We can similarly consider a finite sublattice $\sA\in \subF\sAtt(\phi)$, and let $\sJ(\sA)$ be the poset of join-irreducible elements in $\sA$.
Define a map $\sJ(\sA)\to \sInvset(\phi)$
by $A\mapsto  M(A):=A\cap (\pred A)^* = C_\sAtt(A, \pred A)$. The elements $M(A)$ compose the isomorphic poset  $\sM(\sA) :=\{M(A)~|~A\in \sJ(\sA)\}$
with $M(A)\le M(A')$ if and only if $A \subseteq A'$.
The poset $\sM(\sA)$ is called a \emph{Morse representation} for $\phi$. 

Let $\sMRepr(\phi)$ denote the set of Morse representations for a dynamical system $\phi$, and $\sMTess(\phi)$ the set of Morse tessellations. There are bijections: \label{pg:mrepr}
\[
 \btau_\phi\colon \subF\sANbhd(\phi) \to \sMTess(\phi), \quad \btau_\phi(\sU):=\sT(\sU), ~~  \btau^{-1}_\phi(\sT) :=
\bigl\{U = \bigcup I~|~I\in \sO(\sT)\bigr\}.
\]
\[
\bmu_\phi\colon \subF\sAtt(\phi) \to \sMRepr(\phi), \quad  \bmu_\phi(\sA) := \sM(\sA), ~~ \bmu^{-1}_\phi(\sM) := \bigl\{ A = \bigcup_{M\in I}W^u(M) ~|~ I\in \sO(\sT)\bigr\}.
\]
These bijections let $\sMRepr(\phi)$ and $\sMTess(\phi)$ inherit the lattice structure of their dual counterparts $\subF\sAtt(\phi)$ and $\subF\sANbhd(\phi)$ respectively:
\[
\sT\vee\sT' := \btau_\phi\bigl(\btau_\phi^{-1}(\sT) \wedge \btau_\phi^{-1}(\sT')\bigr), \quad \sT\wedge\sT' := \btau_\phi\bigl(\btau_\phi^{-1}(\sT)\vee\btau_\phi^{-1}(\sT')\bigr)
\]
\[ \sM\vee\sM' := \bmu_\phi\bigl(\bmu_\phi^{-1}(\sM)\wedge\bmu_\phi^{-1}(\sM')\bigr), \quad \sM\wedge\sM' := \bmu_\phi\bigl(\bmu_\phi^{-1}(\sM)\vee\bmu_\phi^{-1}(\sM')\bigr).
\]
As such, $\btau$ and $\bmu$ become lattice isomorphisms. We can view $\sMRepr$ and $\sMTess$ as functors assigning dynamical systems their Morse representations and Morse tesselations respectively. Define an action on morphisms using $\btau$ and $\bmu$:
\[
h\in \hom(\phi, \psi), \quad\sMTess(h) := \btau^{-1}\circ\subF\sANbhd(h)\circ\btau,\quad \sMRepr(h) := \bmu^{-1}\circ\subF\sAtt(h)\circ\bmu.
\]
$\btau$ and $\bmu$ become natural transformations in this way. The result is the following diagram of functors:

%and can be captured by the diagram of functors and natural transformations in Diagram~\eqref{funcdiags2}[]left]:
\begin{equation}\label{funcdiags2}
\begin{tikzcd}[column sep=huge]
\sDS(\T,X)
  \arrow[bend left=50]{r}[name=U,label=above:$\subF\sANbhd$]{}
  \arrow[bend right=50]{r}[name=D,label=below:$\sMTess$]{} &
\sLat
  \arrow[shorten <=10pt,shorten >=10pt,Leftrightarrow,to path={(U) -- node[label=left:$\btau$] {} (D)}]{}
  \arrow[shorten <=10pt,shorten >=10pt,Leftrightarrow,to path={(U) -- node[label=right:$\btau^{-1}$] {} (D)}]{}
\end{tikzcd}
\qquad
\begin{tikzcd}[column sep=huge]
\sDS(\T,X)
  \arrow[bend left=50]{r}[name=U,label=above:$\subF\sAtt$]{}
  \arrow[bend right=50]{r}[name=D,label=below:$\sMRepr$]{} &
\sLat
  \arrow[shorten <=10pt,shorten >=10pt,Leftrightarrow,to path={(U) -- node[label=left:$\bmu$] {} (D)}]{}
  \arrow[shorten <=10pt,shorten >=10pt,Leftrightarrow,to path={(U) -- node[label=right:$\bmu^{-1}$] {} (D)}]{}
\end{tikzcd}
\end{equation}

%The natural transformations $\btau$ and $\btau^{-1}$ are defined by
%$\btau_\phi(\sU):=\sT(\sU)$ and $\btau^{-1}_\phi(\sT) :=
%\{U = \bigcup I~|~I\in \sO(\sT)\}$.

% Similarly, for a sublattice $\sA\in \subF\sAtt(\phi)$ we define the following poset: let $\sJ(\sA)$ be the poset of join-irreducible elements in $\sA$ and define the map 
% $\sJ(\sA)\ni A \mapsto A\cap (\pred A)^* =: M(A)$.
% % $\sJ(\sA)\ni A \mapsto A-\pred A = \CF_\sAtt(A,\pred A) =: M(A)$.

 %and can be captured by   Diagram \eqref{funcdiags2}[right],\label{pg:mreprmtess}
%
% From the diagrams in \eqref{funcdiag1} and \eqref{funcdiags2} we obtain the following commutative diagrams of natural transformations:
% \begin{equation}
% \label{funcdiag3}
% \begin{tikzcd}
% [column sep = huge, row sep = huge]
% \subF\sANbhd\arrow[r, Leftrightarrow, "\btau^{-1}"', "\btau", outer sep = 2pt]\arrow[d, Rightarrow, "\mathrm{\Omega}"', outer sep = 2pt]& \sMTess \arrow[d, Rightarrow, "\bm{\mathsf{\Delta}}", outer sep = 2pt] 
% \\ \subF\sAtt \arrow[r, Leftrightarrow, "\bmu^{-1}"', "\bmu", outer sep = 2pt ]& \sMRepr 
% \end{tikzcd}
% \end{equation}

%where the natural transformations $\bmu$ and $\bmu^{-1}$ are defined by
%$\sA \mapsto \bmu_\phi(\sA):=\sM(\sA)$ and $\sM \mapsto \bmu^{-1}_\phi(\sM) :=
%\bigl\{ A = \bigcup_{M\in I}W^u(M) ~|~ I\in \sO(\sT)\bigr\}$. 

% and where $\bm{\mathsf{\Delta}}$ is defined by the natural transformation 
Let $\bm{\mathsf{\Delta}}$ be  the natural transformation defined by
\[
\bm{\mathsf{\Delta}} := \bmu\circ \bm{\mathsf{\omega}} \circ \btau^{-1}.
\]

Given a Morse tessellation $\sT$ we obtain a Morse representation via $\sT \mapsto \bm{\mathsf{\Delta}}_\phi(\sT) =: \sM$.
By the above correspondences $\sM = \bmu_\phi(\sA)$ with $\sA = \bmu_\phi^{-1}(\bm{\mathsf{\Delta}}_\phi(\sT))$ and
$\sT = \btau_\phi(\sU)$ with $\sU = \btau^{-1}_\phi(\sT)$
and $\bm{\mathsf{\omega}}_\phi(\btau^{-1}_\phi(\sT)) = \sA$.
For the elements $\sU$ and $\bm{\mathsf{\omega}}_\phi(\sU)$ we have the homomorphism $\omega_\phi\colon \sU \twoheadrightarrow \sA=\bm{\mathsf{\omega}}_\phi(\sU)$. This implies the following relation for Morse tessellations and Morse representations. Associated with $\sU$ we have $\sT=\btau_\phi(\sU)$ maps to $\sM = \bm{\mathsf{\Delta}}_\phi(\sT)$ and we obtain a canonical embedding $\iota\colon\sM \hookrightarrow \sT$,
which will be called a \emph{tessellated Morse decomposition}. The embedding $\iota$ is induced by the homomorphism $\omega_\phi$.
The continuation frame is given by the following diagrams:
\begin{equation}
    \label{dualsub12}
\begin{tikzcd}[column sep=huge]
\sDS(\T,X)
  \arrow[bend left=50]{r}[name=U,label=above:$\sMTess$]{}
  \arrow[bend right=50]{r}[name=D,label=below:$\sOrdTess$]{} &
\sLat
  \arrow[shorten <=10pt,shorten >=10pt,Rightarrow,to path={(U) -- node[label=right:$\bm{\mathsf{\iota}}$] {} (D)}]{}
\end{tikzcd}
\quad\quad
\begin{tikzcd}[column sep=huge]
\sDS(\T,X)
  \arrow[bend left=50]{r}[name=U,label=above:$\sMTess$]{}
  \arrow[bend right=50]{r}[name=D,label=below:$\sMRepr$]{} &
\sLat
  \arrow[shorten <=10pt,shorten >=10pt,Rightarrow,to path={(U) -- node[label=right:$\bm{\mathsf{\Delta}}$] {} (D)}]{}
\end{tikzcd}
\end{equation}
and establishes the continuation frame
 $\bigl(\sMRepr,\sMTess,\bm{\mathsf{\Delta}} \bigr)$. %In Section 
Here the universe is given by $\sOrdTess(X)$ which is the (complete) bounded lattice
of finite ordered tessellations of $X$. 
%{\color{red}This can be related to $\subF\sSet(X)$ as follows. Let $\sL\in \subF\sSet(X)$, then the lattice embedding $\sL \subset \sSet(X)$ is dual to pre-order given by the projection $\pi\colon (X,\le) \to
%X/_\sim\cong \sJ(\sL)$.
%The partial equivalence classes of the pre-order $(X,\le)$ define an ordered tessellation of $X$ and the lattice $\sOrdTess(X)$ is isomorphic to $\subF\sSet(X)$.}
In Section \ref{sec:shofsec} we %{\color{red} use the latter to}
define the Morse representation sheaf which generalizes  \cite{fran3}.

\begin{remark}
\label{dualityforsub}
The lattice operations on $\sMRepr(\phi)$ and $\sMTess(\phi)$ are motivated by the duality between (Priestley) pre-orders and sublattices, i.e. for a bounded distributive lattice $\sL$ there exist an anti-isomorphism to the lattice of
Priestley pre-order on the Priestley $\sfS\sL$ and the lattice $\sub~\sL$ of sublattices of $\sL$, cf.\ \cite[Thm.\ 3.7]{Schmid}, \cite[Thm.\ 2.5]{Cignoli}.
\end{remark}

% {\color{red}
% \begin{remark}
% \label{antiornot}
% Not that an anti-isomorphism in the category of lattices is a morphism in the category by interchaging the role of $\vee$ and $\wedge$. This makes for proper natural transformations in \eqref{funcdiags2}. The same applies to order-reversing maps in the category of posets.
% \end{remark}
% }

\section{Algebraic constructions}
\label{sec:algconstr}
In this section we incorporate the binary operations of lattices, groups, rings, etc. and augment the \etale spaces with these operations.

\subsection{Binary operations and lattices}
\label{ssc:bin1}
Given two \etale spaces $(\mathrm{\Pi}, \pi), (\mathrm{\Pi}', \pi')$ over a topological space. Define 
\[
\mathrm{\Pi}\bullet\mathrm{\Pi}' := \big\{(\sigma, \sigma') \in \mathrm{\Pi}\times\mathrm{\Pi}' : \pi(\sigma) = \pi'(\sigma')\big\},
\]
which is also an \etale space with the same projection map and the product topology, cf.\ \cite[Sect.\ 2.5]{Tenni}. 

\begin{proposition}\label{prop:shprod}
Suppose the category $\cC$ has concrete binary products.
Let $(\sG,\sE,\sw), (\sG',\sE',\sw')$ be $\cC$-continuation frames on $\cD$. Then, $(\sG\times\sG',\sE\times\sE',\sw\times\sw')$ is  a continuation frame and the map 
\[
g\colon \mathrm{\Pi}[\sG\times \sG']\to \mathrm{\Pi}[\sG]\bullet \mathrm{\Pi}[\sG'],\qquad \big(\phi, (S,S')\big)\mapsto \bigl((\phi, S),(\phi, S')\bigr)
\]
is a homeomorphism.
\end{proposition}

\proof
Since both $\sw$ and $\sw'$ are surjective componentwise, so is $\sw\times\sw'$. For the openness conditions:
\[
\mathrm{\Phi}[\sG\times \sG';(U,U')] = \big\{\phi\in \sob(\cD) ~|~ (U,U')\in (\sG\times \sG')(\phi)\big\}=\mathrm{\Phi}[\sG;U]\cap \mathrm{\Phi}[\sG';U'].
\]
\[
\begin{aligned}
\{\phi \in \cD : (\sw\times \sw')_\phi(U_1, U_2) = (\sw\times \sw')_\phi(U_1', U_2')\} = &\{\phi \in \cD : \sw_\phi(U_1) = \sw_\phi(U_1')\} 
\\ &\cap \{\phi \in \cD : \sw'_\phi(U_2) = \sw'_\phi(U_2')\}
\end{aligned}
\]
which is open. Bijectivity of $g$ is immediate; it remains to be shown that $g$ is continuous and open on subbasis elements. Let $U,U'\in \sPow$ and $\mathrm{\Omega},\mathrm{\Omega}'$ open in $\mathrm{\Phi}[\sE; U]$ and $\mathrm{\Phi}[\sE'; U']$ respectively. 

%\corrc why $\circ$ and not $\bullet$ here -- BK 2-3-21 <<>>

Then, 
$$
g^{-1}\Big(\mathrm{\Theta}[\sw;U](\mathrm{\Omega})\times \mathrm{\Theta}[\sw; U'](\mathrm{\Omega}')\cap \mathrm{\Pi}[\sG]\bullet \mathrm{\Pi}[\sG']\Big) =\mathrm{\Theta}\big[\sw\times\sw; (U, U')\big](\mathrm{\Omega}\cap \mathrm{\Omega}'),
$$
which is open. Similarly, letting $U, U'\in \sPow$ and $\mathrm{\Omega}\subset \mathrm{\Phi}[\sE\times\sE'; (U,U')]$ open we have: 
$$
g\Big(\mathrm{\Theta}[\sw\times\sw'; (U,U')]\Big) = \mathrm{\Theta}[\sw; U](\mathrm{\Omega})\times \mathrm{\Theta}[\sw', U'](\mathrm{\Omega}) \cap \mathrm{\Pi}[\sG]\bullet \mathrm{\Pi}[\sG']
$$
which proves  that $g$ is an open map and therefore a homeomorphism. 
\eproof

\begin{remark}
The universe functor in the product continuation frame is the product $\sPow\times \sPow'$.
% of its factors' universe functors: $(U,U')\in \sPow\times \sPow'$. 
\end{remark}

We can apply Propositions \ref{prop:shmor} and  \ref{prop:shprod} to the 
the $\sBDLat$-continuation frames $(\sAtt,\sANbhd,\omega)$ 
and $(\sRep,\sRNbhd,\alpha)$
to interpret lattice operations as morphisms of \etale spaces. This permits us to regard
$\mathrm{\Pi}[\sAtt]$ and $\mathrm{\Pi}[\sRep]$ as
 $\sBDLat$-valued.
 
For example $\wedge_\phi\colon \sAtt(\phi) \times \sAtt(\phi) \to \sAtt(\phi)$ given by $(A,A') \mapsto A\wedge A'$ forms a natural transformation \[
\wedge\colon \sAtt\times\sAtt \Rightarrow \sAtt.
\]
From Proposition \ref{prop:shmor}
$\widetilde\wedge\colon \sANbhd \times \sANbhd \Rightarrow \sANbhd$, 
given by $(U,U') \mapsto U\cap U'$ with $\omega_\phi(U) = A$ and
$\omega_\phi(U') =A'$,
acts as a lift for $\wedge$
which yields
 an \etale space morphism from $\mathrm{\Pi}[\sAtt\times\sAtt]$ to $\mathrm{\Pi}[\sAtt]$.
Combining the latter with Proposition \ref{prop:shprod} yields an \etale space morphism:
\[
\mathrm{\Pi}[\wedge]\colon \mathrm{\Pi}[\sAtt]\bullet \mathrm{\Pi}[\sAtt] \to \mathrm{\Pi}[\sAtt],\quad \bigl((\phi,A),(\phi,A')\bigr) \mapsto (\phi,A\wedge A'),
\]
which establishes $\wedge$ as a continuous binary operation on $\mathrm{\Pi}[\sAtt]$. The same can be achieved for $\vee$.
Absorption, distributivity, and associativity follows immediately from the properties of $\wedge$ and $\vee$.
It remains to show that the assignments of the neutral elements
\[
\phi \mapsto (\phi,\emptyset) \in \mathrm{\Pi}[\sAtt],\qquad \phi \mapsto \bigl(\phi,\omega_\phi(X)\bigr)\in
\mathrm{\Pi}[\sAtt],
\]
are continuous. By composing the constant sections $\mathrm{\Theta}[\sId; \emptyset], \mathrm{\Theta}[\sId; X]\colon \sDS(\T,X) \to \mathrm{\Pi}[\sANbhd]$ with the continuous map $\mathrm{\Pi}[\omega]$ we obtain the desired result.
%The first one is trivial and for the second one we choose the constant section $(\phi,X) \in \mathrm{\Pi}[\sANbhd]$.
A similar argument holds for $\sRep$. % the latter is constant, i.e. $\phi \mapsto X$ which trivially proves continuity.
Consequently, we have established $\mathrm{\Pi}[\sAtt]$ and $\mathrm{\Pi}[\sRep]$ as $\sBDLat$-valued \etale spaces. 
We later explore abelian structures and ring structures which are used in the treatment of sheaf cohomology.

\subsection{The Conley form}
\label{ssc:conley}
Recall that the Conley form assigns to two attractors $A,A'\in \sAtt(\phi)$ an associated invariant set $(A,A')\mapsto \CF_\sAtt(A,A') := A\cap A'^*$, where
$A'^*\in \sRep(\phi)$ is dual to $A'$ in the sense that $A'^* = \alpha_\phi(U^c)$ where
$U\in \sANbhd(\phi)$ with $\omega_\phi(U) = A'$. The repeller $A'^*$ is called the \emph{dual repeller} to $A'$.
The Conley form has a universal property in the sense that it is a unique extension
of set-difference for bounded, distributive lattices, cf.\ \cite{KMV-1c}. \label{pg:conleyform}

% For a dynamical system $\phi$ and subset $U\subset X$ define the invariant set  $\Inv_\phi(U) := \bigcap \{S\subset U~|~S \hbox{~is invariant}\}$.
A \emph{Morse neighborhood} is a subset $T\subset X$ given by $T=U\cap V$ with $U\in \sANbhd(\phi)$ and $V\in \sRNbhd(\phi)$. It holds that $\Inv_\phi(T) = \omega_\phi(U) \cap \alpha_\phi(V):= M$ which is called a \emph{Morse set}. \label{pg:morseset} By construction $M \subset \Int T$, cf.\ \cite{KMV-1c}.
The Morse sets are denoted by $\sMorse(\phi)$ which is a bounded, meet-semilattice with binary operation $M\wedge M' := \Inv_\phi(M\cap M')$. The Morse neighborhoods are denoted by $\sMNbhd(\phi)$ and  form a bounded, meet-semilattice with intersection as binary operation. Both $\emptyset$ and $\omega_\phi(X)$ are neutral elements.
As before $\Inv\colon \sMNbhd \Rightarrow \sMorse$ is a stable $\MeetLat$-structure, where $\cC= \MeetLat$ is the category of bounded, meet-semilattices. The triple $(\sMorse,\sMNbhd,\Inv)$ is a $\MeetLat$-continuation frame and by the general theory in Section \ref{sec:ShStr} we obtain the $\MeetLat$-\etale space $\mathrm{\Pi}[\sMorse]$ of Morse sets.

By the same token we can treat the Conley form as natural transformation 
\[
\CF_\sAtt\colon \sAtt\times\sAtt \Longrightarrow \sMorse,
\]
where the functor $\sMorse$ assigns the bounded, meet-semilattice of Morse sets to $\phi$. By Proposition \ref{prop:shprod}
 this leads to 
a continuous operation
\[
\mathrm{\Pi}[\CF_\sAtt]\colon \mathrm{\Pi}[\sAtt]\bullet \mathrm{\Pi}[\sAtt] \to \mathrm{\Pi}[\sMorse]\quad \bigl((\phi, A)(\phi, A')\bigr) \mapsto \bigl(\phi, \CF_\sAtt(A,A')\bigr).
\]
The map $\mathrm{\Pi}[\CF_\sAtt]$ will play a role in setting up the appropriate algebraic construction for sheaf cohomology.

A variation on the Conley form is the \emph{symmetric Conley form} which is defined as follows:
\[
(A,A')\mapsto \CF^*(A,A') := \CF_\sAtt(A\cup A',A\wedge A') = (A\cap A'^*) \cup (A'\cap A^*).
\]
For the symmetric Conley form we use the following notation: $(A,A') \mapsto A+A'$.
\begin{remark}
The range of the symmetric Conley  form is the same as for the standard Conley form. Indeed, if $A'\subset A$ then $\CF^*(A,A') = \CF_\sAtt(A,A')$. For any pair of attractor $A,A'$ absorption implies that $\CF_\sAtt(A,A') = 
\CF_\sAtt(A,A\wedge A')$ which shows that the Conley form can always be determined from nested pairs, in which case the standard and symmetric Conley forms coincide.
\end{remark}

% \begin{remark}
% Instead of Morse neighborhoods we can also employ Morse tiles which are defined via attracting blocks and repelling blocks, i.e. a subset $T\subset X$ given by
%  $T=U\cap V$ with $U\in \sABlock(\phi)$ and $V\in \sRBlock(\phi)$
%  is called a \emph{Morse tile} and the meet-semilattice of Morse tiles is denoted by $\sMTile(\phi)$.
% \end{remark}

\subsection{The algebra of attractors}
\label{ssc:alg-att}
In this section we take a closer look at the algebraic structure of attractors.
Algebraic structures and in particular (abelian) group structures are important for the (co)homological theory of sheaves.
Our starting point is the lattice of attractors $\sAtt(\phi)$ of a fixed dynamical system $\phi$, which is a bounded, distributive lattice. 
Before treating the lattice of attractors we first consider bounded, distributive lattices from a more general point of view.

Let $(\sL,\wedge,\vee,0,1)$ be bounded, distributive lattice. Then, by the Priestley representation theorem, $\sL$ is isomorphic to the lattice $\cO^{\rm clp}(\sfS \sL)$ of clopen downsets in the ordered topological space $\sfS \sL$, the spectrum of $\sL$, whose points are the prime ideals in $\sL$  and whose topology is the Priestley topology. The latter is a zero-dimensional, compact Hausdorff space. The Boolean algebra $\B\sL$ of clopen sets in $\sfS\sL$ is called the \emph{Booleanization}, or \emph{minimal Boolean extension} of $\sL$, and $j\colon \sL\to \B\sL$ is a lattice-embedding given by $j(a) = \{I\in \sfS\sL~|~a\not\in I\}$. %{\color{red}The Boolean algebra $\B\sL:=\sPow^{\rm clp}(\mathrm{\Sigma} \sL)$ of clopen sets in $\mathrm{\Sigma}\sL$ is called the \emph{Booleanization}, or \emph{minimal Boolean extension} of $\sL$ and $j\colon \sL\to \B\sL$ is a lattice-embedding given by $j(a) = \{I\in \mathrm{\Sigma}\sL~|~a\not\in I\}$.} 
The Priestley topology is generated by the basis consisting of elements
$j(a)\smin j(b)$, $a,b\in \sL$.
This construction is functorial; we have the \emph{Booleanization functor} $\sB \colon \sBDLat \to \sBool$.

Boolean algebras can be given the structure of a ring, i.e.\ given a Boolean algebra $(\sB,\wedge,\vee,^c,0,1)$ define
\[
a+b := (a\wedge b^c) \vee (b\wedge a^c) \quad\hbox{(symmetric difference) and }\quad
a\cdot b := a\wedge b.
\]
Then, $(\sB,+,-,0,1)$ is a commutative, idempotent ring (idempotency with respect to multiplication). One retrieves the Boolean algebra structure via $a\vee b = a+b+a\cdot b$. We can formulate this as a faithful functor $\sI \colon \sBool \to \sRing$ from the category of Boolean algebras to the category of rings.
\[
\begin{tikzcd}
\sBDLat \arrow{rr}{\sB} && \sBool \arrow{rr}{\sI} && \sRing
\end{tikzcd}
\]

Define the ring  obtained from Booleanization of $\sL$ as the \emph{(Boolean) lattice ring} of $\sL$: \label{pg:latticering}
\begin{equation}
    \label{BLring}
    \sR\sL := (\sI\circ\sB)(\sL)
\end{equation}
the composition is also denoted by $\sR := \sI\circ \sB$.
This is the natural way to give an abelian structure to a bounded distributive lattice $\sL$.
We note that $\sR\sL$ is in general not free as additive $\Z_2$-module (vector space), nor as multiplicative monoid.
Since $\sL$ may be regarded as a (commutative) monoid with respect to $\wedge$ we can use the monoid ring construction, cf. \cite{Lang,Bergman,Tarizadeh}, to define the $\Z_2$-algebra 
$\monZ\sL$, which is a free $\Z_2$-module (vector space). 
The elements of $\monZ\sL$ are finite formal sums $\sum_ia_i$, $a_i\in \sL$, with the additional requirement that $2a=a+a=0$.
Multiplication is given by $a\cdot b := a\wedge b$. We refer to $\monZ\sL$ as the \emph{lattice algebra} of $\sL$  which was first introduced in the context of minimal Boolean extension by MacNeille \cite{MacNeille}. \label{pg:latticealg}
The lattice algebra $\monZ\sL$ is clearly a Boolean ring as is the lattice ring.
The $\Z_2$-monoid ring construction defines a covariant functor
\[
\begin{tikzcd}
\sBDLat \arrow{rr}{\monZ}  && \sRing.
\end{tikzcd}
\]
The analog of the homomorphism $j\colon \sL \to \B\sL$ is now given by the ring homomorphism:
\[
j\colon \monZ\sL \to \sR\sL,\qquad j\Bigl(\sum_i a_i \Bigr) := 
\sum_i j(a_i) = \sum_i\alpha_i.
\]
By definition $j(a\wedge b) = j(a) \cap j(b)$, which makes $j$ an algebra homomorphism.
The image of the generators of $\monZ\sL$ in $\sR\sL$ are downsets in $\sfS\sL$ and via the induced $\vee$ operation the lattice $\sL$ can be retrieved.
By construction 
\[
\begin{aligned}
j(a) + j(b) &=\bigl(j(a)\cup j(b)\bigr)\smin \bigl(j(a)\cap j(b)\bigr)
= j(a\vee b)\smin j(a\wedge b)\\
&= \CF^\sigma(a\vee b, a\wedge b),
\end{aligned}
\]
When $b\subset a$, then $j(a)+j(b) = \CF^\sigma(a,b)$ and thus the sums $j(a)+j(b)$ exhaust the range of the Conley form $\CF^\sigma\colon \sL \times \sL \to \sR\sL$.
We define the set $\sC\sL := \bigl\{\CF^\sigma(a,b)~|~a,b\in \sL\bigr\}$ as the \emph{convexity monoid}:  for $\sigma,\sigma'\in \sC\sL$ we have $\sigma \cdot \sigma' = \CF^\sigma(a,b)\cap \CF^\sigma(a',b') = \CF^\sigma(a\wedge a',b\vee b')\in \sC\sL$
and $\sigma \cdot 1 = \CF^\sigma(a,b)\cap \CF^\sigma(1,0) = \CF^\sigma(a\wedge 1,b\vee 0)= \CF^\sigma(a,b) = \sigma$. Clearly the embedding
$i\colon \sC\sL \to \sR\sL$
is a monoid homomorphism.

\begin{lemma}
\label{ontoas}
The ring homomorphism $j\colon \monZ\sL \to \sR\sL$ is surjective.
\end{lemma}

\proof
Let $\gamma \in \sR\sL$, then by a property of the Priestley topology we can express
$\gamma$ as finite union of the form: $\gamma = \bigcup_i \alpha_i\smin \alpha_i'$,
with $\alpha_i = j(a_i)$, $\alpha_i' = j(a'_i)$ and $a_i,a'_i \in \sL$.
The objective is to prove that $\gamma$ is in the range of $j$.
Consider $\alpha\smin \beta \cup \gamma\smin \delta$. We may assume without loss of generality that $\beta\subset \alpha$ and $\delta\subset \gamma$. Indeed, use $\alpha\smin \beta = \alpha\smin (\alpha\cap \beta)$.
Therefore, $\alpha\smin \beta \cup \gamma\smin \delta = (\alpha+\beta) \cup (\gamma+\delta)$, and
\[
\begin{aligned}
    (\alpha+\beta) \cup (\gamma+\delta) &= \alpha+\beta+\gamma+\delta + (\alpha+\beta)\cap (\gamma+\delta)\\
    &= \alpha+\beta+\gamma+\delta + (\alpha\cap \gamma) + (\alpha\cap \delta) + (\beta\cap \gamma) + (\beta\cap \delta),
\end{aligned}
\]
which corresponds to a sum of $j$-images of elements in $\sL$.
We conclude that $\gamma = \bigcup_i \alpha_i\smin \alpha_i' = \sum_k \tilde \alpha_k = \sum_k j(\tilde a_k)$, $\tilde a_k\in \sL$, which proves that $\gamma$ is in the range of $j$.
\eproof

We now have the following short exact sequence:
\begin{equation}
    \label{shortexact1}
\begin{diagram}
\node{0}\arrow{e,l}{} \node{\ker j}\arrow{e,l,V}{\subset} \node{\monZ\sL}\arrow{e,l,A}{j}\node{\sR\sL}\arrow{e,l}{}\node{0,}
\end{diagram}
\end{equation}
and since
the kernel $\ker j$ is an ideal in $\monZ\sL$ the first isomorphism theorem for rings yields
\[
\sR\sL \cong \frac{\monZ\sL}{\ker j},
\]
where the isomorphism is given $\sum_i a_i + \ker j \mapsto \sum_i \alpha_i$.
If we regard $j$ as a module (vector space) homomorphism then 
 both $\ker j$ and $\monZ\sAtt(\phi)$ are free $\Z_2$-modules. 
The ideal $\ker j$ can be characterized as follows.
\begin{lemma}
\label{ontoas2}
$\ker j$ is the ideal freely generated by elements of the form $a\vee b + a+b+a\cdot b$.
\end{lemma}

\proof
For an element $a\vee b + a+b+a\cdot b$ we have that 
\[
\begin{aligned}
j(a\vee b + a+b+a\cdot b) &= j(a\vee b) + j(a) + j(b) + j(a\cdot b)\\
&= j(a) \cup j(b) + j(a) + j(b) + j(a) \cap j(b)\\
&= 2\bigl(j(a) \cup j(b)\bigr) =\emptyset, 
\end{aligned}
\]
which proves that finite sums of elements of the form $a\vee b + a+b+a\cdot b$
are contained in $\ker j$.
Let $j\Bigl(\sum_i a_i \Bigr) = \sum_i\alpha_i = \emptyset$, then the sum must have an even number of terms. We can rearrange the sequence to a filtration $
\alpha_1' \subset \cdots \subset \alpha'_{2m}$ such that $\sum_i\alpha_i
= \sum_i\alpha'_i = \emptyset$. Consequently $\alpha'_{2k-1}+\alpha'_{2k} = \emptyset$ for $k=1,\cdots,m$, i.e. $\alpha'_{2k-1}=\alpha'_{2k}$ for all $k$.
In order to have distinct elements mapping to $\alpha'_{2j-1}=\alpha'_{2j}$
we have 
\[
j(b_k+c_k) = \alpha'_{2k-1}=\alpha'_{2k} = j(b_k\vee c_k+b_k\cdot c_k),
\]
which proves that element in $\ker j$ is contained in the set of formal sums 
generated by terms of the form $a\vee b + a+b+a\cdot b$.
\eproof

Let us return to the lattice of attractors $\sAtt(\phi)$. Define the \emph{attractor ring} of a dynamical system $\phi$ as $\sR\sAtt(\phi):= (\sI\circ\B)\bigl(\sAtt(\phi)\bigr)$ as the Boolean lattice ring of $\sAtt(\phi)$. This is the natural way to give an abelian structure to the attractors of a dynamical system. Via the monoid ring construction  we obtain
the algebra $\monZ\sAtt(\phi)$ which 
is called the \emph{free attractor ring} over $\Z_2$.

\section{Continuation sheaves}
\label{sec:shofsec}
From an abstract continuation frame we have shown how to build an \etale space which encodes the continuation of the unstable structure of interest. This \etale space $\mathrm{\Pi}[\sG]$ connects the topology of the base space to the algebraic structure of $\sG$. To study this connection, we shift our attention to the sheaves of sections generated by the \etale spaces of continuation frames. While sheaves and \etale spaces are equivalent from a categorical viewpoint, the theory of sheaves contributes a rich algebraic toolkit to our study of continuation. Perhaps most prominent is the idea of sheaf cohomology. 

\begin{definition}\label{defn:sheafofsec}
Let  $(\sG,\sE,\sw)$ be a  continuation frame and let $(\mathrm{\Pi}[\sG],\pi)$
be the associated \etale space  over $\cD$. A \textit{section} in $\mathrm{\Pi}[\sG]$ over an open set $\mathrm{\Omega}$ in $\cD$ is a continuous map $\sigma\colon \mathrm{\Omega} \to \mathrm{\Pi}[\sG]$ such that
\[
\pi\circ \sigma = \id.
\]
The set of all sections over $\mathrm{\Omega}$ is denoted by $\scrS^\sG(\mathrm{\Omega})$. The set of global sections, $\scrS^\sG(\cD),$ is also written as $\mathrm{\Gamma}(\scrS^\sG)$.
\end{definition}

 The presheaf
\[
\scrS^\sG\colon \cO(\cD)\to \sSetCat,
\]
where $\cO(\cD)$ is the category of open sets in $\cD$, is in fact a sheaf over $\cD$ and is
 called the \emph{sheaf of sections}, 
  cf.\ \cite[Sect.\ 2.2C]{Tenni}.
  A stalk of the sheaf $\scrS^\sG$ at $\phi \in \cD$ is the object $\sG(\phi)$. 
By considering sections in $\mathrm{\Pi}[\sE]$ we obtain the sheaf of sections   $\scrS^\sE$ and
stalks in $\scrS^\sE$ are denoted by $\sE(\phi)$, cf.\ \cite[Prop.\ 3.6]{Tenni}.

\begin{remark}\label{stalks}
There are multiple equivalent ways to define stalks. The sheaf-theoretic definition is a direct limit $\scrS^\sG_\phi:=\varinjlim \scrS^\sG(U)$ over open neighborhoods $U$ of a point $\phi$. Equivalently, for an \etale space $\pi\colon\mathrm{\Pi}[\sG] \to \cD$ the stalk at $\phi$ can be defined as $\pi^{-1}(\phi)$. In our setting, we make the identification between $\pi^{-1}(\phi)$ and $\sG(\phi)$.
\end{remark}

\begin{lemma}
\label{charsec2}
Let $(\sG,\sE,\sw)$ be a continuation frame and let 
$\sigma \colon \mathrm{\Omega} \to \mathrm{\Pi}[\sG]$ be a map with property that $\pi\circ \sigma = \id$ on $\mathrm{\Omega}$ (open).
Then, $\sigma$ is a section in $\mathrm{\Pi}[\sG]$ if and only if for every $\phi\in \mathrm{\Omega}$ there exists an open neighborhood $\mathrm{\Omega}_0\subset \mathrm{\Omega}$ of $\phi$ and $U\in \sE(\phi)$, such that 
$\restr{\sigma}{\mathrm{\Omega}_0} = \restr{\mathrm{\Theta}[\sw;U]}{\mathrm{\Omega}_0}$.
\end{lemma}

\proof
This follows immediately from the definition of sheaves.
\eproof

 Sections therefore act locally like $\mathrm{\Theta}[\sw;U]$.  Following this intuition, observe that $\mathrm{\Theta}[\sw; U]$ is a section in $\mathrm{\Pi}[\sG]$ over $\mathrm{\Phi}[\sE;U]$.

The above lemma means we only need to verify that a candidate section locally agrees with $\mathrm{\Theta}[\sw;U]$ for a \textit{particular} $U\in \sE(\phi)$ for some $\phi\in \sob(\cD)$, rather than all such $U$. 
By the same token sections $\sigma\colon \mathrm{\Omega} \to \mathrm{\Pi}[\sE]$ are given locally by $\mathrm{\Theta}[\sId;U]$, i.e. $\sigma(\phi) = (\phi,U)$.

From the construction of the sheaves $\scrS^\sE$ and $\scrS^\sG$ we have the following property of the natural transformation $\sw$:
\[
\begin{diagram}
\node{\scrS^\sE(\mathrm{\Omega})}\arrow{s,l}{\rho_{\mathrm{\Omega}',\mathrm{\Omega}}}\arrow{e,l}{\sw(\mathrm{\Omega})}\node{\scrS^\sG(\mathrm{\Omega})}\arrow{s,r}{\rho_{\mathrm{\Omega}',\mathrm{\Omega}}}\\
\node{\scrS^\sE(\mathrm{\Omega}')}\arrow{e,l}{\sw(\mathrm{\Omega}')}\node{\scrS^\sG(\mathrm{\Omega}')}
\end{diagram}
\]
where $\sw(\mathrm{\Omega})\colon \scrS^\sE(\mathrm{\Omega}) \to \scrS^\sG(\mathrm{\Omega})$ is defined by $\sigma\mapsto 
\mathrm{\Pi}[\sw]\circ \sigma$, $\mathrm{\Omega}$ open, and similarly for $\mathrm{\Omega}'\subset \mathrm{\Omega}$.
The maps $\rho_{\mathrm{\Omega}',\mathrm{\Omega}}$ are the restriction maps.
The latter defines a \emph{morphism of sheaves} $\sw\colon \scrS^\sE \to \scrS^\sG$.
Since $\sw$ yields the
 stalkwise surjections $\sw_\phi\colon \sE(\phi) \twoheadrightarrow \sG(\phi)$, we say that the morphism $\sw\colon \scrS^\sE \to \scrS^\sG$ is surjective.  

% \begin{proposition}[cf.\ \cite{Hartshorne}]
% \label{Hrtslem}
%  The surjectivity of $\sw$ is equivalent to the following condition:
%  for every open set $\mathrm{\Omega}\subset\cD$ and for every section $\sigma\in \scrS^\sG(\mathrm{\Omega})$  there exists an open covering $\{\mathrm{\Omega}_i\}$ of $\mathrm{\Omega}$ and 
%  sections $ \sigma_i\in \scrS^\sE(\mathrm{\Omega}_i)$, such that $\sigma|_{\mathrm{\Omega}_i} = \sw(\sigma_i)$, for all $i$.
%  \end{proposition}
 
% As a consequence of the above characterization of the surjective morphisms $\sw\colon \scrS^\sE \twoheadrightarrow \scrS^\sG$ we have that sections are locally given by $\phi \mapsto \mathrm{\Theta}[\sw;U](\phi)$ as indicated in Lemma \ref{charsec2}.

\subsection{Attractor sheaves}\label{ssc:attsheaves}

% \corrc this section was one big long paragraph  so I broke it up -- BK 2/10/22 <<>>

% \corrl BK 2/10/22 
% <<
% In Section \ref{ssc:bin1} we constructed \etale spaces in categories of bounded, distributive lattice, semi-lattices, etc. The above construction via sheaves of sections creates sheaves with values in bounded, distributive lattices, semi-lattices, etc.
% For the relevant $\cC$-structures in this paper we obtain the following sheaves.
% The $\cC$-structure $(\sAtt,\sANbhd,\omega)$
% yields 
% the \etale morphism $\mathrm{\Pi}[\omega]\colon \mathrm{\Pi}[\sANbhd] \twoheadrightarrow \mathrm{\Pi}[\sAtt]$ and the $\sBDLat$-valued sheaves 
%  $\scrS^\sANbhd\colon \cO(\sDS(\T,X)) \to \sBDLat$ and
%  $\scrS^\sAtt\colon \cO(\sDS(\T,X)) \to \sBDLat$ 
%  the following fundamental morphism of sheaves:
% \[
% \omega\colon \scrS^\sANbhd \twoheadrightarrow \scrS^\sAtt,
% \]
% that assigns to every section $\sigma\colon\mathrm{\Omega} \to \mathrm{\Pi}[\sANbhd]$ the section
% $\mathrm{\Pi}[\omega](\sigma)\colon \mathrm{\Omega} \to \mathrm{\Pi}[\sAtt]$.
% The fundamental sheaf $\scrS^\sAtt$ is called the \emph{attractor lattice sheaf} 
% over $\sDS(\T,X)$.
% ||
In Section \ref{ssc:bin1} we constructed \etale spaces in various categories such as bounded, distributive lattices. The above sheaf of sections construction creates sheaves with values in these same categories.
For example, the $\cC$-structure $(\sAtt,\sANbhd,\omega)$
yields 
the \etale morphism $\mathrm{\Pi}[\omega]\colon \mathrm{\Pi}[\sANbhd] \twoheadrightarrow \mathrm{\Pi}[\sAtt]$ and the $\sBDLat$-valued sheaves 
 $\scrS^\sANbhd\colon \cO(\sDS(\T,X)) \to \sBDLat$ and
 $\scrS^\sAtt\colon \cO(\sDS(\T,X)) \to \sBDLat.$ 
Hence, we obtain the following morphism of sheaves
\[
\omega\colon \scrS^\sANbhd \twoheadrightarrow \scrS^\sAtt
\]
that assigns to every section $\sigma\colon\mathrm{\Omega} \to \mathrm{\Pi}[\sANbhd]$ the section
$\mathrm{\Pi}[\omega](\sigma)\colon \mathrm{\Omega} \to \mathrm{\Pi}[\sAtt]$.
The  sheaf $\scrS^\sAtt$ is called the \emph{attractor lattice sheaf} \label{pg:attlatsheaf}
over $\sDS(\T,X)$.
%>>

Similarly, we have the morphism of sheaves
\[
\alpha\colon \scrS^\sRNbhd \twoheadrightarrow \scrS^\sRep,
\]
where $\scrS^\sRep$ is the  \emph{repeller lattice sheaf}.
Duality between $\sAtt$ and $\sRep,$ as well as between $\sANbhd$ and $\sRNbhd,$ yields the following commutative diagram of sheaves:
\[
\begin{diagram}
\node{\scrS^\sANbhd}\arrow{s,l,A}{\omega}\arrow{e,l,<>}{^c}\node{\scrS^\sRNbhd}\arrow{s,r,A}{\alpha}\\
\node{\scrS^\sAtt}\arrow{e,l,<>}{^*}\node{\scrS^\sRep}
\end{diagram}
\]
The Conley form on \etale spaces in Section \ref{ssc:conley} gives rise to  the $\MeetLat$-valued sheaf $\scrS^\sMorse\colon \cO(\sDS(\T,X)) \to \MeetLat$. 

%\corrl
%<<
%For the above lattice-valued and meet lattice-valued sheaves 
%||
For $\scrS^\sAtt$ and $\scrS^\sMorse$, lattice-valued and meet lattice-valued sheaves respectively,
%>>
we need to 
 construct a suitable abelian structure in order to define their sheaf cohomology.
 In general,  let $\cC$ be a concrete category, and let $\sK\colon \cC \to \sRing$  a covariant functor. Moreover, let $\scrS\colon \cO(\cD) \to \cC$ be a $\cC$-valued sheaf over a topological category $\cD$.
For every open set $\mathrm{\Omega}$ in $\cD$ we have  the ring $\sK\scrS(\mathrm{\Omega})$ and the ring-valued presheaf
\[
\sK\scrS\colon\cO(\cD) \to \sRing.
%,\quad \mathrm{\Omega} \mapsto \sK\scrS(\mathrm{\Omega}).
\]
Via the sheafification functor $^\#\colon \sPrSh_\cC(X) \to \sSh_\cC(X)$, %cf.\ Appendix \ref{SheafTheory}, 
we then obtain the sheaf
\[
\scrK:= \bigl(\sK\scrS\bigr)^\#\colon\cO(\cD) \to \sRing.
\]
In Section \ref{ssc:alg-att} we consider two functors that take values in the category of rings: the \emph{Boolean ring functor} $\sR = \sI\circ \sB\colon \sBDLat\to \sRing$ and
 the \emph{monoid ring functor} $\Z_2 \colon \sMonoid \to \sRing$. %, where $\sMonoid$ is the category of monoids.
In the case of the sheaf of attractors we obtain the abelian sheaves:
\begin{equation}
    \label{absheaves1}
    \scrA := \bigl(\sR\scrS^\sAtt\bigr)^\#\colon \cO(\sDS(\T,X)) \to \sRing,
\end{equation}
and 
\begin{equation}
    \label{absheaves7}
   \gAtt  := \bigl(\monZ\scrS^\sAtt\bigr)^\#\colon \cO(\sDS(\T,X)) \to \sRing,
\end{equation}
which are called the \emph{attractor sheaf} and the \emph{free attractor sheaf} 
over $\sDS(\T,X)$ respectively.

%\longcorrl
%<<
%The same construction can be carried out for other sheaves mentioned above.
%The construction via the functor $\monZ\colon \sMonoid \to \sRing$ works for all of the above examples since both bounded, distributive lattices and semilattices compose subcategories of in the category of (commutative) monoids. Of particular interest is the
%\emph{free Morse sheaf}
%\[
%\gMorse := \bigl(\monZ\scrS^\sMorse\bigr)^\#\colon \cO(\sDS(\T,X)) \to \sRing.
%\]
%||
\begin{remark}\label{rmk:freeMorseshef}
Similar constructions can be applied to sheaves in other dynamical contexts. 
The construction via the functor $\monZ\colon \sMonoid \to \sRing$ works for all of the above examples since both bounded, distributive lattices and semilattices compose subcategories of the category of (commutative) monoids. Of particular interest is the
\emph{free Morse sheaf}
\[
\gMorse := \bigl(\monZ\scrS^\sMorse\bigr)^\#\colon \cO(\sDS(\T,X)) \to \sRing.
\]
\end{remark}
%>>

\begin{remark}
The short exact sequence in \eqref{shortexact1} for $\sAtt$ yields:
\label{fundseq}
\begin{equation}
    \label{shortexact2}
\begin{diagram}
\node{0}\arrow{e,l}{} \node{\ker {\mathcal{j}}}\arrow{e,l,V}{\subset} \node{\gAtt}\arrow{e,l,A}{{\mathcal{j}}}\node{\scrA}\arrow{e,l}{}\node{0,}
\end{diagram}
\end{equation}
where the stalks
\[
\scrA_\phi = \sR\sAtt(\phi)\quad \text{and} \quad \gAtt_\phi = \monZ\sAtt(\phi)
\]
are the attractor ring  at $\phi$ and the free  attractor ring over $\Z_2$ at $\phi$ respectively. We define  this to be the \emph{fundamental short exact sequence of the attractor sheaf}.
The fundamental exact sequence allows us
 to relate the sheaves $\scrA$ and $\gMorse$. The generators of $\gAtt$ and the ring structure of $\scrA$ recover the attractor lattice sheaf  $\scrS^\sAtt$.
\end{remark}

%\corrc this next sentence is not obvious from the description so far -- BK 2/10/22 <<>>

\begin{remark}
\label{diffdef1}
An alternative way to define the sheaves $\scrA$ and $\gAtt$ is a direct definition via \etale spaces. In the case of $\gAtt$ we define an \etale space $\mathrm{\Pi}[\monZ\sAtt]$ using the stable $\cC$-structure $\bigl(\monZ\sAtt,\monZ\sANbhd,\monZ(\omega) \bigr)$ via the monoid ring functor. The stability follows from the fact that stability is preserved under free sums. We obtain the \etale space $\pi\colon \mathrm{\Pi}[\monZ\sAtt]\to  \sDS(\T,X)$ and the associated sheaf of sections $\scrS^{\monZ\sAtt}$. It holds that 
$\scrS^{\monZ\sAtt} \cong \gAtt$.
For the Boolean ring functor it is more involved to prove that $\bigl(\sR\sAtt,\sR\sANbhd,\sR(\omega) \bigr)$ is a continuation frame, but $\scrS^{\sR\sAtt} \cong \scrA$.
\end{remark}

\subsection{Finite sublattice and Morse representation sheaves}
Following Lemma \ref{contframe}, we also have an $\sLat$-valued \etale space $\mathrm{\Pi}[\subF\sAtt]$, encoding the continuation of finite sublattices of attractors. As earlier, we can consider the corresponding sheaf of sections $\scrS^{\subF\sAtt}\colon \sDS(\T,X) \to \sLat$. For an open set $\mathrm{\Omega} \subset \sDS(\T,X)$, a section in $\scrS^{\subF\sAtt}(\mathrm{\Omega})$ assigns to each dynamical system $\phi \in \mathrm{\Omega}$ a finite sublattice of $\sA\subset \sAtt(\phi)$. The lattice operations on $\scrS^{\subF\sAtt}(\mathrm{\Omega})$, on stalks, send two finite sublattices to their intersection or the smallest sublattice containing both. 
This yields the following question concerning the structure of the sheaf $\scrS^{\subF\sAtt}$:
\begin{center}
\emph{Can we view sections of $\scrS^{\subF\sAtt}$ as a lattice of sections of $\scrS^\sAtt$?}
\end{center}
This is not always possible, see Example \ref{ex:mobiusmonodromy}. 

To understand  this structure one needs to be able to relate the sheaves $\scrS^{\subF\sAtt}$ and $\scrS^{\sAtt}$. Define the following \etale space on $\sDS(\T,X)$:
 \[
 \mathrm{\Pi} := \bigl\{(\phi, \sA, A) \in \mathrm{\Pi}[\subF\sAtt]\bullet\mathrm{\Pi}[\sAtt] : A \in \sA\bigr\}.
 \]
%\[
%\mathrm{\Pi} := \{(\phi, A, \sA) \in \mathrm{\Pi}[\sAtt]\bullet %\mathrm{\Pi}[\subF\sAtt] : A \in \sA\}.
%\]
The projection from $\mathrm{\Pi}[\subF\sAtt]\bullet\mathrm{\Pi}[\sAtt]$
%$\mathrm{\Pi}[\sAtt]\bullet \mathrm{\Pi}[\subF\sAtt]$
remains a surjective local homeomorphism when restricted to the subspace $\mathrm{\Pi}$. 
% We leave it as an exercise to verify that the projection from $\mathrm{\Pi}[\subF\sAtt]\bullet\mathrm{\Pi}[\sAtt]$
% %$\mathrm{\Pi}[\sAtt]\bullet \mathrm{\Pi}[\subF\sAtt]$
% remains a surjective local homeomorphism when restricted to the subspace $\mathrm{\Pi}$. 
There is a commutative diagram of restriction maps for \etale spaces:
% \[
% \begin{diagram}
% \node{}\node{\mathrm{\Pi}[\subF\sAtt]\bullet\mathrm{\Pi}[\sAtt]}\arrow{sw,r}{}\arrow{se,r}{} \node{}\\
% \node{\mathrm{\Pi}[\sAtt]}\arrow{se,r}{\pi}  \node{E} \arrow{e,r}{}  \arrow{w,r}{} \arrow{n,r}{} \arrow{s,r}{\pi} \node{\mathrm{\Pi}[\subF\sAtt]}\arrow{sw,r}{\pi}\\
% \node{}\node{\sDS(\T,X)}\node{}
% \end{diagram}
% \]
\[
\begin{diagram}
\node{}\node{\mathrm{\Pi}[\subF\sAtt]\bullet \mathrm{\Pi}[\sAtt]}\arrow{sw,r}{}\arrow{se,r}{} \node{}\\
\node{\mathrm{\Pi}[\subF\sAtt]}\arrow{se,r}{\pi}  \node{\mathrm{\Pi}} \arrow{e,r}{}  \arrow{w,r}{} \arrow{n,r}{} \arrow{s,r}{\pi} \node{\mathrm{\Pi}[\sAtt]}\arrow{sw,r}{\pi}\\
\node{}\node{\sDS(\T,X)}\node{}
\end{diagram}
\]

% \corrc this notation only shows the dependence on $\subF\sAtt$ -- BK <<>>

Denote the sheaf of sections associated to the \etale space $\mathrm{\Pi}$ by
\[
\scrE^{\subF\sAtt}\colon \cO(\sDS(\T,X)) \to \sSetCat.
\]
A section of $\scrE^{\subF\sAtt}$ traces out the continuation of a finite sublattice of attractors, as well as a specific attractor in that sublattice. From the diagram, there are two morphisms:
\[
q\colon \scrE^{\subF\sAtt} \to \scrS^\sAtt\quad\quad
r\colon \scrE^{\subF\sAtt} \to \scrS^{\subF\sAtt},
\]
which restricts sections to their attractor and finite sublattice components respectively. For an open set $\mathrm{\Omega} \subset \sDS(\T,X)$ and a section $\nu \in \scrS^{\subF\sAtt}(\mathrm{\Omega})$, we can consider the set $r_{\mathrm{\Omega}}^{-1}(\nu)$ consisting of sections of $\scrE^{\subF\sAtt}$ which agree with $\nu$ on their finite sublattice component. This set has a bounded distributive lattice structure, defined on the attractor component. Suppose $\sigma, \sigma' \in r_\mathrm{\Omega}^{-1}(\nu)$:
\[
\sigma(\phi) = (\nu(\phi), A), \quad \sigma'(\phi) = (\nu(\phi), A'), \quad \quad (\sigma\wedge\sigma')(\phi) := (\nu(\phi), A \wedge A').
\]
The meet operation is defined similarly. We can then pass this lattice through $q$, and achieve a bounded distributive lattice of sections of $\scrS^\sAtt$:
\[
q_\mathrm{\Omega} \colon r_\mathrm{\Omega}^{-1}(\nu) \subset \scrE^{\subF\sAtt}(\mathrm{\Omega}) \to \scrS^\sAtt(\mathrm{\Omega})
\]

We make the following observations:
\begin{itemize}
    \item For any section $\sigma\in q_\mathrm{\Omega}\bigl(r_\mathrm{\Omega}^{-1}(\nu)\bigr)$, with $\sigma(\phi) = (\phi,A)$, we have that $A\in \sA$, where $\nu(\phi) = (\phi, \sA)$. In other words, the value of $\sigma$  at $\phi$ is contained in the value of $\nu$ at $\phi$.
     \item Composing  the stalk restriction map $\rho_\phi\colon \scrS^\sAtt(\mathrm{\Omega}) \to \sAtt(\phi)$  yields the composite lattice homomorphism 
     %\item Composing  the stalk restriction map $\rho_\phi\colon \scrS^\sAtt(\mathrm{\Omega}) \to \scrS^\sAtt_\phi$ and  the isomorphism $\scrS^\sAtt_\phi \cong \sAtt(\phi)$ yields the composite lattice homomorphism 
    \[
    f_{\mathrm{\Omega},\phi}\colon q_\mathrm{\Omega}\bigl(r_\mathrm{\Omega}^{-1}(\nu)\bigr) \to \sA \subset \sAtt(\phi), 
    \]
    where $\nu(\phi) = (\phi, \sA)$.
    \item If $f_{\mathrm{\Omega},\phi}$ is surjective at every $\phi \in \mathrm{\Omega}$, we retrieve $\nu$ from these sections: 
    \[
    \nu(\phi) = \bigl(\phi, \{A_\sigma\}\bigr), \quad  \textnormal{where} \quad {\sigma \in q_\mathrm{\Omega}\bigl(r_\mathrm{\Omega}^{-1}(\nu)\bigr)}, \quad \sigma(\phi) = (\phi, A_{\sigma}).
    \]
\end{itemize}

\begin{proposition}\label{prop:localcomposition}
Let $\nu\in \scrS^{\subF\sAtt}(\mathrm{\Omega})$ for some open set $\mathrm{\Omega}\subset \sDS(\T,X)$ and let  $\phi \in \mathrm{\Omega}$. Then, there is an open neighborhood $\mathrm{\Omega}'$ of $\phi$ such that $f_{\mathrm{\Omega}', \phi'}$ defined by $\restr{\nu}{\mathrm{\Omega}'}$ is surjective for all $\phi'\in \mathrm{\Omega}'$.
\end{proposition}
\proof
By Lemma \ref{charsec2}  $\nu$ yields a neighborhood $\mathrm{\Omega}'$ of $\phi$ upon which $\restr{\nu}{\mathrm{\Omega}'} = \restr{\mathrm{\Theta}[\bm{\mathsf{\omega}}; \sN]}{\mathrm{\Omega}'}$ for some $\sN\in \subF\sANbhd(\phi)$. For each $U \in \sN$, we have a section $\restr{\mathrm{\Theta}[\omega; U]}{\mathrm{\Omega}'} \in \scrS^\sAtt(\mathrm{\Omega}')$. Indeed, we can define the following section in $\scrE^{\subF\sAtt}(\mathrm{\Omega}')$:
\[
\phi \mapsto \bigl(\phi, \bm{\mathsf{\omega}}_\phi(\sN), \omega_\phi(U)\bigr)
\]
which maps to $\restr{\mathrm{\Theta}[\omega; U]}{\mathrm{\Omega}'}$ under $q$, and therefore $\restr{\mathrm{\Theta}[\omega; U]}{\mathrm{\Omega}'} \in q_{\mathrm{\Omega}'}\bigl(r_{\mathrm{\Omega}'}^{-1}(\nu)\bigr)$. Let $\phi' \in \mathrm{\Omega}'$, and $A\in \sA$, where $\nu(\phi') = (\phi', \sA)$. Then $A=\omega_\psi(U)$ for some $U\in \sN$, since $\bm{\mathsf{\omega}}_\psi(\sN) = \sA$. Moreover, since $f_{\mathrm{\Omega}',\phi'}\big(\restr{\mathrm{\Theta}[\omega; U]}{\mathrm{\Omega}'}\big) = A$ for arbitrary choices of $\phi'$ and $A$, the proof is complete.
\eproof 

% \begin{proposition}\label{prop:localcomposition}
% Let $\nu\in \scrS^{\subF\sAtt}(\mathrm{\Omega})$ for some open set $\mathrm{\Omega}\subset \sDS(\T,X)$ and let  $\phi \in \mathrm{\Omega}$. Then, there is an open neighborhood $\mathrm{\Omega}'$ of $\phi$ such that $f_{\mathrm{\Omega}_0, \psi}$ is surjective for all $\psi\in \mathrm{\Omega}_0$.
% \end{proposition}
% \proof
% Applying Lemma \ref{charsec2} to $\nu$ gives us a neighborhood $\mathrm{\Omega}_0$ of $\phi$ upon which $\restr{\nu}{\mathrm{\Omega}_0} = \restr{\mathrm{\Theta}[\bm{\mathsf{\omega}}; \sK]}{\mathrm{\Omega}_0}$ for some $\sK\in \subF\sANbhd(\phi)$. For each $U \in \sK$, we have a section $\restr{\mathrm{\Theta}[\omega; U]}{\mathrm{\Omega}_0} \in \scrS^\sAtt(\mathrm{\Omega}_0)$. Indeed, we can define the following section in $\scrE^{\subF\sAtt}(\mathrm{\Omega}_0)$:
% \[
% \phi \mapsto (\phi, \bm{\mathsf{\omega}}_\phi(\sK), \omega_\phi(U))
% \]
% which maps to $\restr{\mathrm{\Theta}[\omega; U]}{\mathrm{\Omega}_0}$ under $q_{\mathrm{\Omega}_0}$. So $\restr{\mathrm{\Theta}[\omega; U]}{\mathrm{\Omega}_0} \in q_{\mathrm{\Omega}'}(r^{-1}_{\mathrm{\Omega}_0}(\nu))$. Let $\psi \in \mathrm{\Omega}_0$, and $A\in \sL$ where $\nu(\psi) = (\psi, \sL)$. Then $A=\omega_\psi(U)$ for some $U\in \sK$, since $\bm{\mathsf{\omega}}_\psi(\sK) = \sL$. Because $f_{\mathrm{\Omega}_0,\psi}\big(\restr{\mathrm{\Theta}[\omega; U]}{\mathrm{\Omega}_0}\big) = A$, and this held for arbitrary choice of $\psi$ and $A$, we are done.
% \eproof \\
%%% to here.  
Proposition \ref{prop:localcomposition} justifies that \emph{locally} a section in $\scrS^{\subF\sAtt}$ may be interpreted as a finite %bounded 
distributive lattice of sections in $\scrS^\sAtt$. We will investigate when this interpretation extends globally at a later stage. \\

Dually, we can consider the continuation frame  $\bigl(\sMRepr,\sMTess,\bm{\mathsf{\Delta}} \bigr)$, which defines the $\sLat$-valued \emph{Morse representation sheaf} $\scrS^\sMRepr$, cf.\ Sect.\ \ref{Morserepcont}. Applying Proposition \ref{prop:shmor} to the natural transformation $\bmu\colon \subF\sAtt \to \sMRepr$ with stable extension $\btau \colon \subF\sANbhd \to \sMTess$ yields a sheaf isomorphism
\[
\bmu\colon \scrS^{\subF\sAtt} \to \scrS^\sMRepr.
\]
The lattice structure of $\scrS^\sMRepr$ allows common coarsings and refinements of Morse representations:
let $\sigma_\sM,\sigma_\sM'\in \scrS^\sMRepr(\mathrm{\Omega})$, then
\[
\sigma_\sM \vee\sigma_\sM' \in \scrS^\sMRepr(\mathrm{\Omega})\quad\hbox{and}\quad 
\sigma_\sM \wedge\sigma_\sM' \in \scrS^\sMRepr(\mathrm{\Omega}),
\]
% let $\sigma\in \scrS^\sMRepr(\mathrm{\Omega})$ and 
% $\sigma'\in \scrS^\sMRepr(\mathrm{\Omega}')$, then
% \[
% \sigma \vee\sigma' \in \scrS^\sMRepr(\mathrm{\Omega}\cup \mathrm{\Omega}')\quad\hbox{and}\quad 
% \sigma \wedge\sigma' \in \scrS^\sMRepr(\mathrm{\Omega}\cap \mathrm{\Omega}'),
% \]
the \emph{common coarsening} and \emph{common refinement} of Morse representations continuations respectively. 
The binary operations are defined in the sheaf $\scrS^{\subF\sAtt}$ via 
\[
\bmu^{-1}(\sigma_\sM)\wedge \bmu^{-1}(\sigma_\sM')\quad \hbox{and} \quad \bmu^{-1}(\sigma_\sM)\vee \bmu^{-1}(\sigma_\sM'),
\]
respectively.
We can dualize the earlier theory for $\scrS^{\subF\sAtt}$ to describe sections of $\scrS^\sMRepr$. For a section $\zeta \in \sMRepr(\mathrm{\Omega})$, there is a corresponding section $\nu:=\bmu(\zeta)\in \scrS^{\subF\sAtt}(\mathrm{\Omega})$. We again have $f_{\mathrm{\Omega},\phi}\colon q_\mathrm{\Omega}(r^{-1}_\mathrm{\Omega}(\nu)) \to \sA$ where $\nu(\phi) = (\phi, \sA)$. Suppose $q_\mathrm{\Omega}(r^{-1}_\mathrm{\Omega}(\nu))$ is finite. We can dualize to achieve:
\[
g_{\mathrm{\Omega},\phi}\colon \sM(\sA) \to \sP_{\mathrm{\Omega}}, %\mathrm{\Sigma}(q_\mathrm{\Omega}(r^{-1}_\mathrm{\Omega}(\nu))):= \sP_{\mathrm{\Omega}},
\]
where $\sP_{\mathrm{\Omega}}$ denotes the poset of join-irreducible elements of $q_\mathrm{\Omega}(r^{-1}_\mathrm{\Omega}(\nu))$. The map $g_{\mathrm{\Omega}, \phi}$ composes the isomorphism between $\sJ(\sA)$, the join-irreducible elements of $\sA$, and $\sM(\sA)$ with the dual of $f_{\mathrm{\Omega},\phi}$. The Morse representation $\sM(\sA)$ is exactly the value of $\zeta$ at $\phi$, in other words, $\zeta(\phi) = (\phi, \sM(\sA))$. If the lattice morphism $f_{\mathrm{\Omega}, \phi}$ is surjective, then $g_{\mathrm{\Omega}, \phi}$ is an embedding and thus a \emph{Morse decomposition}, cf.\ \cite[Def.\ 7]{KMV-1c}.
%injective. 
Thus we get an analogous statement to Proposition \ref{prop:localcomposition}.

%\corrc Morse decomposition has not really been defined --- does this MD have a natural tesselation from this construction? -- tesselated MD's have been defined -- BK --
%4/7/22 *** TO RESOLVE THIS COMMENT
%We will take care of this. See previous remarks
%<<>>

\begin{corollary}
Let $\zeta\in \scrS^{\sMRepr}(\mathrm{\Omega})$ for an open $\mathrm{\Omega}\subset \sDS(\T,X)$ such that $q_\mathrm{\Omega}(r^{-1}_\mathrm{\Omega}(\bmu_\mathrm{\Omega}(\zeta)))$ is finite, and $\phi \in \mathrm{\Omega}$. Then there is an open neighborhood $\mathrm{\Omega}'$ of $\phi$ such that $g_{\mathrm{\Omega}', \phi'}$ is a Morse decomposition for all  $\phi'\in \mathrm{\Omega'}$.
\end{corollary} \noindent

% {\color{red}
% The continuation frame  $\bigl(\sMRepr,\sMTess,\bm{\mathsf{\Delta}} \bigr)$ defines the $\sLat$-valued \emph{Morse representation sheaf} $\scrS^\sMRepr$. Let $\sigma \in \scrS^\sMRepr(\mathrm{\Omega})$ be a section over $\mathrm{\Omega}$. This $\sigma$ represents an abstract finite poset $(\sP,\le)$. For $\phi\in \mathrm{\Omega}$ this yields the maps $i\colon \sigma(\phi) \hookrightarrow \sP$ and the order-retraction $\pi\colon \sP \twoheadrightarrow\sigma(\phi)$. The former is called a \emph{Morse decomposition} at $\phi$. We say that $\sM_\phi:=\sigma(\phi)$
% is related by continuation to $\sM_{\phi'}:=\sigma(\phi')$
% for any two $\phi,\phi'\in \mathrm{\Omega}$. The partial order
% $\sP$ continues across $\mathrm{\Omega}$.
% }

\section{Parameter spaces and pullbacks}
\label{sec:bifs}
In this section we discuss continuation frames for parametrized families of dynamical systems
and how the associated sheaves can be constructed.

\subsection{Parametrized dynamical systems}
\label{ssc:paramsys}

Let $\mathrm{\Lambda}$ be a topological space. 
In keeping with the spirit of the paper we keep the conditions mild but in practical situations $\mathrm\Lambda$ is a CW-space.
%We emphasize that no requirements on the topology are made at this point.
\begin{definition}
\label{paradyn1}
Let $X$ be a compact topological space.
A \emph{parametrized dynamical system over $\mathrm{\Lambda}$} on $X$ is a continuous map
$\bphi\colon \T\times X\times \mathrm{\Lambda} \to X$ such that 
$\phi^\lambda := \bphi(\cdot,\cdot,\lambda) \in \sDS(\T,X)$ for all $\lambda\in \mathrm{\Lambda}$.
\end{definition}

The category of dynamical systems $\sDS(\T,X)$ is a function space equipped with the compact-open topology. 
For a parametrized dynamical system $\bphi$ we define the \emph{transpose}  $\obphi\colon\mathrm{\Lambda} \to \sDS(\T,X)$ by
\[
\obphi(\lambda) = \phi^\lambda := \bphi(\cdot,\cdot,\lambda).
\]
The transpose $\obphi\colon\mathrm{\Lambda} \to \sDS(\T,X)$ is a continuous map without additional assumptions on the topological spaces $\mathrm{\Lambda}$ and $X$, cf.\ Appendix \ref{cot}.

 For  the continuation frame $(\sAtt,\sANbhd,\omega)$ on $\sDS(\T,X)$ a parametrized dynamical system yields a pullback \etale space on $\mathrm{\Lambda}$: 
 \[
 \iobphi\mathrm{\Pi}[\sAtt] := \bigl\{(\lambda,\phi,A)\in \mathrm{\Lambda} \times \mathrm{\Pi}[\sAtt] ~|~ \obphi(\lambda) = \pi(\phi,A)=\phi\bigr\},
 \]

% \corrc
% In the appendix we have superscript. Shall we change the notation in the appendix to subscript?
% <<>>
 %\corrc $\pi(\phi,A)=\phi$ <<>>
 
 i.e. the follows diagram commutes
 \[
\begin{tikzcd}
 \iobphi\mathrm{\Pi}[\sAtt] \arrow{rrr}{(\lambda,\phi,A)\mapsto (\phi,A)} \arrow{d}{(\lambda,\phi,A) \mapsto \lambda} &&& \mathrm{\Pi}[\sAtt] \arrow{d}{\pi} \\ \mathrm{\Lambda} \arrow{rrr}{\obphi} &&& \sDS(\T,X)
\end{tikzcd}
 \]
 where $\iobphi\mathrm{\Pi}[\sAtt]$ is the pullback in the category of topological spaces, cf.\ \cite[Sect.\ I.3]{Bredon}. From \cite[Prop.\ 2.4.9]{Borceux} it follow that $\iobphi\mathrm{\Pi}[\sAtt] \to \mathrm{\Lambda}$ is an \etale space over $\mathrm{\Lambda}$.
The binary operations on $\mathrm{\Pi}[\sAtt]$ can be verified to be continuous on the inverse image \etale space.
As before we obtain the following $\sBDLat$-valued pullback sheaf
\[
\iobphi\scrS^\sAtt\colon \cO(\mathrm{\Lambda}) \to \sBDLat,
\]
as the sheaf of sections of $\iobphi\mathrm{\Pi}[\sAtt]$.
Applying the boolean ring functor $\sR$ to the sheaf of sections yields a ring valued sheaf:
\[
\scrA^{\obphi} 
:= \bigl(\sR\iobphi\scrS^{\sAtt}\bigr)^\# \colon \cO(\mathrm{\Lambda}) \to \sRing.
\]\label{pg:attsheaf}
The \emph{ringed space} $(\mathrm{\Lambda},\scrA^{\obphi})$ encodes the continuation data of attractors for the parametrized dynamical system. Similarly, for the Monoid ring functor $\Z_2\cdot $ we obtain:
\[
\gAtt^{\obphi} 
:= \bigl(\Z_2\iobphi\scrS^{\sAtt}\bigr)^\# \colon \cO(\mathrm{\Lambda}) \to \sRing.
\]
where the multiplication is inherited from the monoidal structure of $\sAtt$, cf. Section \ref{ssc:alg-att}. We are now in the setting of sheaf cohomology. Since the category of sheaves of abelian groups has enough injectives, the \emph{$i$th sheaf cohomology groups} may be defined as the right derived functors of the global section functor. A more direct and detailed construction can be found in \cite{Bredon}. We apply these, and their relative versions, to the sheaves $\scrA^\obphi$ and $\gAtt^\obphi$. Theorem \ref{thm:CIT} and the later sections will show the cohomology groups $H^i(\mathrm{\Lambda}, \scrA^\obphi)$ and $H^i(\mathrm{\Lambda}, \mathrm{\Lambda}'; \scrA^\obphi)$ are algebraic invariants which can detect bifurcations.
\begin{remark}
\label{diffdef2}
The sheaf $\scrA^{\obphi}$ can be alternatively defined as follows. The Boolean ring functor yields the \etale space $\mathrm{\Pi}[\sR\sAtt]$ and the associated sheaf of sections \[
\iobphi\scrS^{\sR\sAtt} \cong \iobphi\scrA \cong \scrA^\obphi,
\]
which defines the sheaf as the pullback sheaf with respect to $\obphi$.
\end{remark}

\subsection{Conjugate dynamical systems and homeomorphic \etale spaces}
\label{ssc:conjinvariance}
We start off with the basic notion of conjugacy in dynamical systems.
\begin{definition}
\label{paradyn2}
Let $X$ and $Y$ be compact topological spaces, and 
let $\obphi\colon \mathrm{\Lambda} \to \sDS(\T,X)$ and $\obpsi\colon \mathrm{\Lambda} \to \sDS(\T,Y)$ be parametrized dynamical systems.
A \emph{conjugacy} between $\obphi$ and $\obpsi$ is a continuous map $h\colon \mathrm{\Lambda}\times X \to Y$ and a continuous reparametrization $\tau\colon \mathrm{\Lambda} \times \T \times X \to \T$, such that 
\begin{enumerate}
    \item [(i)] $h^\lambda\times\tau^\lambda:= h(\lambda,\cdot)\times \tau(\lambda, \cdot, \cdot)$
is a conjugacy in $\Hom(\phi^\lambda,\psi^\lambda)$ for all $\lambda\in \mathrm{\Lambda}$;
\item [(ii)] $h^\lambda(X_i) = Y_i$ uniformly for all  $\lambda\in \mathrm{\Lambda}$, where $X_i$ and $Y_i$ are the connected components of $X$ and $Y$ respectively.
\end{enumerate}
If a conjugacy $h$ exists, then $\obphi$ and $\obpsi$ are said to be \emph{conjugate parametrized dynamical systems}.
\end{definition}

\begin{remark}
\label{paradyn3}
Assumption (ii) is always satisfied pointwise for $\lambda$ by
appropriately indexing the components of $X$ and $Y$. The uniformity in the above definition is not guaranteed since no restrictions on the topology of $\mathrm{\Lambda}$ are required. For specific topologies on $\mathrm{\Lambda}$ condition (ii) may be superfluous. 
\end{remark}

\begin{remark}
One may
%can 
also consider quasiconjugacies between parametrized dynamical systems over $\mathrm{\Lambda}$. 
\end{remark}

Since $h^\lambda$ is a conjugacy we know from Remark \ref{conjatt} that
the push-forward $U^\lambda \mapsto h^\lambda(U^\lambda)$ is an attracting neighborhood for $\psi_\lambda$ and similarly, the push-forward $A^\lambda \mapsto h^\lambda(A^\lambda)$
is an attractor for $\psi_\lambda$.
\begin{lemma}
\label{commutes1}
The following diagram  commutes:
\[
\begin{diagram}
\node{\sANbhd(\phi^\lambda)}\arrow{s,l,A}{\omega_{\phi^\lambda}}\node{} \node{\sANbhd(\psi^\lambda)}\arrow{s,r,A}{\omega_{\psi^\lambda}}\arrow[2]{w,l,<>}{\cong}\\
\node{\sAtt(\phi^\lambda)} \node{}\node{\sAtt(\psi^\lambda)}\arrow[2]{w,l,<>}{\cong}
\end{diagram}
\]
\end{lemma}
\proof
Indeed, the maps from above we have $U^\lambda \mapsto h^\lambda(U^\lambda)\mapsto \omega_{\psi^\lambda}\bigl(h^\lambda(U^\lambda) \bigr)$. From below yields $U^\lambda \mapsto A^\lambda =\omega_{\phi^\lambda}(U^\lambda) \mapsto h^\lambda(A^\lambda)$.
Since $h$ is a conjugacy it follows from Remark \ref{conjatt} that Lemma \ref{lem:finv} applies to both $h^\lambda$ and $(h^\lambda)^{-1}$. This gives:
\begin{equation}
    \label{equivconj1}
\omega_{\psi^\lambda}\bigl(h^\lambda(U^\lambda) \bigr) =
\omega_{\psi^\lambda}\bigl(h^\lambda\bigl(\omega_{\phi^\lambda}(U^\lambda)\bigr) \bigr)
= \omega_{\psi^\lambda}\bigl(h^\lambda(A^\lambda) \bigr) = h^\lambda(A^\lambda),
\end{equation}
which proves commutativity.
\eproof

Lemma \ref{commutes1}
 holds for all $\lambda \in \mathrm{\Lambda}$ and which provides stalkwise isomorphisms between the associated sheaves of attractors. This however does not  give isomorphic sheaves necessarily!

% Before proving conjugacy invariance, we state an elementary lemma regarding the Hausdorff metric.

% \begin{lemma}\label{lemma:haus}
% Let $X$, $Y$ be compact metric spaces, $\mathrm{\Lambda}$ a topological space, and $h\colon\mathrm{\Lambda} \times X \to Y$ continuous map.  Then, the function  
% $$h_\sH\colon\mathrm{\Lambda} \times \sH(X)  \to \sH(Y) \quad \quad (\lambda,K) \mapsto h_\sH\big(\{\lambda\} \times K\big)$$
% is continuous. 

% \proof
% Omitted.
% \eproof
% \end{lemma}

\begin{theorem}[Conjugacy Invariance Theorem]
\label{thm:CIT}
Let $X$, $Y$ be 
compact metric spaces. Suppose $\obphi\colon \mathrm{\Lambda} \to \sDS(\T, X)$ and $\obpsi\colon \mathrm{\Lambda} \to \sDS(\T, Y)$ are conjugate parametrized dynamical systems.
Then, the \etale spaces $\iobphi\mathrm{\Pi}[\sAtt]$ and $\iobpsi\mathrm{\Pi}[\sAtt]$ are homeomorphic.
\end{theorem}

\proof
From Lemma \ref{commutes1} we have the following commutative diagram of maps:
\[
\begin{diagram}
\node{\iobphi\mathrm{\Pi}[\sAtt]}\arrow[2]{e,l}{h_*}\arrow{se,r}{\pi}\node{}\node{\iobpsi\mathrm{\Pi}[\sAtt]}\arrow{sw,r}{\pi}\\
\node{}\node{\mathrm{\Lambda}}\node{}
\end{diagram}
\]
where $h_*$ is defined by $(\lambda,\phi^\lambda,A^\lambda) \mapsto h_*(\lambda,\phi^\lambda,A^\lambda):= \bigl(\lambda,\psi^\lambda,h^\lambda(A^\lambda)\bigr)$.
It is sufficient to show continuity, since if $h_*$ is continuous, then $h_*$ is a local homeomorphism, in which case $h_*^{-1}$ is a also a local homeomorphism ($h_*$ is a bijection), cf.\ \cite[Prop.\ 2.4.8]{Borceux}. This proves that $\iobphi\mathrm{\Pi}[\sAtt]$ and $\iobpsi\mathrm{\Pi}[\sAtt]$ are homeomorphic. 

In order to prove continuity we argue as follows. Consider the following commutative diagram
\[
\begin{diagram}
\node{} \node{\iobphi\mathrm{\Pi}[\sAtt]}\arrow{s,l}{\pi}\\
\node{\iobphi\mathrm{\Phi}[\sANbhd;U]} \arrow{ne,l}{\iobphi\mathrm{\Theta}[\omega;U]}\arrow{e,l}{\subset}\node{\mathrm{\Lambda}}
\end{diagram}
\]
where $\iobphi\mathrm{\Phi}[\sANbhd;U] = \bigl\{\lambda~|~ U\in \sANbhd(\phi^\lambda) \bigr\}$ and
$\iobphi\mathrm{\Theta}[\omega;U](\lambda) = \bigl(\lambda,\phi^\lambda,\omega_{\phi^\lambda}(U) \bigr)$.
Let $D_0\subset \mathrm{\Lambda}$ be an open neighborhood of $\lambda_0\in \mathrm{\Lambda}$ and let 
\[
\iobpsi\mathrm{\Theta}\big[\omega; h^{\lambda_0}(U^{\lambda_0})\big](D_0) =
\biggl\{\Bigl(\lambda,\psi^\lambda,\omega_{\psi^\lambda}\bigl(h^{\lambda_0}(U^{\lambda_0})\bigr)\Bigr)~|~\lambda\in D_0 \biggr\}
\]
be an open neighborhood of $h_*\bigl(\lambda_0,\phi^\lambda,A^{\lambda_0}\bigr) = \bigl(\lambda_0,\psi^\lambda,h^{\lambda_0}(A^{\lambda_0}) \bigr)$ in $\iobpsi\mathrm{\Pi}[\sAtt]$ for some  compact $U^{\lambda_0} \in\sANbhd(\phi^{\lambda_0})$. 
In order to establish continuity we seek a neighborhood $D_0'\subset D_0\subset \mathrm{\Lambda}$
such that 
\[
\begin{aligned}
h_*\Bigl(\iobphi\mathrm{\Theta}[\omega; U^{\lambda_0}]&(D_0')\Bigr) =
\bigg\{\Big(\lambda,\psi^\lambda, h^\lambda\bigl(\omega_{\phi^\lambda}(U^{\lambda_0})\bigr)\Big)~|~ \lambda \in D_0'\bigg\}\\
&= \bigg\{\Big(\lambda,\psi^\lambda, \omega_{\psi^\lambda}\big(h^{\lambda}(U^{\lambda_0})\big)\Big)~|~ \lambda \in D_0'\bigg\} \subset \iobpsi\mathrm{\Theta}\big[\omega; h^{\lambda_0}(U^{\lambda_0})\big](D_0),
\end{aligned}
\]
where the second equality follows from  Lemma \ref{commutes1}, Eqn.\ \eqref{equivconj1}.
This is equivalent to saying 
\begin{equation*}
\omega_{\psi^\lambda}\big(h^{\lambda}(U^{\lambda_0})\bigr) =
\omega_{\psi^\lambda}\bigl(h^{\lambda_0}(U^{\lambda_0})\bigr),\quad \forall \lambda \in D_0'.
\end{equation*}
For notational convenience we write 
\[
U := h^{\lambda_0}(U^{\lambda_0})\in \sANbhd(\psi^{\lambda_0}),\quad \hbox{and}\quad A = h^{\lambda_0}(A^{\lambda_0})
=  \omega_{\psi^{\lambda_0}}(U') \in \sAtt(\psi^{\lambda_0}). 
\]
We rephrase the above condition as:
\begin{equation}
 \label{id12}
\omega_{\psi^\lambda}\big(h^{\lambda}(U^{\lambda_0})\bigr) = \omega_{\psi^\lambda}(U), \quad \forall \lambda \in D_0'.
\end{equation}

For $U^{\lambda_0} = \varnothing$, or for  $U^{\lambda_0} = \bigsqcup_i X_i\subset X$, any union 
of connected components of $X$, Eqn.\ \eqref{id12} is satisfied by the uniform conjugacy condition in Defn.\ \ref{paradyn2}(ii), cf.\ Remark \ref{paradyn3}.
For the remainder of the proof we assume $U^{\lambda_0} \not= \varnothing$ and
$U^{\lambda_0} \neq  \bigsqcup_i X_i$, for all unions
of connected components of $X$.
Therefore, we may carry out the arguments for the components $U^{\lambda_0}_i = U^{\lambda_0}\cap X_i \neq  \varnothing, X_i$.

Choose a compact attracting neighborhood $U'\in \sANbhd(\psi^{\lambda_0})$ such that $U'\subset \Int U$
and $\omega_{\psi^{\lambda_0}}(U') = A$. Indeed, since $A$ is an attractor $\cl U^c \cap A=\varnothing$, cf.\ \cite[Lemma 3.23]{KMV-1a}. Therefore there exists open sets $N,N'$ such
that $A\subset N$, $\cl U^c\subset N'$ and $N\cap N'=\varnothing$. As a matter of fact 
$\cl N\cap N' = \varnothing$. Define $U' = \cl N$. By construction $A^*\subset U^c\subset \cl U^c\subset N'$ and thus $U' \cap A^*=\varnothing$ which proves that (i) $\omega_{\psi^{\lambda_0}}(U') = A$, (ii) $A\subset N\subset U'$, (iii) $U' = \cl N \subset N'^c \subset \bigl(\cl U^c \bigr)^c = \Int U$,
and thus
$U'$ is an attracting neighborhood satisfying the properties stated above, cf.\ \cite[Lemma 3.21]{KMV-1a}. From the fact that $U \neq  \bigsqcup_i Y_i$, a union of components,
it follows that $\Int U \subsetneq U$. Thus by Property (iii) 
there exists a $\delta_1>0$ such that $B_{\delta_1}(U')\subset U$ and therefore $d_\sH(U,U')\ge \delta_1>0$, where $d_\sH$ is the Hausdorff metric on the space $\sH(X)$ of compact subsets of $X$.%, cf.\ Appendix \ref{cot}. 

By the same token we can choose a compact repelling neighborhood $V\in \sRNbhd(\psi^{\lambda_0})$
such that $V\cap U = \varnothing$ and  $\omega_{\psi^{\lambda_0}}(V^c) = A$.
 Indeed, repeat the above arguments starting with $U\cap A^*=\varnothing$. $V$ is compact, so there exists a $\delta_2>0$ such that $d_\sH(U,V)\ge \delta_2>0$.

Since,  $\iobpsi\mathrm{\Theta}\big[\omega; U\big]$, $\iobpsi\mathrm{\Theta}\big[\omega; U'\big]$
and $\iobpsi\mathrm{\Theta}\big[\omega; V^c\big]$ define local sections in $\iobpsi\mathrm{\Pi}[\sAtt]$ over $\iobpsi\mathrm{\Phi}[\sANbhd;U]$, $\iobpsi\mathrm{\Phi}[\sANbhd;U']$ and $\iobpsi\mathrm{\Phi}[\sANbhd;V^c]$ respectively, and since
\[
\iobpsi\mathrm{\Theta}\big[\omega; U\big](\lambda_0) = \iobpsi\mathrm{\Theta}\big[\omega; U'\big](\lambda_0) = \iobpsi\mathrm{\Theta}\big[\omega; V^c\big](\lambda_0)
\]
there exists an open set $E_0\subset \mathrm{\Lambda}$ on which three sections coincide, i.e.
\[
B^\lambda:=\omega_{\psi^{\lambda}}(U) = \omega_{\psi^{\lambda}}(U') = \omega_{\psi^{\lambda}}(V^c), \quad\forall \lambda \in E_0,
\]
and $B^\lambda \subset \Int U$, $B^\lambda \subset \Int U'$ and $B^\lambda \subset \Int V^c$ for all $\lambda \in E_0$.

Let $\widetilde U$ be any compact neighborhood such that $d_\sH(U,\widetilde U)<\delta = \min\{\delta_1,\delta_2\}/2$ and let $\lambda\in E_0$.
Then, 
\[
B^\lambda \subset U'\subset \widetilde U,\qquad \widetilde U\cap (B^\lambda)^* \subset \widetilde U\cap V=\varnothing,
\]
which by \cite[Lemma 3.21]{KMV-1a} implies that $\omega_{\psi^\lambda}(\widetilde U) = B^\lambda$ for all $\lambda \in E_0$.

Finally, using the continuity of $h^\lambda_\sH$ in Lemma \ref{lemma:haus}, choose an open sets $D_0'\subset E_0\cap D_0$  such that 
$d_\sH\bigl(h^\lambda(U^{\lambda_0}), U\bigr) < \delta$ for all $\lambda \in D_0'$.
By the previous we choose $\widetilde U = h^\lambda(U^{\lambda_0})$ which proves that 
\[
\omega_{\psi^\lambda}\bigl(h^{\lambda}(U^{\lambda_0})\bigr) = B^\lambda = \omega_{\psi^\lambda}(U),
\quad\forall \lambda \in D_0',
\]
establishing \eqref{id12} and thereby the theorem.
\eproof

\begin{remark}
The condition that the spaces $X$ and $Y$ are compact metric spaces is used at several places in the proof and in particular for using the Hausdorff metric. The characterizations of attracting and repelling neighborhoods via attractors and dual repellers at least works in compact Hausdorff spaces.
\end{remark}

Theorem \ref{thm:CIT} can be extended to other structures. Since $\iobphi\mathrm{\Pi}[\sAtt]$ is homeomorphic (as a sheaf of sets) to $\iobphi\mathrm{\Pi}[\sRep]$, we can get a homeomorphism between $\iobphi\mathrm{\Pi}[\sRep]$ and $\iobpsi\mathrm{\Pi}[\sRep]$. There is the following commutative diagram for Morse sets:

\[
\begin{tikzcd}
\iobphi\mathrm{\Pi}[\sAtt]\bullet\iobphi\mathrm{\Pi}[\sAtt]\arrow[rrr,"{  }"]\arrow[d, "\mathrm{\Pi}{[}\sC_\sAtt{]}"]&&&
\iobpsi\mathrm{\Pi}[\sAtt]\bullet\iobpsi\mathrm{\Pi}[\sAtt]\arrow[d, "\mathrm{\Pi}{[}\sC_\sAtt{]}"] \\ \iobphi\mathrm{\Pi}[\sMorse]\arrow[rrr, "{}"]\arrow[dr]&&&
\iobpsi\mathrm{\Pi}[\sMorse]\arrow[dll]
\\
& \mathrm{\mathrm{\Lambda}}
\end{tikzcd}
\]
where the top horizontal map is given by 
\[
(\lambda ,\phi^\lambda , A),(\lambda ,\phi^\lambda , A') \mapsto   (\lambda ,\psi^\lambda , h^\lambda (A)),(\lambda ,\psi^\lambda , h^\lambda (A'))
\]
and the bottom horizontal map is given by
\[
(\lambda , \phi^\lambda , M) \mapsto (\lambda , \psi^\lambda , h^\lambda (M)),
\]
which, using a similar argument to Propositon \ref{prop:shmor}, establishes that the \etale spaces $\iobphi\mathrm{\Pi}[\sMorse]$ and $\iobpsi\mathrm{\Pi}[\sMorse]$ are homeomorphic. 

\begin{corollary}
\label{CIT2}
Let $X$ and $Y$ be homeomorphic compact metric spaces and let $\sAtt_X$ and $\sAtt_Y$ be the attractor functors on $\sDS(\T,X)$ and $\sDS(\T,Y)$ respectively. Then, the \etale spaces $\mathrm{\Pi}[\sAtt_X]$ and 
$\mathrm{\Pi}[\sAtt_Y]$ are homeomorphic.
\end{corollary}

\proof
Let $h\colon X\to Y$ be a homeomorphism and let $\mathrm{\Lambda} = \sDS(\T,X)$. Then, $\obphi$ is the identity map. The map $\obpsi\colon \mathrm{\Lambda} \to \sDS(\T,Y)$ is defined as follows: $\mathrm{\Lambda}\ni\phi \mapsto h\circ \phi \circ h^{-1} = \psi$. Then,
\[
h\bigl(\phi_t(x)\bigr) = h\Bigl(\phi_t\bigl(h^{-1}(y)\bigr)\Bigr) = \psi_t(y) = 
\psi_t\bigl(h(x)\bigr),
\]
which proves that  $\obphi$ and $\obpsi$ are  conjugate parametrized dynamical systems.
\eproof

\section{Bifurcations and sheaf cohomology}
\label{ssc:examples}

Sheaves attach both local and global data to a topological space. In our setting of continuation, they encode how dynamical structures vary with parameter values on open sets. Oftentimes, given an open cover of the topological space, one can glue together the local information on each element of the cover to obtain global information. 
However, sometimes local information fails to extend globally. Sheaf cohomology, which can be viewed as a generalization of singular cohomology, is a powerful tool for studying this. An interpretation for singular cohomology groups is that they constitute obstructions to a topological space being contractible. Sheaf cohomology generalizes this by representing barriers for local sections to extend to global sections.
 
One can always solve an attractor's continuation \textit{locally} using an attracting neighborhood. But this problem is sometimes impossible \textit{globally}. 
Sheaf cohomology provides a framework for quantifying when and how this occurs. Together with the conjugacy invariance theorem, this will build an algebraic invariant for parametrized dynamical systems, which can be used to study bifurcations. %A description of how sheaf cohomology is constructed, and then computed, is in Appendix \ref{sec:SheafCohomology}.

Recall that a parametrized dynamical system on a topological space $\mathrm{\Lambda}$ is a continuous
map $\obphi\colon \mathrm{\Lambda}\to \sDS(\T,X)$ such that $\obphi(\lambda)\colon \T\times X \to X$ is a dynamical system for all $\lambda\in \mathrm{\Lambda}$. In principal $\mathrm{\Lambda}$ may be $\sDS(\T,X)$ but in practice simpler topological spaces for $\mathrm{\Lambda}$ are used. In this section, to utilize Theorem \ref{thm:CIT}, \emph{we assume $X$ is a compact metric space.}

\begin{definition}\label{def:stability}
A parametrized dynamical system $\obphi\colon\mathrm{\Lambda} \to \sDS(\T,X)$ is \emph{stable} at a point $\lambda_0 \in \mathrm{\Lambda}$ if there exists an open neighborhood $\mathrm{\Lambda}' \ni \lambda_0$ such that $\restr{\obphi}{\mathrm{\Lambda}'}$ is conjugate to the constant parametrization $\obtheta\colon \mathrm{\Lambda}'\to \sDS(\T,X)$, given by $\lambda \mapsto \obphi(\lambda_0)$ for all $\lambda \in \mathrm{\Lambda}'$. If $\lambda_0$ is not stable, it is called a \emph{bifurcation point}. 
A parametrized dynamical system $\obphi$ is \emph{stable} on a subset $\mathrm{\Lambda}'\subset \mathrm{\Lambda}$ if it is stable at every point in $\mathrm{\Lambda}'\subset \mathrm{\Lambda}$.

If a parametrized dynamical system $\obphi\colon\mathrm{\Lambda} \to \sDS(\T,X)$ is conjugate to the constant parametrization $\obtheta\colon \mathrm{\Lambda}\to \sDS(\T,X)$ on $\mathrm{\Lambda}$ it is called \emph{uniformly stable}.
\end{definition}

In general stability of a parametrized dynamical system does not imply uniform stability. For instance if $\mathrm{\Lambda}$ is not connected 
then $\obphi$ need not be conjugate to a fixed constant system $\obtheta$. 
This example indicates that stability does not imply uniform stability in general if $\mathrm{\Lambda}$
is disconnected. See Example \ref{counterconstant} for an counter example with a connected space $\mathrm{\Lambda}$.

\subsection{Locally constant sheaves}
\label{locconstsh}
Let $\obphi\colon \mathrm{\Lambda} \to \sDS(\T,X)$ be a parametrized dynamical system. 
From the previous we have the induced attractor sheaf and free attractor sheaf  over $\mathrm{\Lambda}$:
\[
\scrA^\obphi  \colon \cO(\mathrm{\Lambda}) \to \sRing,\qquad \gAtt^\obphi  \colon \cO(\mathrm{\Lambda}) \to \sRing.
\]
The ringed spaces $(\mathrm{\Lambda},\scrA^\obphi)$ and $(\mathrm{\Lambda},\gAtt^\obphi)$ encode the continuation data of attractors for the parametrized dynamical system. At a later stage we also include the attracting neighborhood sheaf and free attracting neighborhood sheaf $\scrN$ and $\gANbhd$ respectively.

Recall that for an abelian group
$\sE\in \sAb$ the presheaf $\scrE\colon \cO(\mathrm{\Lambda}) \to \sAb$  defined by $\scrE(\mathrm{\Lambda}') := \bigl\{\sigma\colon \mathrm{\Lambda}' \to \sE~~\text{constant} \bigr\}$, $\mathrm{\Lambda}'\subset \mathrm{\Lambda}$ open, is called the \emph{constant presheaf} over $\mathrm{\Lambda}$ with values in $\sE$. The sheafification $\underline{\sE} := \scrE^\#$ is called the \emph{constant sheaf} over $\Lambda$ with values in $\sE$. The constant sheaf can be characterized as the sheaf of locally constant functions  with values in $\sE$, i.e.
\[
\underline \sE(\mathrm{\Lambda}') = \bigl\{\sigma\colon \mathrm{\Lambda}' \to \sE~~\text{locally constant} \bigr\},\quad \mathrm{\Lambda}'\subset \mathrm{\Lambda},~~\text{open}.
\]
If we equip $\sE$ with the discrete topology then such functions are continuous functions $\sigma\colon \mathrm{\Lambda}'\to \sE$. This corresponds to the sheaf of section of the \etale space $\mathrm{\Lambda}\times\sE$, with $\sE$ equipped with the discrete topology, cf.\ \cite[Sect.\ 2.4]{Tenni}, \cite[Ex.\ 3.31 and 3.40]{Wedhorn}.
If $\mathrm{\Lambda}'\subset \mathrm{\Lambda}$ is open and connected then $\underline \sE(\mathrm{\Lambda}') \cong\sE$. For an open set whose connected components are open then $\underline \sE(\mathrm{\Lambda}')$ is isomorphic to a direct  product of copies of $\sE$, one for each connected component, cf.\ \cite[Ex.\ 1.0.3]{Hartshorne}.
An abelian sheaf $\scrF$ is called \emph{locally constant} is there exists an open covering $\{\mathrm{\Lambda}_i\}$
of $\mathrm{\Lambda}$ such that $\scrF|_{\mathrm{\Lambda}_i}$ is a constant sheaf for all $i$. This is equivalent to saying that every point allows a neighborhood $\mathrm{\Lambda}'\subset \mathrm{\Lambda}$ such that $\scrF|_{\mathrm{\Lambda}'}$ is constant, cf.\ \cite[Defn.\ I.1.9]{Bergman}.
Locally constant sheaves are sheaves of sections of covering spaces, \cite[Ex.\ 3.41]{Wedhorn}.

\begin{lemma}
\label{prop:stableconstant00}
Let  $\obtheta\colon \mathrm{\Lambda} \to \sDS(\T,X)$ be a constant parametrization. Then,  the sheaves $\scrA^\obtheta$
and $\gAtt^\obtheta$ are  constant sheaves. 
\end{lemma}

\proof
The pullback \etale space $\iobtheta\mathrm{\Pi}[\sAtt]$ is given by
\[
\iobtheta\mathrm{\Pi}[\sAtt] \cong \mathrm{\Lambda} \times \sA,
\]
where $\sA = \sAtt(\phi^{\lambda_0})$, for some $\lambda_0\in \mathrm{\Lambda}$, 
is given the discrete topology. Therefore the sheaf of sections 
$\iobtheta\scrS^{\sAtt}$ is a constant sheaf.
%, cf.\ \ref{constshex}.
Consequently, $\scrA^\obtheta$ and $\gAtt^\obtheta$ are also constant sheaves.
\eproof

\begin{lemma}\label{prop:stableconstant}
Let  $\obphi\colon \mathrm{\Lambda} \to \sDS(\T,X)$ be stable. Then,  the sheaves $\scrA^\obphi$
and $\gAtt^\obphi$ are locally constant sheaves. 
\end{lemma}

\proof

Pick a point $\lambda_0\in \mathrm{\Lambda}$. 
Since
$\obphi$ is stable 
there exists a neighborhood $\mathrm{\Lambda}'\ni \lambda_0$ such that 
$\obphi|_\mathrm{\Lambda'}$ is conjugate to the constant parametrization. 
By the Conjugacy Invariance Theorem in \ref{thm:CIT} we have that 
$\scrA^\obphi|_{\mathrm{\Lambda}'} \cong \scrA^\obtheta|_{\mathrm{\Lambda}'}$ as sheaves. The latter is a constant sheaf over $\mathrm{\Lambda}'$ and therefore $\scrA^\obphi|_{\mathrm{\Lambda}'}$ is a constant sheaf over $\mathrm{\Lambda}'$ by definition. 
We conclude that $\scrA^\obphi$ is locally constant. The same applies to $\gAtt^\obphi$.
\eproof

\begin{remark}
If $\obphi$ is uniformly stable then $\obphi$ is conjugate to a constant parametrization $\obtheta $ on $\mathrm{\Lambda}$. The associated \etale spaces are homemorphic  by Theorem
\ref{thm:CIT} and thus the sheaves $\scrA^\obphi$
and $\gAtt^\obphi$ are   constant sheaves is this case.
\end{remark}

\begin{example}
\label{counterconstant}
Let $X$ be the 2-point compactification of the line and consider the following family of differential equations
\[
\dot x = \sin(x+\lambda),\quad x\in \R,~~\lambda \in \Sb^1=\R/2\pi\Z.
\]
The above system defines a 1-parameter family of flows $\obphi\colon \mathrm{\Lambda}\to \sDS(\R,X)$
of flows on $X$ over parameter space $\mathrm{\Lambda} = \Sb^1$. Via the conjugacy $x\mapsto x-\lambda$ we conclude that $\obphi$ is stable and thus the attractor sheaf $\scrA^\obphi$ is a locally constant sheaf as indicated by Lemma \ref{prop:stableconstant}. Since $\pm\infty$ are not attractors, the only global sections in $\scrA^\obphi$ are $\varnothing$ and $X$. The stalks of $\scrA^\obphi$ are infinite  complete, atomic Boolean algebras which proves that $\scrA^\obphi$ is not a constant sheaf.
\end{example}

The above example shows that even if $\mathrm{\Lambda}$ is connected, then a stable system need not be uniformly stable. Indeed, $\obphi$ in Example \ref{counterconstant} allows a conjugacy over $\mathrm{\Lambda} = \Sb^1$, then
the attractor sheaf $\scrA^\obphi$ is constant which contradicts above statement that $\scrA^\obphi$ is locally constant but not constant. 

\begin{example}\label{ex:mobiusmonodromy}
Define a vector field restricted to the compact subset $[-2,2]\times [-2,2]$ of $\R^2$: 
\[
F(x,y) = (-x(x+1)(x-1), -y).
\]
We can rotate the vector field with a parameter $\theta$:
\[
F_\theta(x,y) = R_{-\theta}F(R_{\theta}(x,y)),
\]
where $R_\theta$ denotes the rotation matrix of angle $\theta$. Since $F_\pi(x,y)=F_0(x,y)$, gluing at $0$ and $\pi$ (not 2$\pi$) gives us a parametrized dynamical system $\obphi\colon \Sb^1 \to \sDS(\R,[0,2]\times[0,2])$  by integrating the vector field.
%from solving the ODE. 
The invariant set $[-1,1]\times \{0\}$ undergoes a half-twist over $\Sb^1$. There are only three global sections of $\iobphi\scrS^\sAtt$:
\[\theta \mapsto \emptyset, \quad \theta \mapsto R_\theta(\{(1,0),(-1,0)\}), \quad \theta \mapsto R_\theta([-1,1]\times \{0\}).\]
Alas, each stalk is a five element lattice and $\Sb^1$ is connected, so $\iobphi\scrS^\sAtt$ is not the constant sheaf. Additionally, the five element attractor lattice is a global section in $\iobphi\scrS^{\subF\sAtt}$, but cannot be represented as a collection of global sections of $\iobphi\scrS^\sAtt$.

\begin{figure}[]
\label{fig:flip}
    \centering
    \includegraphics[width=\textwidth]{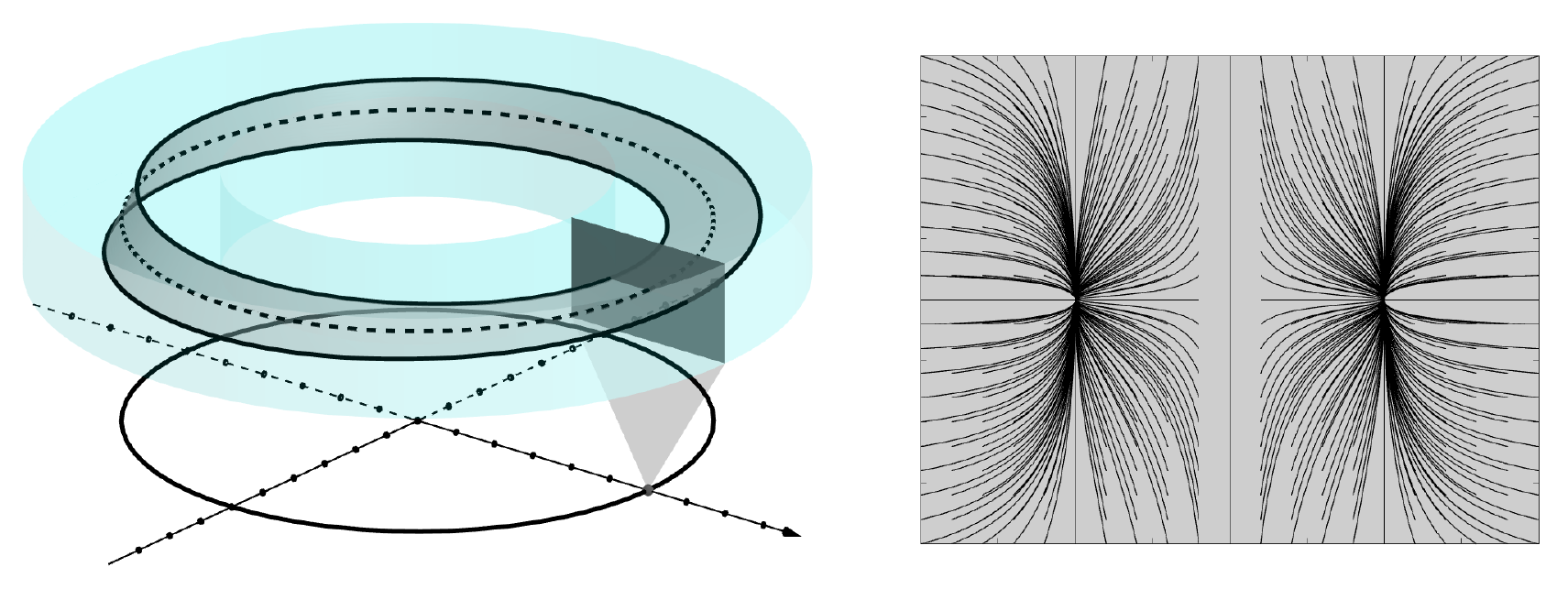}

\caption{ An illustration of example \ref{ex:mobiusmonodromy}. The vector field at the cross section is rotated over the parameter space $\Sb^1$. The system is stable, so the attractor sheaf is locally constant. However, it is impossible to continue either attracting fixed point globally.}
\end{figure}

\end{example}

As pointed out above, a locally constant sheaf is the sheaf of sections of a covering space. 
With additional conditions on $\mathrm{\Lambda}$
such sheaves may be constant sheaves.

\begin{proposition}[cf.\ \cite{Iversen}, Prop.\ 4.20 and \cite{Wedhorn}, Prop.\ 7.5]
\label{prop:stableconstant1a}
Let  $\mathrm{\Lambda}$ be a simply connected and locally path connected topological space,
%\footnote{A topological space $\mathrm{\Lambda}$ is locally path connected at $\lambda$ if for every  neighborhood $\mathrm{\Lambda}'\ni \lambda$ there exists a path connected neighborhood $\mathrm{\Lambda}''\subset \mathrm{\Lambda}'$ containing $\lambda$. A space is locally path connected if it is locally path connected at every point, cf.\ \cite[Sect.\ 25]{Munkres}} 
and let $\scrF$ be a locally constant sheaf of rings on $\mathrm{\Lambda}$.
Then,  $\scrF$ is a  constant sheaf.
\end{proposition}

The same statement holds for contractible spaces $\mathrm{\Lambda}$, cf.\ \cite[Exer.\ II.4]{Schapira}.
We can apply the above proposition to the attractor sheaf $\scrA^\obphi$ and free attractor sheaf $\gAtt^\obphi$ for simple parametrized systems $\obphi$.

\begin{corollary}
\label{prop:stableconstant2}
Let $\obphi\colon \mathrm{\Lambda} \to \sDS(\T,X)$ be stable and 
let  $\mathrm{\Lambda}$ be a simply connected and locally path connected topological space. Then,  $\scrA^\obphi$
and $\gAtt^\obphi$ are constant sheaves.
\end{corollary}

For constant sheaves the sheaf cohomology can be related to singular cohomology which is a useful tool in our treatment of bifurcations.

\begin{proposition}[cf.\ \cite{Notzeb}, Thm.\ 9]
\label{prop:stableconstant3}
Let $\mathrm{\Lambda}$ be a locally contractible topological space (\cite[p.\ 57]{Spanier}),
%\footnote{A topological space $\mathrm{\Lambda}$ is locally contractible at $\lambda$ if for every  neighborhood $\mathrm{\Lambda}'\ni \lambda$ there exists a  neighborhood $\mathrm{\Lambda}''\subset \mathrm{\Lambda}'$ containing $\lambda$ which is contractible in $\mathrm{\Lambda}'$. A space is locally contractible if it is locally contractible at every point, cf.\ \cite[Exer.\ Ch.\ 1]{Spanier}.}
and let $\sR$ be an arbitrary  ring. If  $\underline\sR$ denotes the constant sheaf with values in $\sR$, then $H^k(\mathrm{\Lambda};\underline\sR) \cong H^k_{\rm sing}(\mathrm{\Lambda};\sR)$ for all $k$.
\end{proposition}

If we combine Lemma \ref{prop:stableconstant}, Corollary \ref{prop:stableconstant2} and Proposition \ref{prop:stableconstant3}
we obtain a result that determines the sheaf cohomology of the attractor sheaves for simple parametrized dynamical systems.

\begin{corollary}
\label{prop:stableconstant4}
Let $\obphi\colon \mathrm{\Lambda} \to \sDS(\T,X)$ be stable and 
let $\mathrm{\Lambda}$ be a locally contractible and simply connected topological space. Then, 
\[
H^k(\mathrm{\Lambda};\scrA^\obphi) \cong H^k_{\rm sing}(\mathrm{\Lambda};\scrA^\obphi_{\lambda_0}),\quad \forall k,
\]
where $\scrA^\obphi_{\lambda_0}\in \sRing$ is a stalk at any $\lambda_0\in \mathrm{\Lambda}$.
A similar statement holds for $H^k(\mathrm{\Lambda};\gAtt^\obphi)$.
\end{corollary}

\proof
Lemma \ref{prop:stableconstant} implies that $\scrA^\obphi$ is a constant sheaf.
A locally contractible space is locally simply connected and locally path connected,
but not necessarily simply connected.
In combination with the condition of simple connectedness we can combine  Corollary \ref{prop:stableconstant2} and Proposition \ref{prop:stableconstant3}, which completes the proof.
\eproof

\subsection{Sufficient conditions}
\label{sufficient}

%Contractible spaces are path connected and simply connected 
%but not necessarily locally contractible.
%Contractible subsets of a manifold (or CW-complex) are locally contractible and simply connected which yields the following corollary.
%We refer such spaces a contractible manifolds (or contractible CW-complexes).  We formulate the theorems for contractible manifolds but they can also be formulated for contractible CW-complexes.
The statements about sheaf cohomology in Section \ref{locconstsh} imply the following sufficient condition for bifurcations to exist. The theorems stated for the attractor sheaf $\scrA^\obphi$ can also be stated for the free attractor sheaf $\gAtt^\obphi$.

\begin{theorem}
\label{thm:nonzerobifur}
Let $\mathrm{\Lambda}$ be both contractible and locally contractible. 
Suppose that
\[
H^k(\mathrm{\Lambda}; \scrA^\obphi)\neq0,
\quad\text{ for some }\quad k>0.
\]
Then, there  exist a bifurcation point in $\lambda_0 \in \mathrm{\Lambda}$.
\end{theorem}

\proof
Suppose there are no bifurcation points. This implies that $\obphi$ is stable 
which by Corollary \ref{prop:stableconstant4} implies that 
$H^k(\mathrm{\Lambda};\scrA^\obphi) \cong H^k_{\rm sing}(\mathrm{\Lambda};\sR)$ for all $k$ (where $\sR$ is isomorphic to a stalk of $\scrA^\obphi$).
Since $\mathrm{\Lambda}$ is contractible, we have that 
$H^k_{\rm sing}(\mathrm{\Lambda};\sR) =0$ for all $k>0$.
Combining these statements yields that $H^k(\mathrm{\Lambda};\scrA^\obphi) \cong H^k_{\rm sing}(\mathrm{\Lambda};\sR) =0$ for all $k>0$, which contradicts the above assumptions.
\eproof

% {\color{red}
% \begin{remark}
% Contractible spaces are path connected and simply connected 
% but not necessarily locally contractible.
% Contractible subsets of a manifold, or CW-complex are locally contractible and simply connected which implies that the results apply to contractible CW-complexes.
% %We refer such spaces a contractible manifolds (or contractible CW-complexes).  We formulate the theorems for contractible manifolds but they can also be formulated for contractible CW-complexes.
% \end{remark}
% }

As we will see in Section \ref{sec:exonepar} the above criterion does not always detect bifurcations.
In order to get a more in depth look into local bifucations we  consider its relative sheaf cohomology for $\scrA^\obphi$. 
We use the following lemma about long exact sequences in sheaf cohomology.

\begin{lemma}
\label{relshco1}
Let $\scrF$ be a sheaf of rings on $\mathrm{\Lambda}$ and let
 $\mathrm{\Lambda}'\xhookrightarrow{i} \mathrm{\Lambda}$. 
 Assume that the induced homomorhisms
$i_*^k\colon H^k(\mathrm{\Lambda};\scrF) \to H^k(\mathrm{\Lambda}';\scrF)$ are isomorphisms for all $k\ge 0$.
Then, 
\[
H^k(\mathrm{\Lambda},\mathrm{\Lambda}';\scrF) \cong 0,\quad \forall k\ge 0.
\]
\end{lemma}

\proof
For triple $(\mathrm{\Lambda}',\varnothing) \xhookrightarrow{i}(\mathrm{\Lambda},\varnothing) \xhookrightarrow{j} (\mathrm{\Lambda},\mathrm{\Lambda}')$ we have the long exact sequence,
\[
\begin{tikzcd}[column sep=large]
 0  \arrow{r}{\delta^0}  & 
 H^0(\mathrm{\Lambda},\mathrm{\Lambda}';\scrF) \arrow{r}{j_*^0}  &
 H^0(\mathrm{\Lambda};\scrF)  \arrow{r}{i_*^0} \ar[draw=none]{d}[name=X, anchor=center]{} &
 H^0(\mathrm{\Lambda}';\scrF) \ar[rounded corners,
            to path={ -- ([xshift=2ex]\tikztostart.east)
                      |- (X.center) \tikztonodes
                      -| ([xshift=-2ex]\tikztotarget.west)
                      -- (\tikztotarget)}]{dll}[at end]{\delta^1} \\ 
& H^1(\mathrm{\Lambda},\mathrm{\Lambda}';\scrF) \arrow{r}{j_*^1} & 
 H^1(\mathrm{\Lambda};\scrF) \arrow{r}{i_*^1} & 
 H^1(\mathrm{\Lambda}';\scrF) \arrow{r}{\delta^2} &
 \cdots.%,
\end{tikzcd}
\]
 %cf.\ Appendix \ref{sec:SheafCohomology}.
For the exactness of the maps and the isomorphisms $i_*^k$ we have:
$\ker j_*^0 = \im \delta^0=0$, which proves that $j_*^0$ is injective. Furthermore,
since $i_*^0$ is an isomorphism we have
$\ker i_*^0 = 0 = \im j_*^0$ and thus $H^0(\mathrm{\Lambda},\mathrm{\Lambda}';\scrF) \cong 0$.
The remaining relative homology groups are determined as follows: $\ker \delta^1 = \im i_*^0 = H^0(\mathrm{\Lambda}';\scrF)\cong H^0(\mathrm{\Lambda};\scrF)$. Therefore, $\ker j_*^1 = \im \delta^1 = 0$, which shows that $j_*^1$ is injective. Furthermore,
$\ker i_*^1 =0=\im j_*^1$, consequently $H^1(\mathrm{\Lambda},\mathrm{\Lambda}';\scrF) \cong 0$.
The same argument can be repeated now for all other $k$.
\eproof

As an immediate consequence of the long exact sequence we have the following corollary
if we apply Lemma \ref{relshco1} to the attractor sheaf $\scrA^\obphi$.
\begin{corollary}
\label{relshco2}
Suppose $H^k(\mathrm{\Lambda},\mathrm{\Lambda}';\scrA^\obphi) \neq 0$ for some $k$. Then, there exist $k_0\ge 0$ for which
the inclusion $i$ does not imply an isomorphism $i_*^{k_0}\colon H^{k_0}(\mathrm{\Lambda};\scrA^\obphi) \to H^{k_0}(\mathrm{\Lambda}';\scrA^\obphi)$.
\end{corollary}

The relative sheaf cohomology can be used to formulate an analogous criterion as Theorem \ref{thm:nonzerobifur}.

\begin{theorem}
\label{relshco3}
Let $\mathrm{\Lambda}$ be both contractible and locally contractible, and let
 $\mathrm{\Lambda}'\subset \mathrm{\Lambda}$ be a deformation retract of $\mathrm{\Lambda}$ 
with  $\obphi$  stable on $\mathrm{\Lambda}'$.
% Let $\mathrm{\Lambda}$ be a contractible manifold and $\mathrm{\Lambda}'\subset \mathrm{\Lambda}$ be a submanifold which is a deformation retract of $\mathrm{\Lambda}$ 
% with  $\obphi$  stable on $\mathrm{\Lambda}'$.
Suppose that 
\[
H^k(\mathrm{\Lambda}, \mathrm{\Lambda}';\scrA^\obphi)\neq 0,
\quad\text{ for some }\quad k\ge 0.
\]
Then, there  exist a bifurcation point in $\lambda_0 \in \mathrm{\Lambda}\smin \mathrm{\Lambda}'$.
\end{theorem}

\proof
Suppose there are no bifurcation points in $\mathrm{\Lambda}\smin \mathrm{\Lambda}'$. This implies that $\obphi$ is stable on $\mathrm{\Lambda}$. 
Since $\mathrm{\Lambda}$ is contractible and locally contractible, it is simply connected and  locally path connected. It follows from Proposition \ref{prop:stableconstant2} that $\scrA^\obphi$ is a constant sheaf on $\mathrm{\Lambda}$. Since $\mathrm{\Lambda}'$ is a deformation retract of $\mathrm{\Lambda}$, the same holds for $\mathrm{\Lambda}'$ and $\scrA^\obphi|_{\mathrm{\Lambda}'} \cong \scrA^\obphi$. This implies that $H^0(\mathrm{\Lambda};\scrA^\obphi) \cong H^0(\mathrm{\Lambda}';\scrA^\obphi)$. By Corollary \ref{prop:stableconstant4}, since $i^k_* \colon H^k(\mathrm{\Lambda};\sR)\to H^k(\mathrm{\Lambda}';\sR)$ is an isomorphism for all $k$, we have that $H^k(\mathrm{\Lambda};\scrA^\obphi) \cong H^k(\mathrm{\Lambda}';\scrA^\obphi)\cong 0$ for all $k\ge 1$. Combining these statements gives $H^k(\mathrm{\Lambda};\scrA^\obphi)\cong H^k(\mathrm{\Lambda}';\scrA^\obphi)$ for all $k$.
This implies by Lemma \ref{relshco1} that 
$H^k(\mathrm{\Lambda}, \mathrm{\Lambda}';\scrA^\obphi)\cong 0$ for all $k$, which
contradicts the assumption that $H^k(\mathrm{\Lambda}, \mathrm{\Lambda}';\scrA^\obphi)\neq 0$ for some $k$.
Therefore, $\obphi$ is not stable on $\mathrm{\Lambda}\smin\mathrm{\Lambda}'$ and there exists a bifurcation point
$\lambda_0\in \mathrm{\Lambda}\smin \mathrm{\Lambda}'$.
\eproof

\section{Examples of one-parameter bifurcations}
\label{sec:exonepar}
In this section we discuss a number of standard one-parameter bifurcations such as a saddle-node bifurcation and a pitchfork bifurcation. We will also examine bifurcation at multiple bifurcation points. The objective is to show that sheaf cohomology picks up bifurcarions. At a later stage we will discuss the more practical side of computing sheaf cohomology from limited data.

\subsection{One-parameter bifurcations at a single parameter value}
\label{single1par}
In this subsection we list three fundamental bifurcations in one-parameter systems. We apply the above results to compute the sheaf cohomology and to compare the criteria.
 
For example if $\mathrm{\Lambda} =\R$ or $\mathrm{\Lambda} = I$, a bounded interval, then the above theorem applies. This is of interest for one-parameter bifurcations.
The following lemma addresses the case where $\obphi$ has one bifurcation point on $\R$, which will assist in computations.

\begin{lemma}\label{lemma:pitchacyclic}
Let $\scrF$ be a sheaf of rings on $\mathrm{\Lambda}=\R$, such that $\scrF$ is a constant sheaf on both $(-\infty, \lambda_0)$ and $(\lambda_0,\infty)$ for some $\lambda_0\in \R$. Then, $\scrF$ is acyclic, i.e. $H^k(\mathrm{\Lambda},\scrF)=0$ for all $k\ge 1$.
\end{lemma}

\proof
 Let $\epsilon>0$  and let $B_\epsilon$ denote the interval $(\lambda_0-\epsilon,\lambda_0+\epsilon)$. There is a restriction cohomomorphism $r\colon \scrF \leadsto \restr{\scrF}{B_\epsilon}$. We will show this induces an isomorphism of cohomology:
 \begin{equation}
     \label{cohom11}
 r^*\colon H^*(\R;\scrF) \to H^*\big(B_\epsilon; \restr{\scrF}{B_\epsilon}\big).
  \end{equation}
 First we address global sections. Because $\scrF$ is constant on $(-\infty, \lambda_0)$ and $(\lambda_0,\infty)$, sections in $\mathrm{\Gamma}(\restr{\scrF}{B_\epsilon})$ extend uniquely to sections in $\mathrm{\Gamma}(\scrF)$. Thus, $r_0^*\colon \mathrm{\Gamma}(\scrF)\to \mathrm{\Gamma}(\restr{\scrF}{B_\epsilon})$ is an isomorphism. 
 For $k>1$, $H^k(\R;\scrF)$ and $H^k\big(B_\epsilon; \restr{\scrF}{B_\epsilon}\big)$ vanish, since intervals have covering dimension 1, cf.\ \cite[Lemma 2.7.3 and Proposition 3.2.2]{Schapira}. So the maps $$r_k^*\colon H^k(\R;\scrF) \to H^k\big(B_\epsilon; \restr{\scrF}{B_\epsilon}\big)$$ are trivially isomorphisms. Now we consider $k=1$. Let $\R^* = \R \smin \{\lambda_0\}$, so that $B_\epsilon$ and $\R^*$ form a cover of $\R$. Note that $\restr{\scrF}{\R^*}$ is locally constant, with vanishing higher cohomology groups. There is a Mayer-Vietoris exact sequence:
\[
\begin{tikzcd}[column sep=small]
 0  \arrow{r}{}  & 
 \mathrm{\Gamma}(\scrF) \arrow{r}{\alpha}  &
 \mathrm{\Gamma}\big(\restr{\scrF}{B_\epsilon}\big)\oplus\mathrm{\Gamma}\big(\restr{\scrF}{\R^*}\big)  \arrow{r}{\beta} \ar[draw=none]{d}[name=X, anchor=center]{} &
 \mathrm{\Gamma}\big(\restr{\scrF}{B_\epsilon\cap \R^*}\big) \ar[rounded corners,
            to path={ -- ([xshift=2ex]\tikztostart.east)
                      |- (X.center) \tikztonodes
                      -| ([xshift=-2ex]\tikztotarget.west)
                      -- (\tikztotarget)}]{dll}[at end]{\delta} \\ & H^1(\R;\scrF) \arrow{r} & 
 H^1\big(B_\epsilon; \restr{\scrF}{B_\epsilon}\big)\oplus H^1\big(\R^*; \restr{\scrF}{\R^*}\big) \arrow{r} & 
 H^1\big(B_\epsilon\cap \R^*; \restr{\scrF}{B_\epsilon\cap \R^*}\big) \arrow{r} &
 0.
\end{tikzcd}
\]
Since $H^1(\R; \restr{\scrF}{\R^*})$ and $H^1(B_\epsilon\cap \R^*; \restr{\scrF}{B_\epsilon\cap \R^*})$ vanish the sequence simplifies to:
\[
\begin{tikzcd}
 0  \arrow{r}  & 
 \mathrm{\Gamma}(\scrF) \arrow{r}{\alpha}  &
 \mathrm{\Gamma}\big(\restr{\scrF}{B_\epsilon}\big)\oplus\mathrm{\Gamma}\big(\restr{\scrF}{\R^*}\big)  \arrow{r}{\beta} \ar[draw=none]{d}[name=X, anchor=center]{} &
 \mathrm{\Gamma}\big(\restr{\scrF}{B_\epsilon\cap \R^*}\big) \ar[rounded corners,
            to path={ -- ([xshift=2ex]\tikztostart.east)
                      |- (X.center) \tikztonodes
                      -| ([xshift=-2ex]\tikztotarget.west)
                      -- (\tikztotarget)}]{dll}[at end]{\delta} \\
& H^1(\R;\scrF) \arrow{r}{r_1^*} & 
 H^1(B_\epsilon; \restr{\scrF}{B_\epsilon}) \arrow{r} & 
 0.
\end{tikzcd}
\]
The map $\beta$ is surjective, since the restriction from $\mathrm{\Gamma}\big(\restr{\scrF}{\R^*}\big)$ to $\mathrm{\Gamma}\big(\restr{\scrF}{B_\epsilon\cap \R^*}\big)$ is surjective. Following the sequence yields $\Img \, \delta = \ker r_1^* = 0$. $\im \, r_1^* = \ker 0 = H^1(B_\epsilon; \restr{\scrF}{B_\epsilon})$, so $r_1^*$ is also surjective. 
This implies that the restriction cohomomorphism $r\colon \scrF \leadsto \restr{\scrF}{B_\epsilon}$ induces an isomorphism on cohomology and establishes \eqref{cohom11}.
Indeed, for $\epsilon'<\epsilon$, the restriction cohomomorphism from $\restr{\scrF}{B_\epsilon}$ to $\restr{\scrF}{B_{\epsilon'}}$ is an isomorphism, again giving an isomorphism of cohomology. So,
$$H^*\big(\R;\scrF\big) \approx \varinjlim_{\epsilon>0} H^*\big(B_\epsilon; \restr{\scrF}{B_\epsilon}\big).$$
We can compute the limit using \cite[Theorem 10.6]{Bredon}:
$$\varinjlim_{\epsilon>0} H^*\big(B_\epsilon; \restr{\scrF}{B_\epsilon}\big) \approx H^*\big(\{\lambda_0\}; \restr{\scrF}{\{\lambda_0\}}\big).$$
Since $\restr{\scrF}{\{\lambda_0\}}$ is flasque (restriction maps are surjective), it is acyclic, completing the proof.
\eproof

The same results hold for $\mathrm{\Lambda} = I$, a bounded, or semi-bounded interval. 
In the applications below $\mathrm{\Lambda}$ is typically the real line.

\begin{lemma}\label{lemma:pitchacyclic2}
Let $\scrF$ be a sheaf of rings on $\mathrm{\Lambda}$ and let $\mathrm{\Lambda}' \xhookrightarrow{i} \mathrm{\Lambda}$.
Assume that $\scrF$ and $\scrF|_{\mathrm{\Lambda}'}$ are acyclic. 
If
\begin{enumerate}
    \item [(i)] $i_*^0\colon H^0(\mathrm{\Lambda};\scrA^\obphi) \to H^0(\mathrm{\Lambda}';\scrA^\obphi)$ is injective, then $\im i_*^0\cong H^0(\mathrm{\Lambda};\scrF)$ and
    \[
H^1(\mathrm{\Lambda},\mathrm{\Lambda}';\scrF) \cong \frac{H^0(\mathrm{\Lambda}';\scrF)}{\im i_*^0},\quad\text{and}\quad
H^k(\mathrm{\Lambda},\mathrm{\Lambda}';\scrF)\cong 0,\quad\text{for}\quad k\neq 1;
\]
\item [(ii)] $i_*^0\colon H^0(\mathrm{\Lambda};\scrA^\obphi) \to H^0(\mathrm{\Lambda}';\scrA^\obphi)$ is surjective, then
\[
H^0(\mathrm{\Lambda},\mathrm{\Lambda}';\scrF) \cong \ker i_*^0,\quad\text{and}\quad
H^k(\mathrm{\Lambda},\mathrm{\Lambda}';\scrF)\cong 0,\quad\text{for}\quad k\neq 0.
\]
\end{enumerate}
\end{lemma}

\proof
As before for triple $(\mathrm{\Lambda}',\varnothing) \xhookrightarrow{i}(\mathrm{\Lambda},\varnothing) \xhookrightarrow{j} (\mathrm{\Lambda},\mathrm{\Lambda}')$ we have the long exact sequence,
\[
\begin{tikzcd}[column sep=large]
 0  \arrow{r}{\delta^0}  & 
 H^0(\mathrm{\Lambda},\mathrm{\Lambda}';\scrF) \arrow{r}{j_*^0}  &
 H^0(\mathrm{\Lambda};\scrF)  \arrow{r}{i_*^0} \ar[draw=none]{d}[name=X, anchor=center]{} &
 H^0(\mathrm{\Lambda}';\scrF) \ar[rounded corners,
            to path={ -- ([xshift=2ex]\tikztostart.east)
                      |- (X.center) \tikztonodes
                      -| ([xshift=-2ex]\tikztotarget.west)
                      -- (\tikztotarget)}]{dll}[at end]{\delta^1} \\ 
& H^1(\mathrm{\Lambda},\mathrm{\Lambda}';\scrF) \arrow{r}{j_*^1} & 
 H^1(\mathrm{\Lambda};\scrF) \arrow{r}{i_*^1} & 
 H^1(\mathrm{\Lambda}';\scrF) \arrow{r}{\delta^2} &
 \cdots.
\end{tikzcd}
\]
Since, by Lemma \ref{lemma:pitchacyclic}, $\scrF$ is acyclic  we obtain the truncated sequence
\begin{equation}
\label{exact11}
\begin{tikzcd}[column sep=small]
 0  \arrow{r}{\delta^0}  & 
 H^0(\mathrm{\Lambda},\mathrm{\Lambda}';\scrF) \arrow{r}{j_*^0}  &
 H^0(\mathrm{\Lambda};\scrF) \arrow{r}{i_*^0}  &
 H^0(\mathrm{\Lambda}';\scrF) \arrow{r}{\delta^1}  &
 H^1(\mathrm{\Lambda},\mathrm{\Lambda}'; \scrF) \arrow{r}{j_*^1}  &  
 0.
\end{tikzcd}
\end{equation}
Since $i_*^0$ is injective and thus $\ker i_*^0 =0 = \im j_*^0$.
Moreover, $\ker j_*^0 = \im \delta^0 =0$, which implies that $H^0(\mathrm{\Lambda},\mathrm{\Lambda}';\scrF) \cong 0$. Consequently, we have the short exact sequence
\[
\begin{tikzcd}[column sep=small]
0 \arrow{r}{j_*^0}  &
 H^0(\mathrm{\Lambda};\scrF) \arrow{r}{i_*^0}  &
 H^0(\mathrm{\Lambda}';\scrF) \arrow{r}{\delta^1}  &
 H^1(\mathrm{\Lambda},\mathrm{\Lambda}'; \scrF) \arrow{r}{j_*^1}  &  
 0,
\end{tikzcd}
\]
from which the result for $H^1(\mathrm{\Lambda},\mathrm{\Lambda}'\scrF)$ follows.
The cohomology $H^k(\mathrm{\Lambda},\mathrm{\Lambda}'\scrF)\cong 0$, for $k\ge 2$ follows from Lemma \ref{relshco1}, which completes the proof of (i).

As for (ii) we have the truncated exact sequence in \eqref{exact11}. Now $i_*^0$ is surjective
which implies that $\ker \delta^1 = \im i_*^0 = H^0(\mathrm{\Lambda}';\scrF)$. Therefore, 
$\ker j_*^1 = \im \delta^1 =0$ and thus $j_*^1$ is injective. Consequently, $H^1(\mathrm{\Lambda},\mathrm{\Lambda}'\scrF) \cong 0$. We now have the short exact sequence
\[
\begin{tikzcd}[column sep=small]
 0  \arrow{r}{\delta^0}  & 
 H^0(\mathrm{\Lambda},\mathrm{\Lambda}';\scrF) \arrow{r}{j_*^0}  &
 H^0(\mathrm{\Lambda};\scrF) \arrow{r}{i_*^0}  &
 H^0(\mathrm{\Lambda}';\scrF) \arrow{r}{\delta^1}  &
 0,
\end{tikzcd}
\]
which implies that $H^0(\mathrm{\Lambda},\mathrm{\Lambda}';\scrF) \cong \ker i_*^0$. The relative homology for $k\ge 1$ follows from Lemma \ref{relshco1}.
\eproof

\begin{remark}
The sheaf cohomology groups of the abelian attractor sheaf can be equipped with a cup product from the ring structure of the sheaf. We leave cup product computations and their interpretation for later work.
\end{remark}

\subsubsection{The pitchfork bifurcation}
Consider a parametrized dynamical system  on $X = \R\cup\{-\infty,\infty\}$, the 2-point compactification of $\R$, experiencing a pitchfork bifurcation, cf.\ Figure \ref{fig:pitchdiagram}.
The parametrized flow is defined via the differential equation
\[
\dot x = \lambda x- x^3, \quad x\in \R,~\lambda \in \R.
\]
Before the bifurcation point at $\lambda_0=0$
there are two repelling fixed points at $+\infty$ and $-\infty$ and
a single attracting fixed point at $x=0$. After $\lambda_0=0$, there are two additional attracting
fixed points $x=\pm x_\lambda$ and $x=0$ has changed to a repelling fixed point. 
We fix a parametrization:
$$
\obpsi\colon \mathrm{\Lambda} \to \sDS(\T,X), 
$$
where $\mathrm{\Lambda} = \R$ is parameter space, $\T = \R^+$ is the time space and $X$ is the 2-point compactification of $\R$.

%%%%%%%%
%%%%%%%
%\input{pitchfork_bifurcation.tex}
\tikzset{every picture/.style={line width=0.75pt}} %set default line width to 0.75pt        
\begin{figure}[]
\label{fig:pitchdiagram}
    \centering

\begin{tikzpicture}[x=0.75pt,y=0.75pt,yscale=-1.6,xscale=2]
%uncomment if require: \path (0,193); %set diagram left start at 0, and has height of 193

%Shape: Rectangle [id:dp49106859308396444] 
\draw  [color={rgb, 255:red, 200; green, 200; blue, 200 }  ,draw opacity=1 ][fill={rgb, 255:red, 200; green, 200; blue, 200 }  ,fill opacity=1 ] (235.2,113.97) -- (300.2,113.97) -- (300.2,139.97) -- (235.2,139.97) -- cycle ;
%Straight Lines [id:da3819050206335941] 
\draw    (168.48,45.27) -- (168.48,76.27) ;
\draw [shift={(168.48,78.27)}, rotate = 270] [color={rgb, 255:red, 0; green, 0; blue, 0 }  ][line width=0.75]    (4.37,-1.32) .. controls (2.78,-0.56) and (1.32,-0.12) .. (0,0) .. controls (1.32,0.12) and (2.78,0.56) .. (4.37,1.32)   ;
%Curve Lines [id:da06370681231472974] 
\draw [line width=1.5]    (204.12,95) .. controls (204.12,127.7) and (261.2,129.63) .. (303.45,129.7) ;
%Straight Lines [id:da153626188550814] 
\draw    (256.18,75.27) -- (256.18,85.27) ;
\draw [shift={(256.18,73.27)}, rotate = 90] [color={rgb, 255:red, 0; green, 0; blue, 0 }  ][line width=0.75]    (4.37,-1.32) .. controls (2.78,-0.56) and (1.32,-0.12) .. (0,0) .. controls (1.32,0.12) and (2.78,0.56) .. (4.37,1.32)   ;
%Straight Lines [id:da5486206727052367] 
\draw    (255.97,40.27) -- (255.97,50.27) ;
\draw [shift={(255.97,52.27)}, rotate = 270] [color={rgb, 255:red, 0; green, 0; blue, 0 }  ][line width=0.75]    (4.37,-1.32) .. controls (2.78,-0.56) and (1.32,-0.12) .. (0,0) .. controls (1.32,0.12) and (2.78,0.56) .. (4.37,1.32)   ;
%Straight Lines [id:da5244075422558626] 
\draw    (255.97,153.27) -- (255.97,141.27) ;
\draw [shift={(255.97,139.27)}, rotate = 90] [color={rgb, 255:red, 0; green, 0; blue, 0 }  ][line width=0.75]    (4.37,-1.32) .. controls (2.78,-0.56) and (1.32,-0.12) .. (0,0) .. controls (1.32,0.12) and (2.78,0.56) .. (4.37,1.32)   ;
%Straight Lines [id:da2439554257747777] 
\draw [color={rgb, 255:red, 155; green, 155; blue, 155 }  ,draw opacity=1 ][line width=1.5]  [dash pattern={on 5.63pt off 4.5pt}]  (140,24.47) -- (303.2,24.63) ;
%Straight Lines [id:da9174572259104046] 
\draw [color={rgb, 255:red, 155; green, 155; blue, 155 }  ,draw opacity=1 ][line width=1.5]  [dash pattern={on 5.63pt off 4.5pt}]  (140,165.47) -- (303.2,165.63) ;
%Straight Lines [id:da2319565981327334] 
\draw [color={rgb, 255:red, 155; green, 155; blue, 155 }  ,draw opacity=0.25 ]   (204.04,32.37) -- (204.2,157.63) ;
%Straight Lines [id:da3421915897119274] 
\draw [color={rgb, 255:red, 0; green, 0; blue, 0 }  ,draw opacity=1 ][line width=1.5]  [dash pattern={on 5.63pt off 4.5pt}]  (204.2,94.63) -- (303.04,94.37) ;
%Curve Lines [id:da8291341209346833] 
\draw [line width=1.5]    (204.12,95) .. controls (204.12,65.7) and (265.28,62.27) .. (304.53,63) ;
%Straight Lines [id:da5914151511604018] 
\draw [color={rgb, 255:red, 155; green, 155; blue, 155 }  ,draw opacity=0.35 ]   (235.2,99.63) -- (300.2,99.63) ;
%Straight Lines [id:da6445650008462942] 
\draw    (256.18,114.27) -- (256.18,104.27) ;
\draw [shift={(256.18,116.27)}, rotate = 270] [color={rgb, 255:red, 0; green, 0; blue, 0 }  ][line width=0.75]    (4.37,-1.32) .. controls (2.78,-0.56) and (1.32,-0.12) .. (0,0) .. controls (1.32,0.12) and (2.78,0.56) .. (4.37,1.32)   ;
%Straight Lines [id:da9240916428013441] 
\draw [line width=1.5]    (141.2,94.77) -- (204.2,94.63) ;
%Straight Lines [id:da15741559106699943] 
\draw    (168.48,143.27) -- (168.48,113.27) ;
\draw [shift={(168.48,111.27)}, rotate = 90] [color={rgb, 255:red, 0; green, 0; blue, 0 }  ][line width=0.75]    (4.37,-1.32) .. controls (2.78,-0.56) and (1.32,-0.12) .. (0,0) .. controls (1.32,0.12) and (2.78,0.56) .. (4.37,1.32)   ;

% Text Node
\draw (311,20.73) node [anchor=north west][inner sep=0.75pt]  [font=\large]  {$\infty $};
% Text Node
\draw (307,161.73) node [anchor=north west][inner sep=0.75pt]  [font=\large]  {$-\infty $};
% Text Node
\draw (311,91.73) node [anchor=north west][inner sep=0.75pt]  [font=\large]  {$0$};
% Text Node
\draw (141.2,79.03) node [anchor=north west][inner sep=0.75pt]  [font=\large]  {$\mathrm{\Lambda} $};
% Text Node
\draw (190,143.73) node [anchor=north west][inner sep=0.75pt]  [font=\large]  {$X$};
% Text Node
\draw (269.7,103.03) node [anchor=north west][inner sep=0.75pt]  [font=\large]  {$\mathrm{\Lambda} '$};
% Text Node
\draw (287,116.73) node [anchor=north west][inner sep=0.75pt]  [font=\large]  {$U$};
\end{tikzpicture}

\caption{In the pitchfork bifurcation, the section on $\mathrm\Lambda'\subset \mathrm\Lambda$ defined by $\sigma(\lambda) = \big(\lambda,\phi^\lambda, \omega_{\phi^\lambda}(U)\big)$ fails to extend globally. }
\end{figure}
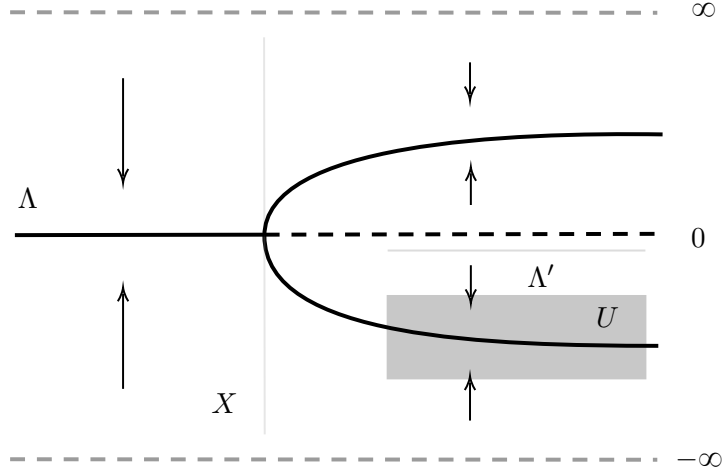\label{pitchfork11}
%%%%%%%%%
%%%%%%%%%%%

\begin{lemma}\label{lem:pitchglobal}
The global sections of the abelian attractor sheaf for the normal form pitchfork bifurcation are $H^0(\mathrm{\Lambda};\scrA^\obpsi)\cong \mathrm{\Gamma}\bigl(\scrA^\obpsi\bigr) \cong \Z_2^3$. 
\end{lemma}
\proof
First, consider the global sections of the attractor lattice sheaf $\obpsi^{-1}\scrS^\sAtt$, which are uniquely characterized by their value at the bifurcation point $\lambda=0$. Assigning $\emptyset$ and $X$ to every parameter value gives the bottom and top elements of $\mathrm{\Gamma}(\obpsi^{-1}\scrS^\sAtt)$. Then, we have the section which assigns $\lambda\leq0$ the sole attracting fixed point $x=0$, and $\lambda>0$ the interval $[-x_\lambda, x_\lambda]$. Finally, there are two sections which are $[-\infty, 0]$ and $[0,\infty]$ for $\lambda\leq 0$, but $[-\infty, x_\lambda]$ and $[-x_\lambda, \infty]$ for $\lambda>0$ respectively. This yields a five element lattice with three join irreducible elements: the latter three sections. Applying the boolean ring functor to this lattice yields $\Z_2^3$.
\eproof \\
In later results we will omit the above types of computations.
\begin{proposition}\label{prop:pitchrelative}
Let $\mathrm{\Lambda}' := [a,\infty)$. If $a>0$, then $H^k(\mathrm{\Lambda}, \mathrm{\Lambda}'; \scrA^\obpsi)\cong \Z^2_2$ for $k=1$, and $H^k(\mathrm{\Lambda},\mathrm{\Lambda}'; \scrA^\obpsi)=0 $ otherwise. When $a\leq 0$, all relative cohomology groups vanish.
\end{proposition}

\proof
Lemma \ref{lem:pitchglobal} gives $\mathrm{\Gamma}\bigl(\scrA^\obpsi\bigr) \cong \Z_2^3$. For $a>0$, we have $H^0(\mathrm{\Lambda}';\scrA^\obpsi)\cong\mathrm{\Gamma}\big(\restr{\scrA^\obpsi}{\mathrm{\Lambda}'}\big)\cong \Z^5_2$, and injectivity of $i_*^0$,
which by Lemma \ref{lemma:pitchacyclic2}(i) yields $H^1(\mathrm{\Lambda},\mathrm{\Lambda}'; \scrA^\obpsi) \cong \Z^2_2$. 
For $a\le 0$, we have $H^0(\mathrm{\Lambda}';\scrA^\obpsi)\cong \mathrm{\Gamma}\big(\restr{\scrA^\obpsi}{\mathrm{\Lambda}'}\big)\cong \Z_2^3$, 
which implies $H^1(\mathrm{\Lambda},\mathrm{\Lambda}'; \scrA^\obpsi) =0$. 
Since by Lemma \ref{lemma:pitchacyclic} both $\scrA^\obpsi$ and $\scrA^\obpsi|_{\mathrm{\Lambda}'}$\footnote{If $\mathrm{\Lambda}'$ does not contain a bifurcation point then acyclicity follows from the fact that $\scrA^\obpsi|_{\mathrm{\Lambda}'}$ is a constant sheaf on a contractible manifold.} are acyclic,
the higher order relative cohomology groups vanish by  Lemma \ref{lemma:pitchacyclic2}(i). 
\eproof

\begin{proposition}\label{prop:pitchrelative2}
Let $\mathrm{\Lambda}' := (-\infty,a]$. Then, $H^k(\mathrm{\Lambda}, \mathrm{\Lambda}'; \scrA^\obpsi)\cong 0$ for all $k$
and for all $a\in \R$.
\end{proposition}

\proof
Note that $\mathrm{\Gamma}\bigl(\scrA^\obpsi\bigr) \cong \mathrm{\Gamma}\big(\restr{\scrA^\obpsi}{\mathrm{\Lambda}'}\big)\cong \Z^3_2$
for all $a\in \R$ (the same computations as in Lemma \ref{lem:pitchglobal} apply to $\Lambda'$). Therefore, $H^0(\mathrm{\Lambda};\scrA^\obpsi)\cong H^0(\mathrm{\Lambda}';\scrA^\obpsi)$ for all $a\in \R$ and thus by Lemma \ref{lemma:pitchacyclic2}(i) $H^k(\mathrm{\Lambda},\mathrm{\Lambda}'; \scrA^\obpsi) \cong 0$
for all $k$.
\eproof

\begin{theorem}
\label{thm:pitchchm}
Let $\obphi$ be a parametrized dynamical system over $\mathrm{\Lambda}$ conjugate to the above canonical parametrization $\obpsi$ for the pitchfork bifurcation. Then,
\[
\scrA^\obphi ~\text{~ is acyclic and ~}~ H^0(\mathrm{\Lambda};\scrA^\obphi) \cong\Z_2^3.
\]
Moreover,
there exists a value $\lambda_0\in \R$ such that
\[
H^k\bigl(\mathrm{\Lambda}, \mathrm{\Lambda}';\scrA^\obphi\bigr) \cong \begin{cases} \Z_2^2 &\textnormal{if } k = 1\textnormal{ and } a>\lambda_0; \\
0 & \textnormal{if $k\neq  1$ or $a\le \lambda_0$,} \end{cases}
\]
where $\mathrm{\Lambda}' = [a,\infty)$,
Furthermore, for $\mathrm{\Lambda}' := (-\infty,a]$, then
$H^k\bigl(\mathrm{\Lambda}, \mathrm{\Lambda}';\scrA^\obphi\bigr) \cong 0$ for all $k$ and for all $a\in \R$.
\end{theorem}

\proof
This follows immediately from Theorem \ref{thm:CIT}, Lemma \ref{lemma:pitchacyclic}, and Propositions \ref{prop:pitchrelative} and \ref{prop:pitchrelative2}. 
\eproof

This theorem can be applied locally in parameter space. If $\obphi\colon \R \to \sDS(\R; I)$ is some parametrized dynamical system such that  $\obphi$ experiences a pitchfork bifurcation on an open set $U$, then $\restr{\scrA^\obphi}{U}$ has the above cohomology groups. 
Another important observation is that the relative cohomology 
in the example below is the same for a local pitchfork bifuraction.
\begin{example}
\label{pitchagain}
Let $\obphi$ be a parametrized flow over $\mathrm{\Lambda} = \R$ on the interval $X=[-1,1]$ 
with a single attracting fixed point at $x=0$ for $\lambda\le 0$. 
This system is a semi-flow with $\T=\R^+$.
For $\lambda\ge 0$ the system undergoes a pitchfork bifurcation with two branches $\pm x_\lambda$ of attracting fixed points converging to $\pm 1$ respectively as $\lambda\to +\infty$, cf.\ Figure \ref{fig:pitchdiagram}.
If we repeat the analysis in Propositions \ref{prop:pitchrelative} and \ref{prop:pitchrelative2}
the sheaf cohomology over $\mathrm{\Lambda}$ is different: $\scrA^\obphi$ is acyclic and $H^0(\mathrm{\Lambda};\scrA^\obphi) \cong \Z_2$. On the other hand the relative sheaf cohomologies
$H^k\bigl(\R, [a,\infty);\scrA^\obphi\bigr)$ and $H^k\bigl(\R, (-\infty,a];\scrA^\obphi\bigr)$ are the same.
\end{example}

\subsubsection{The saddle-node bifurcation}\label{sn1}
Consider a parametrized dynamical system  on  $X = \R\cup\{-\infty,\infty\}$, the 2-point compactification of $\R$, experiencing a saddle-node bifurcation. 
The parametrized flow is defined via the differential equation
\[
\dot x = \lambda - x^2, \quad x\in \R,~\lambda \in \R,
\]
and $+\infty$ and $-\infty$ are a repelling and attracting fixed point respectively.
Before the bifurcation point at $\lambda_0=0$, the entire interval flows from $+\infty$ to $-\infty$. After $\lambda_0=0$, there is an additional attracting and repelling fixed point. 
We fix a parametrization:
$$
\obpsi\colon \mathrm{\Lambda} \to \sDS(\T,X), 
$$
where $\mathrm{\Lambda} = \R$ is parameter space, $\T = \R$ is the time space and $X$ is the 2-point compactification of $\R$.
Lemma \ref{lemma:pitchacyclic} again shows that the attractor sheaf $\scrA^\obpsi$ has vanishing higher order cohomology, but relative cohomology recognizes the bifurcations.

%%%%%%%%%
%%%%%%%%%%
%\input{saddle-node}
\tikzset{every picture/.style={line width=0.5pt}} %set default line width to 0.75pt        
\begin{figure}[]
\label{fig:saddlediagram}
    \centering

\begin{tikzpicture}[x=0.75pt,y=0.75pt,yscale=-1.6,xscale=2]
%uncomment if require: \path (0,193); %set diagram left start at 0, and has height of 193

%Shape: Rectangle [id:dp18603063720976754] 
\draw  [color={rgb, 255:red, 200; green, 200; blue, 200 }  ,draw opacity=1 ][fill={rgb, 255:red, 200; green, 200; blue, 200 }  ,fill opacity=1 ] (234.2,44.3) -- (299.2,44.3) -- (299.2,75.3) -- (234.2,75.3) -- cycle ;
%Straight Lines [id:da7867062362381073] 
\draw    (167.48,75.93) -- (167.48,106.93) ;
\draw [shift={(167.48,108.93)}, rotate = 270] [color={rgb, 255:red, 0; green, 0; blue, 0 }  ][line width=0.75]    (4.37,-1.32) .. controls (2.78,-0.56) and (1.32,-0.12) .. (0,0) .. controls (1.32,0.12) and (2.78,0.56) .. (4.37,1.32)   ;
%Curve Lines [id:da860566391701553] 
\draw [line width=1.5]  [dash pattern={on 5.63pt off 4.5pt}]  (203.12,91.67) .. controls (203.12,124.37) and (260.2,130.3) .. (302.45,130.37) ;
%Straight Lines [id:da7082582014483886] 
\draw    (255.18,80.93) -- (255.18,100.93) ;
\draw [shift={(255.18,78.93)}, rotate = 90] [color={rgb, 255:red, 0; green, 0; blue, 0 }  ][line width=0.75]    (4.37,-1.32) .. controls (2.78,-0.56) and (1.32,-0.12) .. (0,0) .. controls (1.32,0.12) and (2.78,0.56) .. (4.37,1.32)   ;
%Straight Lines [id:da9852898491111931] 
\draw    (254.97,28.93) -- (254.97,38.93) ;
\draw [shift={(254.97,40.93)}, rotate = 270] [color={rgb, 255:red, 0; green, 0; blue, 0 }  ][line width=0.75]    (4.37,-1.32) .. controls (2.78,-0.56) and (1.32,-0.12) .. (0,0) .. controls (1.32,0.12) and (2.78,0.56) .. (4.37,1.32)   ;
%Straight Lines [id:da8781402030669732] 
\draw    (254.97,137.93) -- (254.97,147.93) ;
\draw [shift={(254.97,149.93)}, rotate = 270] [color={rgb, 255:red, 0; green, 0; blue, 0 }  ][line width=0.75]    (4.37,-1.32) .. controls (2.78,-0.56) and (1.32,-0.12) .. (0,0) .. controls (1.32,0.12) and (2.78,0.56) .. (4.37,1.32)   ;
%Straight Lines [id:da3796201293143693] 
\draw [color={rgb, 255:red, 155; green, 155; blue, 155 }  ,draw opacity=1 ][line width=1.5]  [dash pattern={on 5.63pt off 4.5pt}]  (139,21.13) -- (302.2,21.3) ;
%Straight Lines [id:da1536379701333701] 
\draw [color={rgb, 255:red, 155; green, 155; blue, 155 }  ,draw opacity=1 ][line width=1.5]    (139,162.13) -- (302.2,162.3) ;
%Straight Lines [id:da1934991117956415] 
\draw [color={rgb, 255:red, 155; green, 155; blue, 155 }  ,draw opacity=0.25 ]   (203.04,29.03) -- (203.2,154.3) ;
%Straight Lines [id:da6869381584773114] 
\draw [color={rgb, 255:red, 155; green, 155; blue, 155 }  ,draw opacity=0.25 ]   (139.2,91.3) -- (302.04,91.03) ;
%Curve Lines [id:da6107095098143821] 
\draw [line width=1.5]    (203.12,91.67) .. controls (203.12,62.37) and (264.28,53.93) .. (303.53,54.67) ;
%Straight Lines [id:da3970858415367293] 
\draw [color={rgb, 255:red, 155; green, 155; blue, 155 }  ,draw opacity=0.35 ]   (234.2,103.3) -- (299.2,103.3) ;

% Text Node
\draw (310,17.4) node [anchor=north west][inner sep=0.75pt]  [font=\large]  {$\infty $};
% Text Node
\draw (306,158.4) node [anchor=north west][inner sep=0.75pt]  [font=\large]  {$-\infty $};
% Text Node
\draw (310,88.4) node [anchor=north west][inner sep=0.75pt]  [font=\large]  {$0$};
% Text Node
\draw (139.2,77.7) node [anchor=north west][inner sep=0.75pt] [font=\large]  {$\mathrm{\Lambda} $};
% Text Node
\draw (189,140.4) node [anchor=north west][inner sep=0.75pt]  [font=\large]  {$X$};
% Text Node
\draw (268.7,107.7) node [anchor=north west][inner sep=0.75pt]  [font=\large]  {$\mathrm{\Lambda} '$};
% Text Node
\draw (286,61.4) node [anchor=north west][inner sep=0.75pt]  [font=\large]  {$U$};

\end{tikzpicture}

\caption{A saddle-node bifurcation. The section on $\mathrm\Lambda'$ defined by $\sigma(\lambda) = \big(\lambda,\phi^\lambda, \omega_{\phi^\lambda}(U)\big)$ fails to extend globally. }
\end{figure}
%%%%%%%%%%
%%%%%%%%%%

\begin{proposition}\label{prop:saddlerelative}
Let $\mathrm{\Lambda}' = [a,\infty)$. If $a>0$, then
$H^k(\mathrm{\Lambda}, \mathrm{\Lambda}'; \scrA^\obpsi)\cong \Z_2$ for $k=1$, and vanishes otherwise. When $a\leq 0$, then $H^k(\mathrm{\Lambda}, \mathrm{\Lambda}'; \scrA^\obpsi)=0$ for all $k$.
\end{proposition}

\proof
The global sections are $H^0(\mathrm{\Lambda};\scrA^\obpsi)\cong \mathrm{\Gamma}\bigl(\scrA^\obpsi\bigr) \cong \Z_2^3$.
For $a>0$, we have $H^0(\mathrm{\Lambda}';\scrA^\obpsi)\cong\mathrm{\Gamma}\big(\restr{\scrA^\obpsi}{\mathrm{\Lambda}'}\big)\cong \Z_2^4$. The injectivity of $i_*^0$ and
 Lemma \ref{lemma:pitchacyclic2}(i) yields $H^1(\mathrm{\Lambda},\mathrm{\Lambda}'; \scrA^\obpsi) \cong \Z_2$. 
For $a\le 0$, we have $H^0(\mathrm{\Lambda}';\scrA^\obpsi)\cong \mathrm{\Gamma}\big(\restr{\scrA^\obpsi}{\mathrm{\Lambda}'}\big)\cong \Z_2^3$, 
which implies $H^1(\mathrm{\Lambda},\mathrm{\Lambda}'; \scrA^\obpsi) =0$. As before the higher order relative cohomology groups vanish by  Lemma \ref{lemma:pitchacyclic2}(i). 
\eproof

\begin{proposition}
\label{prop:saddlerelative2}
Let $\mathrm{\Lambda}' = (-\infty,a]$. If $a>0$, then 
$H^k(\mathrm{\Lambda}, \mathrm{\Lambda}'; \scrA^\obpsi)\cong 0$ for all $k$. When $a< 0$, then $H^0(\mathrm{\Lambda}, \mathrm{\Lambda}'; \scrA^\obpsi)\cong \Z_2$ and vanishes otherwise.
\end{proposition}

\proof
As before the global sections are $H^0(\mathrm{\Lambda};\scrA^\obpsi)\cong \mathrm{\Gamma}\bigl(\scrA^\obpsi\bigr) \cong \Z_2^3$.
For $a>0$, we have $H^0(\mathrm{\Lambda}';\scrA^\obpsi)\cong\mathrm{\Gamma}\big(\restr{\scrA^\obpsi}{\mathrm{\Lambda}'}\big)\cong \Z_2^3$. The injectivity of $i_*^0$ and
 Lemma \ref{lemma:pitchacyclic2}(i) yields $H^1(\mathrm{\Lambda},\mathrm{\Lambda}'; \scrA^\obpsi) \cong 0$. 
For $a< 0$, we have $H^0(\mathrm{\Lambda}';\scrA^\obpsi)\cong \mathrm{\Gamma}\big(\restr{\scrA^\obpsi}{\mathrm{\Lambda}'}\big)\cong \Z_2^2$. The surjectivity of $i_*^0$ and Lemma \ref{lemma:pitchacyclic2}(ii)
then implies that $H^0(\mathrm{\Lambda},\mathrm{\Lambda}'; \scrA^\obpsi) \cong \Z_2$. The higher order relative cohomology groups vanish by  Lemma \ref{lemma:pitchacyclic2}(i) and (ii). 
\eproof

\begin{theorem}
\label{saddlenode3}
Let $\obphi$ be a parametrized dynamical system over $\mathrm{\Lambda}$ conjugate to the above canonical parametrization $\obpsi$ for the saddle-node bifurcation. Then, 
\[
\scrA^\obphi ~\text{~ is acyclic and~}~ H^0(\mathrm{\Lambda};\scrA^\obphi) \cong\Z_2^3.
\]
Moreover,
there exists a value $\lambda_0\in \R$ such that
$$
H^k(\mathrm{\Lambda}, \mathrm{\Lambda}';\scrA^\obphi) \cong \begin{cases} \Z_2 &\textnormal{if } k = 1\textnormal{ and } a>\lambda_0 \\
0 & k\neq 1, \textnormal{or~} a\le \lambda_0, \end{cases} \quad \text{with}\quad\mathrm{\Lambda}' = [a,\infty),
$$
$$
H^k(\mathrm{\Lambda}, \mathrm{\Lambda}';\scrA^\obphi) \cong \begin{cases} \Z_2 &\textnormal{if } k = 0\textnormal{ and } a< \lambda_0 \\
0 & k\neq 0, \textnormal{or~} a\ge \lambda_0, \end{cases}\quad\text{with}\quad \mathrm{\Lambda}' = (-\infty,a].
$$
\end{theorem}

\proof
Apply Theorem \ref{thm:CIT}, Lemma \ref{lemma:pitchacyclic} and Propositions \ref{prop:saddlerelative}
and \ref{prop:saddlerelative2}. 
\eproof

\begin{remark}\label{rem:morserelative}
The generator of $H^k\bigl(\R, (-\infty,a];\scrA^\obpsi\bigr)$ when $a<{\lambda_0}$ is the sum of two global sections of attractors: the bottom fixed point and the maximal attractor. The two coincide before the bifurcation point, which leaves their sum zero. Afterwards, however, this corresponds to the Morse set between the top two fixed points. Continuation of this Morse set to the empty set via a global section yields nontrivial relative cohomology.
\end{remark}

\begin{example}
Consider a saddle-node bifurcation in the system described in Figure \ref{fig:saddlediagram}.
 We impose an attracting fixed point at the bottom of Figure \ref{fig:saddlediagram}, such that we may restrict phase space to a forward-invariant compact interval $X=[1,0]$. Call this parametrized dynamical system $\obphi\colon \R \to \sDS(\R^+,X)$. Lemma \ref{lemma:pitchacyclic} again shows that $\scrA^\obphi$ has vanishing higher cohomology. However, $H^0(\mathrm{\Lambda};\scrA^\obphi) \cong \Z_2^2$ which differs from the above example. The relative cohomology groups are the same as in the above example as is the case for the pitchfork bifurcation.
\end{example}

\subsubsection{The transcritical bifurcation}
Consider a parametrized dynamical system  on  $X = \R\cup\{-\infty,\infty\}$, the 2-point compactification of $\R$, experiencing a transcritical  bifurcation. 
The parametrized flow is defined via the differential equation
\[
\dot x = \lambda x - x^2, \quad x\in \R,~\lambda \in \R,
\]
and $+\infty$ and $-\infty$ are a repelling and attracting fixed points respectively.
As before we fix a parametrization:
$$
\obpsi\colon \mathrm{\Lambda} \to \sDS(\T,X), 
$$
where $\mathrm{\Lambda} = \R$ is parameter space, $\T = \R$ is the time space and $X$ is the 2-point compactification of $\R$.
Lemma \ref{lemma:pitchacyclic} again shows that the attractor sheaf $\scrA^\obpsi$ has vanishing higher order cohomology.
%%%%%%%%%%%%
%%%%%%%%%%%%%%
%\input{transcritical.tex}
\begin{figure}[]
\label{fig:transcriticaldiagram}
    \centering

\tikzset{every picture/.style={line width=0.75pt}} %set default line width to 0.75pt        

\begin{tikzpicture}[x=0.75pt,y=0.75pt,yscale=-1.6,xscale=2]
%uncomment if require: \path (0,193); %set diagram left start at 0, and has height of 193

%Shape: Rectangle [id:dp45014250149778656] 
\draw  [color={rgb, 255:red, 200; green, 200; blue, 200 }  ,draw opacity=1 ][fill={rgb, 255:red, 200; green, 200; blue, 200 }  ,fill opacity=1 ] (255.75,35.63) -- (297.01,35.63) -- (297.01,90.63) -- (255.75,90.63) -- cycle ;
%Straight Lines [id:da10989690943720443] 
\draw    (165.28,44.27) -- (165.28,75.27) ;
\draw [shift={(165.28,77.27)}, rotate = 270] [color={rgb, 255:red, 0; green, 0; blue, 0 }  ][line width=0.75]    (4.37,-1.32) .. controls (2.78,-0.56) and (1.32,-0.12) .. (0,0) .. controls (1.32,0.12) and (2.78,0.56) .. (4.37,1.32)   ;
%Straight Lines [id:da17741389867464197] 
\draw    (276.36,75.27) -- (276.36,85.27) ;
\draw [shift={(276.36,73.27)}, rotate = 90] [color={rgb, 255:red, 0; green, 0; blue, 0 }  ][line width=0.75]    (4.37,-1.32) .. controls (2.78,-0.56) and (1.32,-0.12) .. (0,0) .. controls (1.32,0.12) and (2.78,0.56) .. (4.37,1.32)   ;
%Straight Lines [id:da5273677630421508] 
\draw    (276.15,40.27) -- (276.15,50.27) ;
\draw [shift={(276.15,52.27)}, rotate = 270] [color={rgb, 255:red, 0; green, 0; blue, 0 }  ][line width=0.75]    (4.37,-1.32) .. controls (2.78,-0.56) and (1.32,-0.12) .. (0,0) .. controls (1.32,0.12) and (2.78,0.56) .. (4.37,1.32)   ;
%Straight Lines [id:da40553538397045075] 
\draw [color={rgb, 255:red, 155; green, 155; blue, 155 }  ,draw opacity=1 ][line width=1.5]  [dash pattern={on 5.63pt off 4.5pt}]  (140.03,24.47) -- (303.2,24.63) ;
%Straight Lines [id:da4906619873259622] 
\draw [color={rgb, 255:red, 155; green, 155; blue, 155 }  ,draw opacity=1 ][line width=1.5]    (140.03,165.47) -- (303.2,165.63) ;
%Straight Lines [id:da7166340898306639] 
\draw [color={rgb, 255:red, 155; green, 155; blue, 155 }  ,draw opacity=0.25 ]   (222.59,32.37) -- (222.75,157.63) ;
%Straight Lines [id:da8318196807068216] 
\draw [color={rgb, 255:red, 0; green, 0; blue, 0 }  ,draw opacity=1 ][line width=1.5]  [dash pattern={on 5.63pt off 4.5pt}]  (221.72,94.63) -- (304.04,94.37) ;
%Straight Lines [id:da07236381270171555] 
\draw [color={rgb, 255:red, 155; green, 155; blue, 155 }  ,draw opacity=0.35 ]   (255.75,99.63) -- (298.11,99.63) ;
%Straight Lines [id:da30593223770292743] 
\draw [line width=1.5]    (140.15,95) -- (221.72,94.63) ;
%Straight Lines [id:da5655853575205821] 
\draw [line width=1.5]    (222.75,94.63) -- (303.2,52.43) ;
%Straight Lines [id:da28692437317260333] 
\draw [line width=1.5]  [dash pattern={on 5.63pt off 4.5pt}]  (140.23,133.63) -- (220.68,95.63) ;
%Straight Lines [id:da7193040391992437] 
\draw    (164.99,102.63) -- (164.99,112.63) ;
\draw [shift={(164.99,100.63)}, rotate = 90] [color={rgb, 255:red, 0; green, 0; blue, 0 }  ][line width=0.75]    (4.37,-1.32) .. controls (2.78,-0.56) and (1.32,-0.12) .. (0,0) .. controls (1.32,0.12) and (2.78,0.56) .. (4.37,1.32)   ;
%Straight Lines [id:da4219323976425515] 
\draw    (276.36,116.27) -- (276.36,147.27) ;
\draw [shift={(276.36,149.27)}, rotate = 270] [color={rgb, 255:red, 0; green, 0; blue, 0 }  ][line width=0.75]    (4.37,-1.32) .. controls (2.78,-0.56) and (1.32,-0.12) .. (0,0) .. controls (1.32,0.12) and (2.78,0.56) .. (4.37,1.32)   ;
%Straight Lines [id:da8270460086176468] 
\draw    (164.99,135.63) -- (164.99,145.63) ;
\draw [shift={(164.99,147.63)}, rotate = 270] [color={rgb, 255:red, 0; green, 0; blue, 0 }  ][line width=0.75]    (4.37,-1.32) .. controls (2.78,-0.56) and (1.32,-0.12) .. (0,0) .. controls (1.32,0.12) and (2.78,0.56) .. (4.37,1.32)   ;

% Text Node
\draw (311,20.73) node [anchor=north west][inner sep=0.75pt]  [font=\large]  {$\infty $};
% Text Node
\draw (307,161.73) node [anchor=north west][inner sep=0.75pt]  [font=\large]  {$-\infty $};
% Text Node
\draw (311,91.73) node [anchor=north west][inner sep=0.75pt]  [font=\large]  {$0$};
% Text Node
\draw (140.42,79.03) node [anchor=north west][inner sep=0.75pt]  [font=\large]  {$\mathrm{\Lambda} $};
% Text Node
\draw (208.29,142.73) node [anchor=north west][inner sep=0.75pt]  [font=\large]  {$X$};
% Text Node
\draw (282.24,103.03) node [anchor=north west][inner sep=0.75pt]  [font=\large]  {$\mathrm{\Lambda} '$};
% Text Node
\draw (284.62,77.73) node [anchor=north west][inner sep=0.75pt]  [font=\large]  {$U$};

\end{tikzpicture}

\caption{A trans-critical bifurcation. The section on $\mathrm\Lambda'$ defined by $\sigma(\lambda) = \big(\lambda,\phi^\lambda, \omega_{\phi^\lambda}(U)\big)$ fails to extend globally.}
\end{figure}
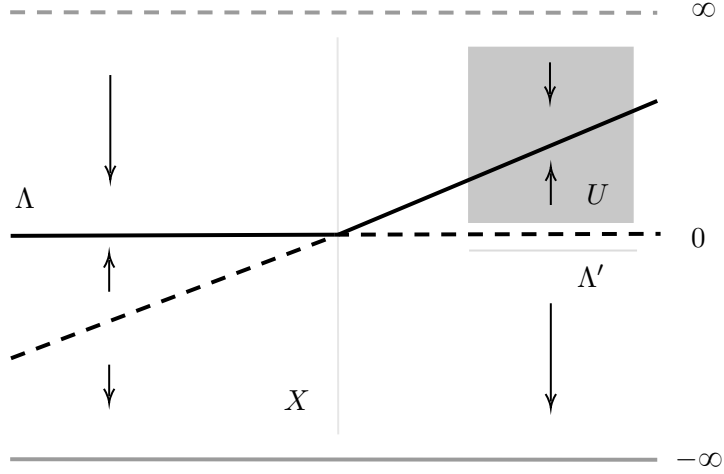
%%%%%%%%%%%%%
%%%%%%%%%%%
\begin{proposition}\label{prop:trans}
Let $\mathrm{\Lambda}' = [a,\infty)$. If $a>0$, then
$H^k(\mathrm{\Lambda}, \mathrm{\Lambda}'; \scrA^\obpsi)\cong \Z_2$ for $k=1$, and vanishes otherwise. When $a\leq 0$, then $H^k(\mathrm{\Lambda}, \mathrm{\Lambda}'; \scrA^\obpsi)=0$ for all $k$.
\end{proposition}

\proof
The global sections are $H^0(\mathrm{\Lambda};\scrA^\obpsi)\cong \mathrm{\Gamma}\bigl(\scrA^\obpsi\bigr) \cong \Z_2^3$.
For $a>0$, we have $H^0(\mathrm{\Lambda}';\scrA^\obpsi)\cong\mathrm{\Gamma}\big(\restr{\scrA^\obpsi}{\mathrm{\Lambda}'}\big)\cong \Z_2^4$. The injectivity of $i_*^0$ and
 Lemma \ref{lemma:pitchacyclic2}(i) yields $H^1(\mathrm{\Lambda},\mathrm{\Lambda}'; \scrA^\obpsi) \cong \Z_2$. 
For $a\le 0$, we have $H^0(\mathrm{\Lambda}';\scrA^\obpsi)\cong \mathrm{\Gamma}\big(\restr{\scrA^\obpsi}{\mathrm{\Lambda}'}\big)\cong \Z_2^3$, 
which implies $H^1(\mathrm{\Lambda},\mathrm{\Lambda}'; \scrA^\obpsi) =0$. As before the higher order relative cohomology groups vanish by  Lemma \ref{lemma:pitchacyclic2}(i). 
\eproof

\begin{proposition}
\label{prop:trans2}
Let $\mathrm{\Lambda}' = (-\infty,a]$. If $a\ge 0$, then 
$H^k(\mathrm{\Lambda}, \mathrm{\Lambda}'; \scrA^\obpsi)\cong 0$ for all $k$. When $a< 0$, then $H^1(\mathrm{\Lambda}, \mathrm{\Lambda}'; \scrA^\obpsi)\cong \Z_2$ and vanishes otherwise.
\end{proposition}

\proof
As before the global sections are $H^0(\mathrm{\Lambda};\scrA^\obpsi)\cong \mathrm{\Gamma}\bigl(\scrA^\obpsi\bigr) \cong \Z_2^3$.
For $a\ge 0$, we have $H^0(\mathrm{\Lambda}';\scrA^\obpsi)\cong\mathrm{\Gamma}\big(\restr{\scrA^\obpsi}{\mathrm{\Lambda}'}\big)\cong \Z_2^3$. The injectivity of $i_*^0$ and
 Lemma \ref{lemma:pitchacyclic2}(i) yields $H^1(\mathrm{\Lambda},\mathrm{\Lambda}'; \scrA^\obpsi) \cong 0$. 
For $a< 0$, we have $H^0(\mathrm{\Lambda}';\scrA^\obpsi)\cong \mathrm{\Gamma}\big(\restr{\scrA^\obpsi}{\mathrm{\Lambda}'}\big)\cong \Z_2^4$. The injectivity of $i_*^0$ and Lemma \ref{lemma:pitchacyclic2}(i)
then implies that $H^1(\mathrm{\Lambda},\mathrm{\Lambda}'; \scrA^\obpsi) \cong \Z_2$. The higher order relative cohomology groups vanish by  Lemma \ref{lemma:pitchacyclic2}(i). 
\eproof

\begin{theorem}
Let $\obphi$ be a parametrized dynamical system over $\mathrm{\Lambda}$ conjugate to the above canonical parametrization for the transcritical bifurcation. Then,
\[
\scrA^\obphi ~\text{~ is acyclic and~}~ H^0(\mathrm{\Lambda};\scrA^\obphi) \cong\Z_2^3.
\]
Moreover,
there exists a value $\lambda_0\in \R$ such that
$$
H^k(\mathrm{\Lambda}, \mathrm{\Lambda}';\scrA^\obphi) \cong \begin{cases} \Z_2 &\textnormal{if } k = 1\textnormal{ and } a>\lambda_0 \\
0 & k\neq 1, \textnormal{or~} a\le \lambda_0, \end{cases} \quad \text{with}\quad\mathrm{\Lambda}' = [a,\infty),
$$
$$
H^k(\mathrm{\Lambda}, \mathrm{\Lambda}';\scrA^\obphi) \cong \begin{cases} \Z_2 &\textnormal{if } k = 1\textnormal{ and } a< \lambda_0 \\
0 & k\neq 1, \textnormal{or~} a\ge \lambda_0, \end{cases}\quad\text{with}\quad \mathrm{\Lambda}' = (-\infty,a].
$$
\end{theorem}

\proof
Apply Theorem \ref{thm:CIT}, Lemma \ref{lemma:pitchacyclic} and Propositions \ref{prop:saddlerelative}
and \ref{prop:saddlerelative2}. 
\eproof

\begin{remark}
Note the subtle difference in the relative sheaf cohomology for the saddle-node and transcritical bifurcations. For the latter we only find relative cohomology at $k=1$ for different choices of $\mathrm{\Lambda}'$, as for the saddle-node we have cohomology at $k=0$ and $k=1$ for various choices of $\mathrm{\Lambda}'$.
\end{remark}

\subsection{One-parameter bifurcations at multiple parameter values}
\label{multiple}
In this subsection we consider a bifurcation that occur at multiple points.

\subsubsection{The S-shaped bifurcation}

Now we study the S-shaped bifurcation, as in Figure \ref{fig:sshaped}. 
Consider a parametrized dynamical system on  $X = \R\cup\{-\infty,\infty\}$, the 2-point compactification of $\R$, experiencing an S-shaped  bifurcation. The parametrized flow is defined via the differential equation
\[
\dot x = \lambda +  x - x^3, \quad x\in \R,~\lambda \in \R,
\]
and $+\infty$ and $-\infty$ are a repelling fixed points.
As before we fix a parametrization:
$$
\obpsi\colon \mathrm{\Lambda} \to \sDS(\T,X), 
$$
where $\mathrm{\Lambda} = \R$ is parameter space, $\T = \R$ is the time space and $X$ is the 2-point compactification of $\R$.
Here, there are two bifurcation points at $\lambda_1=-1$ and $\lambda_2=+1$. 
%%%%%%%%%%
%%%%%%%%%%
%\input{s-shaped}
\tikzset{every picture/.style={line width=0.75pt}} %set default line width to 0.75pt      
\begin{figure}\label{fig:sshaped}
    \centering

\begin{tikzpicture}[x=0.75pt,y=0.75pt,yscale=-1.6,xscale=2]
%uncomment if require: \path (0,193); %set diagram left start at 0, and has height of 193

%Shape: Rectangle [id:dp17452063161827935] 
\draw  [color={rgb, 255:red, 200; green, 200; blue, 200 }  ,draw opacity=1 ][fill={rgb, 255:red, 200; green, 200; blue, 200 }  ,fill opacity=1 ] (211.2,45.1) -- (295.01,45.1) -- (295.01,78.1) -- (211.2,78.1) -- cycle ;
%Straight Lines [id:da45122998134617864] 
\draw [color={rgb, 255:red, 155; green, 155; blue, 155 }  ,draw opacity=1 ][line width=1.5]  [dash pattern={on 5.63pt off 4.5pt}]  (140.03,24.47) -- (303.2,24.63) ;
%Straight Lines [id:da9790446447245383] 
\draw [color={rgb, 255:red, 155; green, 155; blue, 155 }  ,draw opacity=1 ][line width=1.5]  [dash pattern={on 5.63pt off 4.5pt}]  (140.03,165.47) -- (303.2,165.63) ;
%Straight Lines [id:da07804865573759845] 
\draw [color={rgb, 255:red, 155; green, 155; blue, 155 }  ,draw opacity=0.25 ]   (223.59,35.37) -- (223.75,154.63) ;
%Straight Lines [id:da8261421157004802] 
\draw [color={rgb, 255:red, 155; green, 155; blue, 155 }  ,draw opacity=0.35 ]   (212.75,97.63) -- (295.11,97.63) ;
%Straight Lines [id:da8892574525991656] 
\draw    (162.13,64.69) -- (162.13,97.52) ;
\draw [shift={(162.13,99.52)}, rotate = 270] [color={rgb, 255:red, 0; green, 0; blue, 0 }  ][line width=0.75]    (4.37,-1.32) .. controls (2.78,-0.56) and (1.32,-0.12) .. (0,0) .. controls (1.32,0.12) and (2.78,0.56) .. (4.37,1.32)   ;
%Straight Lines [id:da39659066489197214] 
\draw    (277.12,37.75) -- (277.12,48.41) ;
\draw [shift={(277.12,50.41)}, rotate = 270] [color={rgb, 255:red, 0; green, 0; blue, 0 }  ][line width=0.75]    (4.37,-1.32) .. controls (2.78,-0.56) and (1.32,-0.12) .. (0,0) .. controls (1.32,0.12) and (2.78,0.56) .. (4.37,1.32)   ;
%Straight Lines [id:da1881586853946401] 
\draw    (223.98,37.75) -- (223.98,48.41) ;
\draw [shift={(223.98,50.41)}, rotate = 270] [color={rgb, 255:red, 0; green, 0; blue, 0 }  ][line width=0.75]    (4.37,-1.32) .. controls (2.78,-0.56) and (1.32,-0.12) .. (0,0) .. controls (1.32,0.12) and (2.78,0.56) .. (4.37,1.32)   ;
%Curve Lines [id:da4273499808300141] 
\draw [line width=1.5]    (144.2,129.1) .. controls (197.2,128.1) and (246.2,119.1) .. (246.79,104.93) ;
%Curve Lines [id:da2762350648941748] 
\draw [line width=1.5]    (302.79,58.6) .. controls (255.2,58.1) and (200.2,66.1) .. (200.95,80.65) ;
%Curve Lines [id:da07122360199400968] 
\draw [line width=1.5]  [dash pattern={on 5.63pt off 4.5pt}]  (200.95,80.65) .. controls (201.99,103.87) and (246.79,79.6) .. (246.79,104.93) ;
%Straight Lines [id:da9273583432943481] 
\draw    (277.6,108.63) -- (277.6,141.46) ;
\draw [shift={(277.6,106.63)}, rotate = 90] [color={rgb, 255:red, 0; green, 0; blue, 0 }  ][line width=0.75]    (4.37,-1.32) .. controls (2.78,-0.56) and (1.32,-0.12) .. (0,0) .. controls (1.32,0.12) and (2.78,0.56) .. (4.37,1.32)   ;
%Straight Lines [id:da8833348834474932] 
\draw    (223.98,72.47) -- (223.98,83.13) ;
\draw [shift={(223.98,70.47)}, rotate = 90] [color={rgb, 255:red, 0; green, 0; blue, 0 }  ][line width=0.75]    (4.37,-1.32) .. controls (2.78,-0.56) and (1.32,-0.12) .. (0,0) .. controls (1.32,0.12) and (2.78,0.56) .. (4.37,1.32)   ;
%Straight Lines [id:da5766182541688468] 
\draw    (223.98,102.13) -- (223.98,112.79) ;
\draw [shift={(223.98,114.79)}, rotate = 270] [color={rgb, 255:red, 0; green, 0; blue, 0 }  ][line width=0.75]    (4.37,-1.32) .. controls (2.78,-0.56) and (1.32,-0.12) .. (0,0) .. controls (1.32,0.12) and (2.78,0.56) .. (4.37,1.32)   ;
%Straight Lines [id:da5925047572636277] 
\draw    (223.98,134.74) -- (223.98,145.4) ;
\draw [shift={(223.98,132.74)}, rotate = 90] [color={rgb, 255:red, 0; green, 0; blue, 0 }  ][line width=0.75]    (4.37,-1.32) .. controls (2.78,-0.56) and (1.32,-0.12) .. (0,0) .. controls (1.32,0.12) and (2.78,0.56) .. (4.37,1.32)   ;
%Straight Lines [id:da5242503433733616] 
\draw    (162.64,140.58) -- (162.64,151.24) ;
\draw [shift={(162.64,138.58)}, rotate = 90] [color={rgb, 255:red, 0; green, 0; blue, 0 }  ][line width=0.75]    (4.37,-1.32) .. controls (2.78,-0.56) and (1.32,-0.12) .. (0,0) .. controls (1.32,0.12) and (2.78,0.56) .. (4.37,1.32)   ;
%Straight Lines [id:da6241399643665702] 
\draw [color={rgb, 255:red, 155; green, 155; blue, 155 }  ,draw opacity=0.25 ]   (142.2,92.3) -- (305.04,92.03) ;

% Text Node
\draw (311,20.73) node [anchor=north west][inner sep=0.75pt]  [font=\large]  {$\infty $};
% Text Node
\draw (307,161.73) node [anchor=north west][inner sep=0.75pt]  [font=\large]  {$-\infty $};
% Text Node
\draw (311,88.73) node [anchor=north west][inner sep=0.75pt]  [font=\large]  {$0$};
% Text Node
\draw (141.42,77.03) node [anchor=north west][inner sep=0.75pt]  [font=\large]  {$\mathrm{\Lambda} $};
% Text Node
\draw (208.29,145.73) node [anchor=north west][inner sep=0.75pt]  [font=\large]  {$X$};
% Text Node
\draw (255.93,101.03) node [anchor=north west][inner sep=0.75pt]  [font=\large]  {$\mathrm{\Lambda} '$};
% Text Node
\draw (279.62,65.73) node [anchor=north west][inner sep=0.75pt]  [font=\large]  {$U$};

\end{tikzpicture}

\caption{An S-shaped bifurcation. The section on $\mathrm\Lambda'$ defined by $\sigma(\lambda) = \big(\lambda,\phi^\lambda, \omega_{\phi^\lambda}(U)\big)$ fails to extend globally.}
\end{figure}
%%%%%%%%%%
%%%%%%%%%%%
\begin{proposition}\label{prop:sacyclic}
$\scrA^\obpsi$ is acyclic. 
\end{proposition}

\proof
Pick ${\lambda}_1 < a < b <{\lambda}_2 $, such that $\mathrm{\Lambda}_1 := (-\infty, b] $ and $\mathrm{\Lambda}_2 := [a,\infty)$ cover $\R$. Consider the Mayer-Vietoris exact sequence:
\[
\begin{tikzcd}[column sep=small]
 0  \arrow{r}{\delta^0}  & 
 \mathrm{\Gamma}(\scrA^\obpsi) \arrow{r}{\alpha^0_*}  &
 \mathrm{\Gamma}\big(\restr{\scrA^\obpsi}{\mathrm{\Lambda}_1}\big)\oplus\mathrm{\Gamma}\big(\restr{\scrA^\obpsi}{\mathrm{\Lambda}_2}\big)  \arrow{r}{\beta_*^0} \ar[draw=none]{d}[name=X, anchor=center]{}  &
 \mathrm{\Gamma}\big(\restr{\scrA^\obpsi}{[a,b]}\big) \ar[rounded corners,
            to path={ -- ([xshift=2ex]\tikztostart.east)
                      |- (X.center)  \tikztonodes
                      -| ([xshift=-2ex]\tikztotarget.west)
                      -- (\tikztotarget)}]{dll}[at end]{\delta^1} 
                      \\ \phantom{t}& H^1(\mathrm{\Lambda};\scrA^\obpsi) \arrow{r}{\alpha_*^1} & 
 H^1\big(\mathrm{\Lambda}_1; \restr{\scrA^\obpsi}{\mathrm{\Lambda}_1}\big)\oplus H^1\big(\mathrm{\Lambda}_2; \restr{\scrA^\obpsi}{\mathrm{\Lambda}_2}\big) \arrow{r}{\beta_*^1} & 
 H^1\big([a,b]; \restr{\scrA^\obpsi}{[a,b]}\big) \arrow{r} &
 0,
\end{tikzcd}
\]
since $H^2(\R;\scrA^\obpsi)\cong 0$, which uses the fact that intervals have covering dimension 1, cf.\ \cite[Lemma 2.7.3 and Proposition 3.2.2]{Schapira}. 
We can compute the global sections:
\[
\mathrm{\Gamma}\big(\scrA^\obpsi\big) \cong \Z_2^3, \quad \mathrm{\Gamma}\big(\restr{\scrA^\obpsi}{\mathrm{\Lambda}_1}\big)\cong \mathrm{\Gamma}\big(\restr{\scrA^\obpsi}{\mathrm{\Lambda}_2}\big) \cong \Z_2^4, \quad \quad  \mathrm{\Gamma}\big(\restr{\scrA^\obpsi}{[a,b]}\big)\cong \Z_2^5.
\]
Since $\im \delta^0 \cong 0$ and $\ker \delta^0 \cong 0$ we have $\ker\alpha_*^0 = \im \delta^0 \cong 0$. Consequently, $\im \alpha_*^0 \cong \Z_2^3$. Similarly, $\ker\beta_*^0 = \im \alpha_*^0\cong \Z_2^3$ which implies that $\im \beta_*^0 \cong \Z_2^5$. Furthermore, $\ker \delta^1 = \im \beta_*^0\cong \Z_2^5$ and thus $\im \delta^1 \cong 0$.
Since $\mathrm{\Lambda}_1, \mathrm{\Lambda}_2$ both contain only one bifurcation point, we can apply Lemma \ref{lemma:pitchacyclic} to conclude that $ H^1\big(\mathrm{\Lambda}_1; \restr{\scrA^\obpsi}{\mathrm{\Lambda}_1}\big)$ and $ H^1\big(\mathrm{\Lambda}_2; \restr{\scrA^\obpsi}{\mathrm{\Lambda}_2}\big)$ vanish for all $k\ge 1$. Hence,
$ \im \delta^1 =\ker \alpha_*^1 = H^1(\mathrm{\Lambda};\scrA^\obpsi)$, which proves that $H^1(\R, \scrA^\obpsi)$ is zero. The remaining sheaf cohomology vanishes due to the dimension restriction on $\mathrm{\Lambda}$.
\eproof

The S-shaped bifurcation is an example where $\scrA^\obpsi$ and  $\gAtt^\obpsi$, the attractor sheaf and free attractor sheaf respectively, have differing cohomologies.

\begin{proposition}
\label{freeattcoho}
Let $\gAtt^\obpsi$ be the free attractor sheaf associated to $\obpsi$.
$$
H^k(\mathrm{\Lambda};\gAtt^\obpsi) \cong \begin{cases} \Z_2^5 &\textnormal{if } k = 0 \\
\Z_2^4 & \textnormal{if } k=1 \\
0 & \textnormal{if } k \geq 2 \end{cases}
$$
\end{proposition}

\proof
Let $\mathrm{\Lambda}_1 = (-\infty,{\lambda}_2)$ and $\mathrm{\Lambda}_2 =({\lambda}_1,\infty)$ be an open covering for $\mathrm{\Lambda}=\R$. We build the \emph{ordered \Cech complex} from this cover:
\[
\begin{tikzcd}[column sep=small]
 0  \arrow{r}{}  & 
 \overline{C}^0(\{\mathrm{\Lambda}_1, \mathrm{\Lambda}_2\}; \gAtt^\obpsi)\arrow{r}{\delta^0}  &
 \overline{C}^1(\{\mathrm{\Lambda}_1, \mathrm{\Lambda}_2\}; \gAtt^\obpsi)\arrow{r}{\delta^1}  &
 \overline{C}^2(\{\mathrm{\Lambda}_1, \mathrm{\Lambda}_2\}; \gAtt^\obpsi)\arrow{r}{\delta^2}  &
 \dots
\end{tikzcd},
\]
which in our case is:
\[
\begin{tikzcd}[column sep=large]
 0  \arrow{r}{}  & 
 \gAtt^\obpsi(\mathrm{\Lambda}_1) \oplus \gAtt^\obpsi(\mathrm{\Lambda}_2)\arrow{r}{\rho^2_{1,2} - \rho^1_{1,2}}  &
 \gAtt^\obpsi(\mathrm{\Lambda}_1\cap \mathrm{\Lambda}_2)\arrow{r}  &
0,
\end{tikzcd}
\]
where $\rho^2_{1,2}$ denotes the restriction map from $\gAtt^\obpsi(\mathrm{\Lambda}_1)$ to $\gAtt^\obpsi(\mathrm{\Lambda}_1\cap \mathrm{\Lambda}_2)$, and $\rho^2_{1,2}$ from $\gAtt^\obpsi(\mathrm{\Lambda}_2)$. We get cohomology groups from the above chain complex:
$$\check{H}^0(\{\mathrm{\Lambda}_1, \mathrm{\Lambda}_2\}; \gAtt^\obpsi) = \ker \delta_1 \cong \Z_2^5, \quad\quad  \check{H}^1(\{\mathrm{\Lambda}_1, \mathrm{\Lambda}_2\}; \gAtt^\obpsi) = \ker \delta_2 / \Img \delta_1 \cong \Z_2^4,$$
$$\check{H}^k(\{\mathrm{\Lambda}_1, \mathrm{\Lambda}_2\}; \gAtt^\obpsi) = 0 \text{ for } k>1.$$
Since $\mathrm{\Lambda}_1\cap \mathrm{\Lambda}_2$ contains no bifurcation points, $\gAtt^\obpsi$ is locally constant on $\mathrm{\Lambda}_1\cap \mathrm{\Lambda}_2$ and  therefore acyclic on $\mathrm{\Lambda}_1\cap \mathrm{\Lambda}_2$. We now use Leray's Theorem to determine the sheaf cohomology of $\gAtt^\obpsi$ from the above \Cech cohomology groups, %cf.\ Appendix \ref{sec:SheafCohomology}, 
which yields the desired result.  
\eproof

The result of Proposition \ref{freeattcoho} is an example where Theorem \ref{thm:nonzerobifur} applies. The sheaf cohomology of $\gAtt^\obpsi$ picks up bifurcations. For the sheaf cohomology of
$\scrA^\obpsi$  Theorem \ref{thm:nonzerobifur} does not apply.

\begin{proposition}\label{prop:sshapedrelative}
Let $\mathrm{\Lambda}' = [a,\infty)$.
If $a\in ({\lambda}_1,{\lambda}_2]$, then $H^k(\mathrm{\Lambda}, \mathrm{\Lambda}'; \scrA^\obpsi)\cong \Z_2$ for $k=1$, and vanishes otherwise. When $a\notin({\lambda}_1,{\lambda}_2]$, then $H^k(\mathrm{\Lambda}, \mathrm{\Lambda}'; \scrA^\obpsi)=0$ for all $k$. 
\end{proposition}

\proof
We achieve a truncated long exact sequence from Proposition \ref{prop:sacyclic} and Proposition \ref{prop:stableconstant}:
\[
\begin{tikzcd}[column sep=small]
 0  \arrow{r}{}  & 
 H^0(\mathrm{\Lambda},\mathrm{\Lambda}';\scrA^\obpsi) \arrow{r}{j_*^0}  &
 H^0(\mathrm{\Lambda};\scrA^\obpsi) \arrow{r}{i_*^0}  &
 H^0(\mathrm{\Lambda}';\scrA^\obpsi) \arrow{r}{\delta^1}  &
 H^1(\mathrm{\Lambda},\mathrm{\Lambda}'; \scrA^\obpsi) \arrow{r}  &  
 0.
\end{tikzcd}
\]
The map $i_*^0$ is injective and thus $H^0(\mathrm{\Lambda},\mathrm{\Lambda}';\scrA^\obpsi)\cong 0$. 
Lemma \ref{lemma:pitchacyclic2} then yields:
$$
H^1(\mathrm{\Lambda},\mathrm{\Lambda}'; \scrA^\obpsi) \cong \frac{H^0(\mathrm{\Lambda}';\scrA^\obpsi)}{\im i_*^0}
$$
Note that $\im i_*^0\cong H^0(\mathrm{\Lambda};\scrA^\obpsi) = \mathrm{\Gamma}\bigl(\scrA^\obpsi\bigr)\cong \Z_2^3$. For $b\in ({\lambda}_1,{\lambda}_2]$, we have $H^0(\mathrm{\Lambda}';\scrA^\obpsi) = \mathrm{\Gamma}\big(\restr{\scrA^\obpsi}{\mathrm{\Lambda}'}\big)\cong \Z_2^4$, which implies $H^1(\mathrm{\Lambda},\mathrm{\Lambda}'; \scrA^\obpsi) \cong \Z_2$. Otherwise, $H^1(\mathrm{\Lambda},\mathrm{\Lambda}'; \scrA^\obpsi) =0$. 
\eproof

\begin{proposition}\label{prop:sshapedrelative2}
Let $\mathrm{\Lambda}' = (-\infty,a]$.
If $a\in [{\lambda}_1,{\lambda}_2)$, then $H^k(\mathrm{\Lambda},\mathrm{\Lambda}',\scrA^\obpsi)\cong \Z_2$ for $k=1$, and is zero otherwise. When $a\notin [{\lambda}_1,{\lambda}_2)$, then $H^k(\mathrm{\Lambda}, \mathrm{\Lambda}'; \scrA^\obpsi)=0$ for all $k$.
\end{proposition}

\proof
An identical argument as in the proof of Proposition \ref{prop:sshapedrelative}.
\eproof

\begin{theorem}\label{thm:abelianhysteresis}
Let $\obphi$ be a parametrized dynamical system conjugate to the above parametrization $\obpsi$ of the S-shaped bifurcation. Then, 
\[
\scrA^\obphi ~\text{~ is acyclic and~}~ H^0(\mathrm{\Lambda};\scrA^\obphi) \cong\Z_2^3.
\]
Moreover,
there exists $\lambda_1, \lambda_2\in \R$ such that
$$
H^k(\mathrm{\Lambda}, \mathrm{\Lambda}';\scrA^\obphi) \cong \begin{cases} \Z_2 &\textnormal{if } k = 1\textnormal{ and } a\in ({\lambda}_1,{\lambda}_2] \\
0 & \textnormal{otherwise.}  \end{cases} \quad \text{with}\quad\mathrm{\Lambda}' = [a,\infty),
$$
$$
H^k(\mathrm{\Lambda}, \mathrm{\Lambda}';\scrA^\obphi) \cong \begin{cases} \Z_2 &\textnormal{if } k = 1\textnormal{ and } a\in [{\lambda}_1,{\lambda}_2)\\
0 &\textnormal{otherwise.} \end{cases}\quad\text{with}\quad \mathrm{\Lambda}' = (-\infty,a].
$$
\end{theorem}

\proof
Apply Theorem \ref{thm:CIT}, Proposition \ref{prop:sacyclic}, and Propositions \ref{prop:sshapedrelative} and \ref{prop:sshapedrelative2}. 
\eproof

\begin{remark}
Note that the relative cohomologies $H^k(\mathrm{\Lambda},\mathrm{\Lambda}';\scrA^\obphi)$ are the same for the trans-critical and S-shaped bifurcation.
\end{remark}

\begin{remark}\label{rem:comparehysteresis}
If we consider the S-shaped bifurcation on an interval $X=I=[-c,c]$, $c\gg 1$, with time space $\T=\R^+$ and parameter space $\mathrm{\Lambda} = [-\lambda_0,\lambda_0]$, with $\lambda_0 = -c+c^3-\epsilon$, $0<\epsilon\ll 1$ we obtain the following sheaf cohomology:
\[
\scrA^\obphi ~\text{~ is acyclic and~}~ H^0(\mathrm{\Lambda};\scrA^\obphi) \cong\Z_2.
\]
Moreover,
there exists a value $\lambda_0\in \R$ such that
$$
H^k(\mathrm{\Lambda}, \mathrm{\Lambda}';\scrA^\obphi) \cong \begin{cases} \Z_2 &\textnormal{if } k = 1\textnormal{ and } a\in ({\lambda}_1,{\lambda}_2] \\
0 & \textnormal{otherwise.}  \end{cases} \quad \text{with}\quad\mathrm{\Lambda}' = [a,\infty)\cap \mathrm{\Lambda},
$$
$$
H^k(\mathrm{\Lambda}, \mathrm{\Lambda}';\scrA^\obphi) \cong \begin{cases} \Z_2 &\textnormal{if } k = 1\textnormal{ and } a\in [{\lambda}_1,{\lambda}_2)\\
0 &\textnormal{otherwise.} \end{cases}\quad\text{with}\quad \mathrm{\Lambda}' = (-\infty,a]\cap \mathrm{\Lambda}.
$$
For free attractor sheaf we have:
$$
H^k(\mathrm{\Lambda};\gAtt^\obpsi) \cong \begin{cases} \Z_2^2 &\textnormal{if } k = 0 \\
\Z_2 & \textnormal{if } k=1 \\
0 & \textnormal{if } k \geq 2 \end{cases},
$$
which is clearly not acyclic. 
\end{remark}

\subsection{Comparing the attractor and free attractor sheaves}\label{ssc:comparecohomology}
In the above treatment of the pitchfork, the saddle-node and transcritical bifurcations we have only used the attractor sheaf. %If we consider the same examples using the free attractor sheaf  the cohomology groups will be different only for the cohomology groups that are non-zero. 
We can reexamine the saddle-node bifurcation as studied in 
Propositions \ref{prop:saddlerelative} and \ref{prop:saddlerelative2} with the free attractor sheaf.
Let 
$$
\obpsi\colon \mathrm{\Lambda} \to \sDS(\T,X), 
$$
be the parametrized system for the saddle-node bifurcation as given in \ref{sn1}, 
where $\mathrm{\Lambda} = \R$ is parameter space, $\T = \R$ is the time space and $X$ is the 2-point compactification of $\R$.

For the free attractor sheaf we have:
\[
\gAtt^\obpsi ~\text{~ is acyclic and ~}~ H^0(\mathrm{\Lambda};\gAtt^\obpsi) \cong\Z_2^4.
\]
Moreover,
$$
H^k(\mathrm{\Lambda}, \mathrm{\Lambda}';\gAtt^\obpsi) \cong \begin{cases} \Z_2^4 &\textnormal{if } k = 1\textnormal{ and } a>0 \\
0 & k\neq 1, \textnormal{or~} a\le 0, \end{cases} \quad \text{with}\quad\mathrm{\Lambda}' = [a,\infty),
$$
$$
H^k(\mathrm{\Lambda}, \mathrm{\Lambda}';\gAtt^\obpsi) \cong \begin{cases} \Z_2 &\textnormal{if } k = 0\textnormal{ and } a< 0 \\
0 & k\neq 0, \textnormal{or~} a\ge 0, \end{cases}\quad\text{with}\quad \mathrm{\Lambda}' = (-\infty,a].
$$

The abelian attractor sheaf, as shown in Theorem \ref{thm:abelianhysteresis}, is acyclic for the S-shaped bifurcation. Proposition \ref{freeattcoho} demonstrates nontrivial trivial cohomology in dimension one for the free attractor sheaf. Consider the continuation of the union of both attracting fixed points, a section of both the free and abelian attractor sheaves on the interval between both two bifurcation points. Recall the ordered \Cech complex on the open cover $\mathrm{\Lambda}_1 = (-\infty,{\lambda}_2)$ and $\mathrm{\Lambda}_2 =({\lambda}_1,\infty)$: 
\[
\begin{tikzcd}[column sep=large]
 0  \arrow{r}{}  & 
 \gAtt^\obpsi(\mathrm{\Lambda}_1) \oplus \gAtt^\obpsi(\mathrm{\Lambda}_2)\arrow{r}{\rho^2_{1,2} - \rho^1_{1,2}}  &
 \gAtt^\obpsi(\mathrm{\Lambda}_1\cap \mathrm{\Lambda}_2)\arrow{r}  &
0,
\end{tikzcd}
\]
For both sheaves, this is not the restriction of any section from $\mathrm{\Lambda}_1$ or $\mathrm{\Lambda}_2$. However, the symmetric Conley form lets us write this section as the sum of the sections corresponding to each attracting fixed point. Thus, for the abelian attractor sheaf, this gives a trivial cohomology class. For the free attractor sheaf the addition operation is formal, so we cannot write the union of the two attracting fixed points as a sum of sections from $\mathrm{\Lambda}_1$ and $\mathrm{\Lambda}_2$.

%\begin{acknowledgements}
%\end{acknowledgements}

\bibliographystyle{spmpsci}      % mathematics and physical sciences
\bibliography{KMVc-biblist}   % name your BibTeX data base
\vfill\eject

%%%%%%%%%%%%%%%%%%%%%%%%%%%%%%%%%%%%%%%%%%%%
%%%%%%%%%%%%%%%%%%%%%%%%%%%%%%%%%%%%%%%%%%%%
%%%%%%%%%%%%%%%%%%%%%%%%%%%%%%%%%%%%%%%%%%%%

% % Non-BibTeX users please use
% \begin{thebibliography}{}
% %
% % and use \bibitem to create references. Consult the Instructions
% % for authors for reference list style.
% % Format for Journal Reference
% Author, Article title, Journal, Volume, page numbers (year)
% Author, Book title, page numbers. Publisher, place (year)
% \end{thebibliography}
% \end{sloppypar}
% \end{document}
% end of file template.tex

%%%%%%%%%%%%%%%%%%%%%%%%%%%%%%%%%%%%%%%%%%%%
%%%%%%%%%%%%%%%%%%%%%%%%%%%%%%%%%%%%%%%%%%%%
%%%%%%%%%%%%%%%%%%%%%%%%%%%%%%%%%%%%%%%%%%%%

\appendix

\section{Table with important definitions}

\centering
\begin{tabular}{ llll  }
 \hline
  & Notation & Description& Reference\\
 \hline

\multirow{9}{*}{\shortstack{Dynamics}} 
& $\T$ & Time space, either $\Z, \Z_+, \R, $ or $\R_+$&   Sect \ref{sec:DSCat} \\
& $\Inv_\phi(U)$  & Maximal invariant set in $U$ & Sect \ref{sec:AttFun}, pg \pageref{pg:inv}\\
& $\omega_\phi(U)$ & Omega limit set of $U$&  Sect \ref{sec:AttFun}, pg \pageref{pg:omegalimit} \\
& $\alpha_\phi(U)$ & Alpha limit set of $U$&  Rmk \ref{remark:alpha}\\
& $\sANbhd(\phi)$ & Lattice of attracting neighborhoods for $\phi$&  Sect \ref{sec:AttFun}, pg \pageref{pg:anbhd} \\
& $\sAtt(\phi)$ & Lattice of attractors for $\phi$&  Sect \ref{sec:AttFun}, pg \pageref{pg:att} \\
& $\sRep(\phi)$ & Lattice of repellers for $\phi$&  Rmk \ref{remark:alpha} \\
& $\sMorse(\phi)$ & Meet semilattice of Morse sets for $\phi$&  Sect \ref{sec:AttFun}, pg \pageref{pg:att} \\
& $\sMRepr(\phi)$ & Lattice of Morse representations for $\phi$&  Sect \ref{contatt12}, pg \pageref{pg:mrepr} \\
& $\sC_\sAtt(A, A')$ & Conley form of two attractors&  Sect \ref{sec:algconstr}, pg \pageref{pg:conleyform}\\
% & $\sob(\cC)$ & Objects of a Category $\cC$& something  \\
% & $\hom(\phi, \psi)$  & Morphisms&016 \\ 
% & $\hom(\phi, \psi)$  & Morphisms&016 \\ 

\hline
 
\multirow{9}{*}{\shortstack{Category\\Theory}} & $\sDS(\T,X)$ & Category of dynamical systems&   Sect \ref{sec:DSCat}, pg \pageref{pg:dscat} \\
& $\sob(\cC),\hom(\cC) $ & Objects and morphisms of a category $\cC$& \cite{MacLane}  \\
& $\hom(\phi, \psi)$  & Morphisms between two objects & \cite{MacLane}\\
& $\sF\colon \cD \to \cC$  & Universe functor & Sect \ref{Cstruc}, pg \pageref{pg:universefunctor} \\
& $\sF_0 \in \cC$ & Value of universe functor&  Sect \ref{Cstruc}, pg \pageref{pg:universefunctor} \\
%& $\sob(\cC)$ & Objects of a category $\cC$& \cite{MacLane}  \\
%& $\hom(\phi, \psi)$  & Morphisms from $\phi$ to $\psi$ & \cite{MacLane}\\
& $\mathrm{\Pi}[\sE]$ & Category of elements for a functor $\sE$ & \cite{MacLane,MacLaneMoerdijk} \\
& $\mathrm{\Phi}[\sE; U]$ & Objects $\phi$ for which $U\in \sE(\phi)$ & Sect \ref{Cstruc}, pg \pageref{pg:phiopen} \\
& $\mathrm{\Theta}[\sw; U]$ & Partial section functor on $\mathrm{\Phi}[\sE; U]$ & Sect \ref{Cstruc}, pg \pageref{pg:partialsec} \\
& $(\sG, \sE, \sw)$ & Continuation frame & Def \ref{Cframe} \\
\hline

\multirow{7}{*}{\shortstack{Order\\Theory}} 
& $\subF\colon \sLat \to \sLat$ & Lattice of finite sublattices functor & Sect \ref{Morserepcont}, pg \pageref{pg:subf} \\
& $\pred U$  & Unique immediate predecessor of $U$ & \cite{KMV-1c} \\
& $\sO\colon \sPoset \to \sBDLat$ & Down-set functor&  \cite{KMV-1c} \\
& $\sJ\colon \sLat \to \sPoset$ & Poset of join-irreducibles of $\sU$ & \cite{KMV-1c}  \\
%& $\mathrm{\Sigma}\sL$  & Spectrum of $\sL$, poset of prime ideals & \cite{KMV-1c} \\
& $\sB\colon \sBDLat\to \sBool$ & Booleanization functor & \cite{KMV-1c} \\
& $\sR\colon \sBDLat\to \sRing$ & (Boolean) lattice ring of $\sL$ & Sect \ref{ssc:alg-att}, pg \pageref{pg:latticering} \\
& $\Z_2\colon \sBDLat \to \sRing$ & Lattice algebra of $\sL$ & Sect \ref{ssc:alg-att}, pg \pageref{pg:latticealg} \\
\hline

\multirow{8}{*}{\shortstack{Sheaf\\Theory}} 
& $\scrS^\sG\colon \cO(\cD) \to \sSetCat$ & Sheaf of sections for $\mathrm{\Pi}[\sG]$ & Def \ref{defn:sheafofsec} \\
& $\mathrm{\Gamma}(\scrS^\sG)$  & Set of global sections for $\scrS^\sG$ & Def \ref{defn:sheafofsec} \\
& $\scrF_\phi$  & Stalk of a sheaf $\scrF$ at $\phi$ & Rmk \ref{stalks}, \cite{Bredon} \\
& $\scrS^\sAtt$ & Attractor lattice sheaf & Sect \ref{ssc:attsheaves}, pg \pageref{pg:attlatsheaf} \\
& $\scrA^\obphi$ & Attractor sheaf for $\obphi$ & Sect \ref{sec:bifs}, pg \pageref{pg:attsheaf}   \\
& $\gAtt^\obphi$ & Free attractor sheaf for $\obphi$ & Sect \ref{sec:bifs}, pg \pageref{pg:attsheaf}   \\
& $H^*(\mathrm\Lambda; \scrF)$  & Sheaf cohomology of a sheaf $\scrF$ on $\mathrm\Lambda$ & \cite{Bredon} \\
& $H^*(\mathrm\Lambda, \mathrm\Lambda'; \scrF)$  & Relative sheaf cohomology & \cite{Bredon} \\
\hline

\end{tabular}

\section{Functorial properties of attractors}\label{sec:app1}

\proof[Lemma~\ref{lem:finv}]
For $t\ge 0$, we have $\phi_t(h^{-1}(U))\subset (h^{-1}\circ h\circ \phi_t)(h^{-1}(U))$. Since $h$ is a quasiconjugacy, we have $(h^{-1}\circ h\circ \phi_t)(h^{-1}(U)) = h^{-1}(\psi^\dagger_{t}((h\circ h^{-1})(U)))\subset h^{-1}(\psi^\dagger_{t}(U))$ and thus
\[
\phi_t(h^{-1}(U))\subset h^{-1}(\psi^\dagger_{t}(U)),\quad \forall t\ge 0.
\]
The inequality for $\omega$ now follows from elementary properties of inverse images and closures:
\begin{equation*}
\begin{aligned}
    \omega_\phi(h^{-1}(U)) & = \bigcap_{t\ge 0} \cl \bigcup_{s\ge t} \phi_s\bigl(h^{-1}(U)\bigr) 
    \subset \bigcap_{t\ge 0} \cl\bigcup_{s\ge t}h^{-1}\bigl(\psi^\dagger_{s}(U)\bigr) = \bigcap_{t\ge 0} \cl h^{-1}\Big(\bigcup_{s\ge t}\psi^\dagger_{s}(U)\Big)
    \\
    & \subset \bigcap_{t\ge 0} h^{-1}\Big(\cl\bigcup_{s\ge t}\psi^\dagger_{s}(U)\Big) = h^{-1}\bigg( \bigcap_{t\ge 0} \cl\bigcup_{s\ge t} \psi^\dagger_{s}(U) \bigg) \\
    &= h^{-1}\bigg( \bigcap_{t\ge 0} \cl \bigcup_{x\in U} \bigcup_{\sigma \ge \tau(t,x)} \psi_{\sigma}(x) \bigg)= 
    h^{-1}\bigg( \bigcap_{\tau\ge 0} \cl \bigcup_{\sigma \ge \tau} \psi_{\sigma}(U) \bigg) =
    h^{-1}(\omega_\psi(U)),
\end{aligned}
\end{equation*}
which uses the invertibility of the parametrization function $\tau$.
Finally applying $\omega_\phi$ we obtain
$$
\omega_\phi(h^{-1}(U))=\omega_\phi(\omega_\phi(h^{-1}(U)))\subset\omega_\phi(h^{-1}(\omega_\psi(U)))\subset\omega_\phi(h^{-1}(U))
$$
so that 
\begin{equation}
\label{pullbackatt}
\omega_\phi(h^{-1}(U))=\omega_\phi(h^{-1}(\omega_\psi(U))),
\end{equation}
which completes the proof.
\eproof

\proof[Remark~\ref{remark:alpha}]
To deal with negative times we define $\tau(-t,x):= \tau(t,x)$ in which case 
\[
\psi^\dagger_{-t} = \psi\bigl(\tau(-t,\cdot),\cdot\bigr) = \psi\bigl(-\tau(t,\cdot),\cdot\bigr)
= \bigl( \psi^\dagger_t\bigr)^{-1}.
\]
Let $x\in \phi_{-t}(h^{-1}(U))$ so that  $\phi_t(x) \in h^{-1}(U)$. Then, by the quasiconjugacy condition $h(\phi_t(x)) = \psi^\dagger_t(h(x)) \in U$, and therefore $h(x)\in \psi^\dagger_{-t}(U)$. This yields $x\in h^{-1}(\psi^\dagger_{-t}(U))$. 
Summarizing we have 
\[
\phi_{-t}(h^{-1}(U))\subset h^{-1}(\psi^\dagger_{-t}(U)),\quad \forall t\ge 0.
\]
The remainder of the proof is similar to the proof of Lemma \ref{lem:finv}.
\eproof

\proof[Proposition~\ref{prop:finvAtt}]
Since $A$ is an attractor for $\psi$, there exists an attracting neighborhood $U$ such that $\omega_\psi(U)=A$. 
By Eqn.\ \eqref{pullbackatt} we have
\[
\omega_\phi\bigl(h^{-1}(U)\bigr) = \omega_\phi\bigl(h^{-1}(\omega_\psi(U))\bigr) = \omega_\phi\bigl(h^{-1}(A)\bigr),
\]
which proves that $\omega_\phi\bigl(h^{-1}(A)\bigr)$ is an attractor for $\phi$, since we already know $h^{-1}(U)$ is an attracting neighborhood for $\phi$.

Therefore, for a quasiconjugacy $\tau\times h\in \Hom(\phi,\psi)$, the map $\omega_\phi \circ h^{-1}\colon \sAtt(\psi) \to \sAtt(\phi)$ is  well defined. 
It remains to show that the latter is a lattice homomorphism. Preservation of joins is clear, cf.\ Property (v) for omega-limit sets. Let $A, A'\in \sAtt(\psi)$, then
\begin{equation*}
\begin{aligned}
\omega_\phi(h^{-1}(A\wedge A'))) & = \omega_\phi(h^{-1}(\omega_\psi(A\cap A'))) \subset \omega_\phi(h^{-1}(A \cap A')) = \omega_\phi\bigl(h^{-1}(A) \cap h^{-1}(A')\bigr)
\\ 
& = \omega_\phi\big(\omega_\phi\bigl(h^{-1}(A) \cap h^{-1}(A')\bigr)\big) \subset \omega_\phi\big(\omega_\phi(h^{-1}(A)) \cap \omega_\phi(h^{-1}(A'))\big)
\\
& = \omega_\phi(h^{-1}(A)) \wedge \omega_\phi(h^{-1}(A'))
\end{aligned}
\end{equation*}
Idempotency of $\omega_\phi$  and Equation~\eqref{pullbackatt2} imply 
\begin{equation*}
\begin{aligned}
\omega_\phi(h^{-1}(A)) \wedge \omega_\phi(h^{-1}(A')) & = \omega_\phi\big(\omega_\phi(h^{-1}(A)) \cap \omega_\phi(h^{-1}(A'))\big) \\
&\subset \omega_\phi\big(h^{-1}(\omega_\psi(A)) \cap h^{-1}(\omega_\psi(A'))\big)
 = \omega_\phi\big(h^{-1}(\omega_\psi(A) \cap \omega_\psi(A'))\big)\\
 &= \omega_\phi(h^{-1}(A) \cap h^{-1}(A')) 
  = \omega_\phi(h^{-1}(A\cap A'))\\
  &= \omega_\phi\big( \omega_\phi(h^{-1}(A\cap A')) \big)  
\subset \omega_\phi(h^{-1}(\omega_\psi(A\cap A')))\\
&= \omega_\phi(h^{-1}(A\wedge A'))),
\end{aligned}
\end{equation*}
which proves that
\[
\omega_\phi\big(h^{-1}(A\wedge A'))\big) = \omega_\phi(h^{-1}(A)) \wedge \omega_\phi(h^{-1}(A')),
\]
and thus 
$\omega_\phi \circ h^{-1}\colon \sAtt(\psi) \to \sAtt(\phi)$ is a lattice homomorphism. 
\eproof

\proof[Remark~\ref{conjatt}]
If $\tau\times h\in \Hom(\phi,\psi)$ is a conjugacy, then
\begin{equation}
    \label{conj1}
h\bigl( \phi_t(x)\bigr) = \psi_t^\dagger\bigl(h(x)\bigr).
\end{equation}
Define $y=h(x)$ and $s=\tau\bigl(t,h^{-1}(y)\bigr)$.
Since $h$ is a homeomorphism, we obtain $\tau^{-1}(s,y)$, and therefore
\begin{equation}
    \label{conj2}
\phi_s^\dagger\bigl(h^{-1}(y) \bigr) = h^{-1}\bigl( \psi_s(y)\bigr),
\end{equation}
where $\phi_s^\dagger = \psi\bigl(\tau^{-1}(s,\cdot),\cdot  \bigr)$.
This proves that $\tau^{-1}\times h^{-1} \in \Hom(\psi,\phi)$ is a conjugacy.

Let $A\in \sAtt(\psi)$, then by Proposition~\ref{prop:finvAtt}, we have  $\omega_\phi\bigl(h^{-1}(A) \bigr)\in \sAtt(\phi)$. By Equation~\eqref{conj2} we have
$\phi_s^\dagger\bigl(h^{-1}(A) \bigr) = h^{-1}\bigl( \psi_s(A)\bigr) = h^{-1}(A)$ for all $s\ge 0$, which proves invariance of $h^{-1}(A)$. Furthermore, since $h$ is a homeomorphism, it follows that $h^{-1}(A)$ is closed, and thus 
$\omega_\phi\bigl(h^{-1}(A)\bigr) = h^{-1}(A)$, which proves that $h^{-1}(A) \in \sAtt(\phi)$.
Similarly, $h(A) \in \sAtt(\psi)$ for all $A\in \sAtt(\phi)$.
\eproof

\section{Repellers}
\label{repels12}

% \corrc I put pointers to this section in Remarks 3.4 and 3.7 -- should we just move the blue portion below to Section 3 OR perhaps to the appendix and reference it in Section 3? BK 12-29-31 
% Put the content of Rmks 3.4 and 3.7 here, RC 1-12-22
% <<>> 

In Remarks \ref{remark:alpha} and \ref{remark:alpha2} we indicated that one can also construct continuation frames $(\sRep,\sRNbhd,\alpha)$ based repelling neighborhoods and repellers which yields the \etale space $\mathrm{\Pi}[\sRep]$.
For a dynamical system $\phi\colon \T^+\times X \to X$ we define $\phi_{-t} := \phi_t^{-1}$ as the inverse image. The map $\phi(-t,x)$ also satisfies the semigroup property.
This  allows us to define the notion {\em alpha-limit set} as
\[
\alpha_\phi(U) := \bigcap_{t\ge 0} \cl \bigcup_{s\ge t} \phi_{-s}(U).
\]
Some properties of $\alpha_\phi(U)$ are: (i) $\alpha_\phi(U)$ is compact, closed, (ii) $\alpha_\phi(U)$ is a forward-backward invariant set for the dynamics, (iii) $\alpha_\phi\bigl(\alpha_\phi(U)\bigr) \supset \alpha_\phi(U)$, 
(iv) $\alpha_\phi(U\cup V) = \alpha_\phi(U) \cup \alpha_\phi(V)$.
A neighborhood $U\subset X$ is called a {\em repelling neighborhood} if $\alpha_\phi(U)
\subset \Int U$. Repelling neighborhoods form a bounded, distributive lattice denoted by $\sRNbhd(\phi)$. The binary operations are $\cap$ and $\cup$.
A subset $A\subset X$ is called a {\em repeller} if there exists an repelling neighborhood $U\subset X$ such that $R=\alpha_\phi(U)$, which is a neighborhood of $R$ by definition. Repellers are compact, closed, forward-backward invariant sets and compose a bounded, distributive lattice $\sRep(\phi)$ with binary operations $\cup$ and $\cap$. As before
 $\phi \mapsto \sRNbhd(\phi)$  and $\phi \mapsto \sRep(\phi)$ define the contravariant functors $\sRNbhd$  and $\sRep$ from
$\sDS(\T,X) \to \sBDLat$. 
The functor $\sRNbhd\colon\sDS(\T,X)\to\sBDLat$ is a stable structure
and  $(\sRep,\sRNbhd,\alpha)$ forms a continuation frame in a similar way. 
From the continuation frame $(\sRep,\sRNbhd,\alpha)$ we obtain the \etale space
$(\mathrm{\Pi}[\sRep],\pi)$.

For a dynamical system $\phi$ consider the duality isomorphism $A\mapsto A^*$, $A\in \sAtt(\phi)$.
Since for $U\in \sANbhd(\phi)$ the maps
$U\mapsto U^c$ and $\omega_\phi(U) \mapsto \alpha_\phi(U^c)$ define lattice isomorphisms we also have the natural transformations $^c\colon \sANbhd \Longleftrightarrow \sRNbhd$ and $^*\colon \sAtt \Longleftrightarrow \sRep$. This yields the following commutative diagram:
\[
\begin{tikzcd}[row sep=1.5em, column sep = 1.5em]
    \sANbhd(\psi) \arrow[rrrr,leftrightarrow,"c"] \arrow[drrr,"h^{-1}"] \arrow[ddd,"\omega_\psi"] &&&&
    \sRNbhd(\psi) \arrow[ddd,"\alpha_\psi"] \arrow[drrr,"h^{-1}"] \\
    &&& \sANbhd(\phi) \arrow[rrrr,leftrightarrow,"c"] \arrow[ddd,"\omega_\phi"] &&&&
    \sRNbhd(\phi) \arrow[ddd,"\alpha_\phi"] \\ \\
        \sAtt(\psi) \arrow[rrrr,leftrightarrow,"*_\psi"] \arrow[drrr, "\sAtt(h)"] &&&& \sRep(\psi)\arrow[drrr,"\sRep(h)"] \\
    &&& \sAtt(\phi) \arrow[rrrr,leftrightarrow,"*_\phi"] &&&& \sRep(\phi)
    \end{tikzcd}
\]
where $\sAtt(h) = \omega_\phi \circ h^{-1}$ and $\sRep(f) := *_\phi\circ \omega_\phi \circ h^{-1} \circ *_\psi$. 
This asymmetry between attractors and repellers is typical for noninvertible systems.
For invertible systems the symmetry is restored so that $\sRep(h)= \alpha_\phi \circ h^{-1}$.

\section{Function spaces and the compact-open topology}
\label{cot}
We recall some basic facts about topologies on function spaces of continuous functions.
Let $X$ and $Y$ be arbitrary topological spaces and let $C(X,Y)$ the denote the set of all continuous maps $f\colon X\to Y$. A topology on $C(X,Y)$ which is of particular importance is the \emph{compact-open topology} which is defined a
 subbasis of sets of the form 
\[
O(K,U):= \bigl\{f~|~f(K) \subset U\hbox{ for $K$ compact in $X$ and $U$ open in $Y$}\bigr\},
\]
where $K$ ranges over all compact subsets in $X$ and $U$ ranges over all open subsets in $Y$, cf.\ \cite{Fox}. 
 If $X$ is a locally compact, Hausdorff space then the compact-open topology is the weakest topology such that the map $(f,x)\mapsto f(x)$, $f\in C(X,Y)$, is continuous, cf.\ \cite[Cor.\ 1.2.4]{Marcelo}.
If $X$ is compact and $Y$ is a metric space with metric $d$, then the compact-open topology corresponds with the metric topology on $C(X,Y)$ given by the metric:
\[
d(f,g)= \sup\limits_{x\in X} d(f(x),g(x)),\qquad f,g\in C(X,Y),
\]
cf.\ \cite{Muger,Schwendiman}.

Let $\mathrm{\Lambda}$ be an arbitrary topological space. For a  continuous map $h\colon \mathrm{\Lambda} \times X \to Y$ we define the  \emph{transpose} of $h$ by:
\[
h_\bas\colon \mathrm{\Lambda} \to C(X,Y),\quad\quad \lambda \mapsto h_\bas(\lambda)=h^\lambda := h(\lambda,\cdot).
\]
Following the terminology in \cite{Escardo} we say that 
a topology on $C(X,Y)$ is \emph{weak} if  continuity of $h$ implies  continuity of the transpose $h_\bas$, and a topology is \emph{strong} if continuity of the transpose $h_\bas$ implies continuity of $h$. For arbitrary topological spaces $X,Y$ and $\mathrm{\Lambda}$ the compact-open topology is  a weak topology on $C(X,Y)$, i.e. $h$ continuous implies that $h_\bas$ is continuous, cf.\ \cite[Lemma 1]{Fox},\cite{Muger}.
If $X$ is regular and locally compact (in particular for locally compact, Hausdorff spaces), 
then the compact-open topology is is both weak and strong, i.e. $h$ is continuous if and only if $h_\bas$ is continuous, cf.\ \cite[Theorem 1]{Fox}, \cite{Muger}. This implies that for regular and locally compact spaces $X$
 the compact-open topology on $C(X,Y)$ is both weak and strong, which is also referred to as an \emph{exponential topology}, cf.\ \cite{Escardo}. 
 The latter is unique.
 Finally, the map $h\mapsto h_\bas$ is an embedding when both $\mathrm{\Lambda}$ and $X$ are Hausdorff spaces. The map is a homeomorphism when $\mathrm{\Lambda}$ is Hausdorff and $X$ is locally compact, Hausdorff, cf.\ \cite{Muger}.
 
For a compact metric space $(X,d)$ define $\big(\sH(X),d_\sH\big)$ to be the metric space of compact subsets of $X$ equipped with the Hausdorff metric $d_\sH$.
Every continuous function $f\colon X \to Y$ induces a continuous function $f^\sH\colon \sH(X) \to \sH(Y)$, which sends compact subsets to their image under the function $f$. 
Recall the Hausdorff metric:
\[
d_\sH(K,K') := \max\bigg\{\sup_{x\in K}\inf_{x'\in K'} d(x,x') , \sup_{x'\in K'}\inf_{x\in K} d(x,x')\bigg\}, \quad K,K'\in \sH(X).
\]

\begin{lemma}\label{lemma:haus}
Let $X$, $Y$ be compact metric spaces, $\mathrm{\Lambda}$ a topological space, and $h\colon\mathrm{\Lambda} \times X \to Y$ continuous map.  Then, the function  
$$h^\sH\colon\mathrm{\Lambda} \times \sH(X)  \to \sH(Y) \quad \quad (\lambda,K) \mapsto h^\sH\big(\{\lambda\} \times K\big)$$
is continuous. 
\end{lemma}
\proof
We will first prove the assignment
$$
\sD \colon C(X,Y) \to C(\sH(X), \sH(Y)) \quad \quad f \mapsto \sD(f) := f^\sH,
$$
is continuous. Let $f,g\in C(X,Y)$. Then, $d_C(f,g)= \sup\limits_{x\in X} d(f(x),g(x))$ and 
$$
d_{C_\sH}(f^\sH, g^\sH) = \sup\limits_{K\in \sH(X)} d_\sH\big(f(K),g(K)\big).
$$ 
Since $d(y,y')\leq \sup\limits_{x\in X} d(f(x),g(x)) = d(f,g)$ for any choice of $y\in f(K), y' \in g(K)$ it follows that  $d_{C_\sH}(f^\sH, g^\sH) \leq d_C(f,g)$. Moreover, since points are compact subsets, the reversed inequality  holds as well:
$$ 
d_C(f,g) = \sup\limits_{x\in X} d_\sH\big( f(\{x\}),g(\{x\})\big)
\leq d_{C_\sH}(f^\sH, g^\sH),
$$
which proves that $\sD$ is an isometry implying its continuity.  
The metric topology on $C(\sH(X),\sH(Y))$ coincides with the compact-open topology and therefore
$h_\sH$ is continuous if and only if its \emph{transpose} $h^\sh_\bas$, given by  
$$
h^\sH_\bas\colon \mathrm{\Lambda} \to C\big(\sH(X), \sH(Y)\big), \quad \quad \lambda \mapsto h^\sH_\bas(\lambda) =h_\sH^\lambda :=
h_\sH\big(\{\lambda\},\cdot \big),
$$
is continuous, \cite{Fox,Escardo}. 
Note that $h^\sH_\bas = \sD \circ h_\bas$ which proves that $h^\sH_*$ is continuous, which completes the proof.
\eproof

\begin{remark}
In this paper we abuse notation by writing $h(K)$, $K\in \sH(X)$ denoting $h^\sH(K)$ in accordance with the analogous notation for $h(U) = \{y=h(x)~|~x\in U\}$, $U\subset X$.
\end{remark}

\end{sloppypar}
\end{document}